\documentclass[pagesize,paper=A4,11pt,bibliography=totoc,oneside]{scrartcl}
\usepackage[english]{babel}
\usepackage{graphicx}
\usepackage[dvipsnames]{xcolor}
\usepackage{amssymb,amsmath,amsthm}
\usepackage{bbm}
\usepackage{mathrsfs} 
\usepackage[T1]{fontenc}
\usepackage{lmodern}
\usepackage{microtype}
\usepackage[all,cmtip]{xy}
\usepackage{booktabs}
\usepackage{numprint}

\usepackage[numbers,sort&compress]{natbib}
\usepackage[ruled,vlined]{algorithm2e}
\DontPrintSemicolon

\SetCommentSty{commentfont}
\SetAlgorithmName{Algorithm}{Algorithm}{List of algorithms}
\SetKw{KwOr}{or}
\SetKwFunction{gLPM}{gPathMat}
\SetKwFunction{gRec}{gRecursive}

\usepackage{listings}

\usepackage[colorlinks=true,citecolor=CornflowerBlue,urlcolor=OliveGreen,breaklinks=true,linkcolor=BrickRed]{hyperref}
\usepackage{cleveref}
\usepackage{orcidlink}
\crefname{procf}{alg.}{algs.}
\Crefname{procf}{Algorithm}{Algorithms}

\lstnewenvironment{MapleInput}{%
\lstset{
	basicstyle=\bfseries\upshape\ttfamily\color{mapleinput},
	numberstyle=\bfseries\upshape\ttfamily\color{mapleinput},
	breaklines=true,
	columns=flexible,
	breakautoindent=false,
	breakindent=0pt,
	xleftmargin=4ex,
	numbers=left,
	stepnumber=1,
	numbersep=1.3ex,
	aboveskip=\smallskipamount,
	belowskip=\smallskipamount,
}%
}{%
}

\definecolor{mapleinput}{rgb}{0.5,0.0,0.0}
\definecolor{maplemath}{rgb}{0.0,0.0,1.0}
\newenvironment{MapleMath}{%
\color{maplemath}\upshape\rmfamily%
\setlength{\abovedisplayskip}{0ex}%
\setlength{\abovedisplayshortskip}{\abovedisplayskip}%
\setlength{\belowdisplayskip}{\medskipamount}%
\setlength{\belowdisplayshortskip}{0ex}%
\csname gather*\endcsname}{\csname endgather*\endcsname%
{\hrule height 0pt}%
\ignorespacesafterend}
\DeclareMathOperator{\Schubert}{Schubert}

\theoremstyle{plain}
\newtheorem{theorem}{Theorem}[section]
\newtheorem{lemma}[theorem]{Lemma}
\newtheorem{corollary}[theorem]{Corollary}
\newtheorem{proposition}[theorem]{Proposition}
\newtheorem{conjecture}[theorem]{Conjecture}

\theoremstyle{definition}
\newtheorem{example}[theorem]{Example}
\newtheorem{definition}[theorem]{Definition}

\theoremstyle{remark}
\newtheorem{remark}[theorem]{Remark}

\numberwithin{equation}{section}

\newcommand{\Q}{\mathbbm{Q}}
\newcommand{\R}{\mathbbm{R}}
\newcommand{\C}{\mathbbm{C}}
\newcommand{\Z}{\mathbbm{Z}}

\newcommand{\sc}{\mathsf{c}}
\newcommand{\cc}{c}

\newcommand{\asyO}[1]{\mathcal{O}(#1)}
\newcommand{\defas}{\mathrel{\mathop:}=}
\newcommand{\abs}[1]{\left|#1\right|}

\newcommand{\Period}{\mathcal{P}}
\newcommand{\PeriodC}{\widetilde{\Period}}
\newcommand{\Martin}{\operatorname{\mathsf{M}}}
\newcommand{\Tutte}[1]{\mathsf{T}_{#1}}
\newcommand{\Flow}[1]{\mathsf{F}_{#1}}
\DeclareMathOperator{\rk}{rk}
\newcommand{\loops}{\ell}

\newcommand{\Graph}[2][1.0]{\vcenter{\hbox{\includegraphics[scale=#1]{graphs/#2}}}}
\newcommand{\nuc}{\mathscr{N}} 
\newcommand{\UM}[2]{\mathsf{U}^{#1}_{#2}}
\newcommand{\Fano}{F_7}
\newcommand{\Vamos}{V_8}
\newcommand{\FP}{\mathsf{N}} 
\newcommand{\FPv}{\mathsf{N}^\circ} 
\newcommand{\els}{\mathsf{els}}
\newcommand{\LCF}{\mathcal{Z}} 
\newcommand{\LF}{\mathcal{F}} 
\newcommand{\LP}{\mathcal{P}} 
\newcommand{\chains}{\mathscr{C}} 
\newcommand{\SM}[1]{\mathsf{SM}\left(#1\right)} 
\DeclareMathOperator{\cyc}{cyc} 
\DeclareMathOperator{\cl}{cl} 
\newcommand{\PU}{\mathsf{N}} 
\newcommand{\PR}{\mathsf{E}} 
\newcommand{\LPM}[1]{\mathsf{LM}\left(#1\right)} 
\newcommand{\bigO}[1]{\mathcal{O}\left(#1\right)}
\newcommand{\admsq}[1]{\square(#1)} 
\DeclareMathOperator{\trM}{tr} 

\newcommand{\Maple}{\href{http://www.maplesoft.com/products/Maple/}{\textsf{\textup{Maple}}}}
\newcommand{\MapleNote}{\footnote{Maple is a trademark of Waterloo Maple Inc.}}
\newcommand{\MapleTM}{\href{http://www.maplesoft.com/products/Maple/}{\textsf{\textup{Maple}}\texttrademark}}
\newcommand{\nauty}{\href{http://pallini.di.uniroma1.it/}{\texttt{\textup{nauty}}}}
\newcommand{\Sage}{\href{http://www.sagemath.org/}{Sage}}
\newcommand{\HoG}[1]{\href{https://houseofgraphs.org/graphs/#1}{#1}}
\newcommand{\oeis}[1]{\href{http://oeis.org/#1}{$\mathsf{#1}$}}
\newcommand{\Filename}[1]{{\upshape\ttfamily #1}}

\newcommand{\dataurl}{https://people.maths.ox.ac.uk/panzer/data/g.tar.gz}
\newcommand{\programurl}{https://bitbucket.org/PanzerErik/gspeyer}

\title{Graph theoretic properties of Speyer's matroid polynomial $g_M(t)$}

\newcommand{\email}[1]{\href{mailto:#1}{#1}}
\author{%
    \thanks{Mathematical Institute, University of Oxford, OX2 6GG, UK, \email{erik.panzer@maths.ox.ac.uk}}
    Erik Panzer
    \orcidlink{0000-0002-9897-5812}
}

\begin{document}

\maketitle

\begin{abstract}
    We prove relations between the number of $k$-connected components of a graph, Crapo's invariant $\beta(M)$ of a matroid, and Speyer's polynomial $g_M(t)$. These yield a simple interpretation of $g_M'(-1)$ when $M$ is graphic or cographic.
    Furthermore, we improve Ferroni's algorithm to compute $g_M(t)$ and provide an implementation and an extensive data set.
    These calculations reveal a large number of graph theoretic constraints on the second derivative $g_M''(-1)$, which we thus advertise as an intriguing new invariant of graphs. We also propose a relation between the flow polynomial and $g_M''(0)$ for cubic graphs.
\end{abstract}

\section{Introduction}
Isomorphism invariants such as the classical Tutte polynomial $\Tutte{G}(x,y)=\sum_{i,j} t_{i,j}(G)x^iy^j$ are fundamental tools to study graphs and, more generally, matroids \cite{BrylawskiOxley:TutteApp}.
The linear coefficient $\beta(G)=t_{1,0}(G)\in\Z$ in $x$ of $\Tutte{G}(x,y)$ was studied by Crapo \cite{Crapo:HigherInvariant} and captures basic structural features; for example $\beta(G)=0$ if $G$ is not biconnected or $G$ is a self-loop, $\beta(G)=1$ if $G$ is series-parallel, and $\beta(G)>1$ otherwise. In matroid language, $\beta(M)$ is a signed sum over the ranks of all subsets of the ground set (edges):
\begin{equation}\label{eq:Crapo}
    \beta(M)=(-1)^{\rk(M)} \sum_{A\subseteq M} (-1)^{\abs{A}} \rk(A).
\end{equation}

\subsection{Graphs: Crapo and Elser}
Our first results are new graph theoretical interpretations of this invariant. Let $\cc(M)$ denote the number of connected components of a matroid, that is, the largest integer $k$ such that $M$ can be decomposed as a direct sum $M=M_1\oplus\ldots\oplus M_k$ of non-empty matroids $M_i$. In the case where $M$ is the cycle matroid of a graph, $\cc(M)$ counts the number of blocks (biconnected components), that is, the equivalence classes of edges such that two edges are equivalent if they are contained in a cycle.

Using an identity of Elser \cite{Elser:ClusterSAW}, we identify $\beta(G)$ in \Cref{lem:beta=els0} with the \emph{zeroth Elser number} of a graph---answering a question raised in \cite{DBHLMNVW:Elser}---and apply this to prove:
\begin{theorem}\label{thm:betark=bic}
    For any graphic or cographic matroid $M$ with at least two edges,
    \begin{equation}\label{eq:betark=bic}
        \beta(M)\rk(M) = (-1)^{\rk(M)}\sum_{A\subseteq M} (-1)^{\abs{A}} \cc(A).
    \end{equation}
\end{theorem}
\begin{table}
    \centering
    \begin{tabular}{rccccccc}
    \toprule
    $A$ & $\emptyset$ & $\phantom{-}\Graph[0.35]{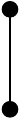}$ & $\Graph[0.35]{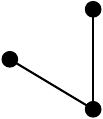}$ & $\Graph[0.35]{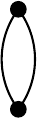}$ & $\Graph[0.35]{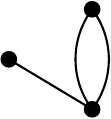}$ & $\Graph[0.35]{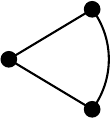}$ & $\Graph[0.35]{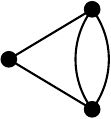}$ \\
    \midrule
    copies           & 1 & $\phantom{-}4$           & 5 & 1 & 2 & 2 & 1 \\
    $\cc(A)$         & 0 & $\phantom{-}1$           & 2 & 1 & 2 & 1 & 1 \\
    $(-1)^{\abs{A}}$ & 1 & $-1$ & 1 & 1 & \llap{$-$}1 & \llap{$-$}1 & 1 \\
    \bottomrule
    \end{tabular}
    \caption{All 16 edge subsets $A$ of the graph $\Graph[0.25]{duncA}$, grouped by isomorphism type, and their contributions to \eqref{eq:betark=bic}. They sum to $0-4+10+1-4-2+1=2$.}%
    \label{tab:bic-dunce}%
\end{table}
For the series-parallel graph $G=\Graph[0.25]{duncA}$, we illustrate the contributions to the right-hand side in \cref{tab:bic-dunce}. They sum to $\beta(G)\rk(G)=1\cdot 2$ on the left-hand side.

For matroids that are not (co)graphic, both sides of \eqref{eq:betark=bic} are well defined, but the identity typically fails.
For example, let $\UM{n}{r}$ denote the uniform matroid of rank $r$ on $n$ elements.
Then $\UM{4}{2}$ gives $\beta(\UM{4}{2})\rk(\UM{4}{2})=2\cdot 2$ on the left, whereas the right-hand side is
\begin{align*}
    &\cc(\emptyset)-4\cc(\UM{1}{1})+6\cc(\UM{2}{2})-4\cc(\UM{3}{2})+\cc(\UM{4}{2})
    =0-4+6\cdot 2-4+1
    = 5.
\end{align*}
Similarly, for most connected matroids that are neither graphic nor cographic, the left-hand side of \eqref{eq:betark=bic} is strictly smaller than the right-hand side.
With {\Sage} \cite{sagemath} we tested all connected matroids with 8 elements or less and found merely four exceptions: the Fano matroid where $\beta(\Fano)\rk(\Fano)=3\cdot 3$ in fact exceeds the value $8$ on the right, its series- and parallel extensions $\Fano\oplus_2 \UM{3}{2}$ and $\Fano\oplus_2 \UM{3}{1}$ (again the left side exceeds the right by one), and the matroid $S_8$ from \cite[p.~648]{Oxley:MatroidTheory} for which both sides are equal (to $16$).

\Cref{thm:betark=bic} can also be stated as follows:
Let $\loops(A)=\abs{A}-\rk(A)$ denote the corank (cyclomatic number) of any set of edges.
Then for every graph without self-loops,
\begin{equation*}
    \sum_{A\subseteq E(G)} (-1)^{\loops(A)} \beta(A) \rk(A) = \cc(G)
\end{equation*}
where we set $\cc(G)=\cc(M)$ for the cycle matroid of $G$. For a biconnected graph, $\cc(G)=1$. More generally, a connected graph has $\cc(G)=1+\sum_v(\abs{\pi_0(G\setminus v)}-1)$ blocks \cite{Harary:Elementary}, using $\pi_0$ to denote the set of connected components of a graph; the sum runs over all vertices $v$. We prove a generalization of this relation. A graph is \emph{$k$-(vertex-)connected} if it has more than $k$ vertices and if $G\setminus S$ is connected for any subset $S$ of less than $k$ vertices.
\begin{theorem}\label{thm:betark-k=cc}
    Let $G$ be a $k$-connected graph without self-loops, where $k\geq 1$. Then
    \begin{equation}\label{eq:betark-k=cc}
        (-1)^{k+1}\sum_{A\subseteq E(G)} (-1)^{\loops(A)} \beta(A) \binom{\rk(A)}{k} = 1+\sum_{\substack{S\subset V(G) \\ \abs{S}=k}} \Big(\abs{\pi_0(G\setminus S)} -1 \Big).
    \end{equation}
\end{theorem}
Remarkably, this identity links a purely matroid-theoretic invariant on the left-hand side, to a graph-theoretic invariant (making explicit reference to vertices) on the right-hand side.
The case $k=1$ is essentially equivalent to \Cref{thm:betark=bic}. Note that for a graph that is $K$-connected, the right-hand side of \eqref{eq:betark-k=cc} equals $1$ for every choice of $1\leq k<K$.
\begin{example}
    Consider $G$ the wheel graph with 4 spokes. For the sum on the left, we only need to consider subgraphs $A$ that are biconnected, because otherwise $\beta(A)=0$. \Cref{tab:ex-ws4} shows all such $A$ and their corresponding contributions to the sum. Their total
    \begin{equation*}
        (-1)^{k+1}\left(
            8\binom{1}{k}
            -4\binom{2}{k}
            -\binom{3}{k}
            +\binom{4}{k}
        \right)
    \end{equation*}
    evaluates to $1$ for $k=1$ and $k=2$, reflecting the fact that $G$ is 3-connected. At $k=3$, the total takes the value $3=1+2$, which accounts for the two 3-vertex cuts of $G$.
\end{example}
\begin{table}
    \centering
    \begin{tabular}{rcccccccccc}
    \toprule
         $A$ &          $\Graph[0.3]{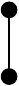}$ & $\Graph[0.3]{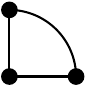}$ & $\Graph[0.3]{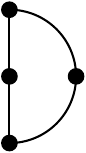}\cong\Graph[0.3]{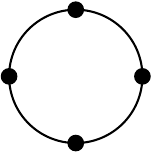}$ & $\Graph[0.3]{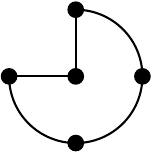}$ & $\Graph[0.3]{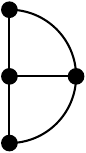}$ & $\Graph[0.3]{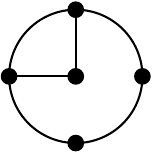}\cong\Graph[0.3]{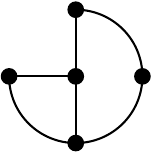}$ & $\Graph[0.3]{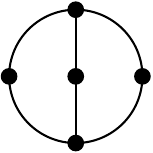}$ & $\Graph[0.3]{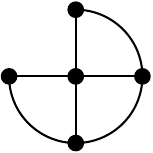}$ & $\Graph[0.3]{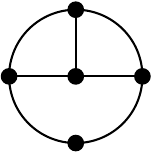}$ & $\Graph[0.3]{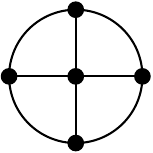}$\\
    \midrule
         copies &            8 & 4 & 4+1 & 4 & 4 & 4+8 & 2 & 4 & 4 & 1\\
         $\beta(A)$ &        1 & 1 & 1 & 1 & 1 & 1 & 1 & 1 & 2 & 3 \\
         $(-1)^{\loops(A)}$ &  1 & \llap{$-$}1 & \llap{$-$}1 & \llap{$-$}1 & 1 & 1 & 1 & \llap{$-$}1 & \llap{$-$}1 & 1\\
         $\rk(A)$ &          1 & 2 & 3 & 4 & 3 & 4 & 4 & 4 & 4 & 4\\
    \bottomrule
    \end{tabular}
    \caption{Contributions to the left-hand side of \Cref{thm:betark-k=cc} for the 4-wheel $G=\Graph[0.2]{ws4A}$.}%
    \label{tab:ex-ws4}%
\end{table}
\begin{remark}
For a geometric interpretation of the case $k=1$, note that $\abs{M}-\cc(M)$ is the dimension of the matroid basis polytope
\cite[Theorem~6]{Shapley:Cores}.                             
Expanding $\beta(A)$ with \eqref{eq:Crapo}, we thus obtain a formula for dimension of the spanning tree polytope \cite{Chopra:STpolyhedron} of a graph:
\begin{equation*}
    \abs{E(G)}-\cc(G)=\sum_{A\subseteq B\subseteq E(G)} (-1)^{\abs{B\setminus A}} \loops(A) \rk(B).
\end{equation*}
\end{remark}

\subsection{Matroids: Speyer and Ferroni}
Our second set of results links these identities to Speyer's polynomial $g_M(t)\in t\Z[t]$. This invariant was constructed from $K$-theory in \cite{Speyer:MatroidKtheory} for all realizable matroids, and extended to all matroids in \cite{FinkSpeyer:K-classes}. It has non-negative coefficients \cite[Theorem~E]{BergetFink:ActiPair} and further intriguing properties, but combinatorial interpretations of its evaluations are rare:
\begin{itemize}
\item
For every matroid $M$ without loops or coloops, \cite{Speyer:MatroidKtheory,FinkSpeyer:K-classes} show that
    \begin{equation}\label{eq:gm10}
        g_M(-1)=(-1)^{\cc(M)}
        \qquad\text{and}\qquad
        g'_M(0)=\beta(M).
    \end{equation}
\item The coefficient 
of the highest power $t^{\rk(M)}$ of $g_M(t)$ is studied in \cite{FinkShawSpeyer:Omega}, especially for matroids whose rank is very small or close to half the number of edges.
\item If $M$ is a Schubert matroid, then the coefficients $\FP_i(M)\in\Z$ in the expansion
\begin{equation}\label{eq:gexpand1}
    g_M(t)=t\sum_{i=0}^{\rk(M)-1} \FP_i(M) \cdot (1+t)^i
\end{equation}
are non-negative and count certain lattice paths with $i$ diagonal steps \cite{Ferroni:SchubertDelannoySpeyer}.
\end{itemize}
In terms of the expansion \eqref{eq:gexpand1}, the first identity of \eqref{eq:gm10} amounts to $\FP_0(M)=(-1)^{\cc(M)-1}$ for any matroid without loops or coloops.

We add to these a useful formula for the next coefficient $\FP_1(M)$, which we state equivalently in terms of the first derivative $g'_M(-1)=(-1)^{\cc(M)-1}-\FP_1(M)$:
\begin{proposition}\label{prop:N1asbeta}
    For every matroid $M$ without loops or coloops,
    \begin{equation*}
        g'_M(-1)=(-1)^{\cc(M)-1}\sum_{A\subseteq M} (-1)^{\loops(A)} \beta(A) \rk(A)
        .
    \end{equation*}
\end{proposition}
\begin{remark}
Expanding $\beta(A)$ according to \eqref{eq:Crapo}, \Cref{prop:N1asbeta} expresses $g_M'(-1)$ as the alternating double sum of $(-1)^{\abs{A}}\rk(B)\rk(A)$ over nested pairs $B\subseteq A\subseteq M$. Therefore, $g_M'(-1)$ can be obtained as an evaluation of a derivative from other matroid polynomials; e.g.\ the \emph{M\"{o}bius polynomial} \cite[Definition~4]{Jurrius:RelMoeCob} or the \emph{2nd chain Tutte polynomial} \cite[\S5.2]{Wakefield:ChainTutte}.
\end{remark}
\begin{example}\label{ex:g'(UM)}
Take a uniform matroid $M=\UM{n}{r}$ without loops or coloops ($1\leq r<n$). Then every $A\subseteq M$ is itself uniform: $A\cong\UM{i}{r}$ when $i\defas\abs{A}>r$, and $A\cong\UM{i}{i}=(\UM{1}{1})^{\oplus i}$ when $i\leq r$. The latter is disconnected with $\beta(A)=0$, unless $i=1$ where $\beta(A)=1$. Using $\beta(\UM{i}{r})=\binom{i-2}{r-1}$ from \cite[Proposition~7]{Crapo:HigherInvariant} in the cases $i>r$, some basic algebra shows
\begin{equation*}
    g'_{\UM{n}{r}}(-1)
    = \sum_{e\in\UM{n}{r}} 1 + \sum_{i=r+1}^n \binom{n}{i} (-1)^{i-r} \binom{i-2}{r-1} r
    =n-r(n-r).
\end{equation*}
\end{example}

Specializing to (co)graphic matroids, combining \Cref{prop:N1asbeta} with \Cref{thm:betark=bic} proves a simple interpretation of $g'_M(-1)$:
\begin{theorem}\label{thm:gprime=cc}
    For every graphic or cographic matroid $M$ without any loops or coloops,
    \begin{equation*}
        g_M'(-1)=(-1)^{\cc(M)-1} \cc(M)
        .
    \end{equation*}
\end{theorem}
Cycle matroids of biconnected graphs thus satisfy $g'_M(-1)=1$ and $\FP_1(M)=0$. These identities for (co)graphs do not extend to arbitrary matroids; e.g.\ the calculation in \Cref{ex:g'(UM)} proves that uniform matroids $\UM{n}{r}$ with $1\leq r\leq n-1$ have $g'_M(-1)=1$ if and only if $r=1$ or $r=n-1$. We can thus deduce from \Cref{thm:gprime=cc} that all uniform matroids with $2\leq r\leq n-2$ are \emph{not} graphic or cographic. This result is well-known, but it seems remarkable that it can be detected by such a simple, algebraic criterion.
\begin{corollary}\label{cor:N1-graphic-criterion}
    If a connected matroid $M$ with at least two elements has $g'_M(-1)\neq 1$, then $M$ is not (co)graphic.
\end{corollary}

Similarly one calculates that $g'_M(-1)=0$ and $\FP_1(M)=1$ for the regular matroid $R_{10}$, which again reproves that it is not (co)graphic.\footnote{The Speyer polynomial is
$g_{R_{10}}^{}(t)=t(10+35t+45t^2+20t^3+t^4)$.
}
For further examples see \cite[Table~5.1]{Ferroni:SchubertDelannoySpeyer}, where $\FP_1(M)$ can be read off as the linear coefficient of $t$ in the rightmost column headed ``$\widetilde{g}(t-1)$''; e.g.\ $\FP_1(\Fano)=-1$ and $\FP_1(\Vamos)=4$ for the Fano and V\'{a}mos matroid.
The converse of \Cref{cor:N1-graphic-criterion} is false: there exist connected matroids that are not (co)graphic but have $\FP_1(M)=0$. There is precisely one such case in \cite[Table~5.1]{Ferroni:SchubertDelannoySpeyer} (the non-Fano matroid), but more examples can be constructed from $2$-sums, like $\Fano\oplus_2 R_{10}$.

\begin{remark}\label{rem:FP1-valuative}
Since $g_M(t)$ is covaluative \cite{Speyer:MatroidKtheory}, it follows that $\FPv_1(M)\defas(-1)^{\cc(M)}\FP_1(M)\in\Z$ is a valuative matroid invariant. It can take negative values, e.g.\ $\FPv_1(V_8)=-4$, but our results above show that for all graphic or cographic matroids $M$, its value
\begin{equation*}
    \FPv_1(M)=\cc(M)-1\geq 0
\end{equation*}
is non-negative. The study of similar valuative constraints on subclasses of matroids was proposed in \cite{FerroniFink:PolMat}, where a \emph{polytope of all matroids} is defined. The hyperplane $\{\FPv_1= 0\}$ cuts this polytope such that all (co)graphic matroids lie on the side $\FPv_1\geq 0$. On the other side ($\FPv_1<0$), we find, for example, all non-(co)graphic uniform matroids $\UM{n}{r}$ with $2\leq r\leq n-2$. For indeed, we saw in \Cref{ex:g'(UM)} that then
\begin{equation*}
    \FPv_1(\UM{n}{r}) = g'_{\UM{n}{r}}(-1)-1
    =-(r-1)(n-r-1)
    <0.
\end{equation*}
\end{remark}
\begin{remark}
    Following \cite[\S7]{FerroniFink:PolMat}, we ask: Which matroids minimize or maximize $\FPv_1(M)$ at a given rank and size? We did not study this question closely, but neither (co)graphic nor uniform matroids are candidates: \cite[Table~5.1]{Ferroni:SchubertDelannoySpeyer} has an example with $\FPv_1(M)=20$ at rank $5$ (larger than the (co)graphic matroids), and calculations for projective geometries indicate that the negative values for uniform matroids are not minimal.
\end{remark}
\begin{remark}
    The properties of $g_M(t)$ proved in \cite{Speyer:MatroidKtheory} also imply that $\FPv_1(M^\star)=\FPv_1(M)$ for the dual matroid $M^\star$,
    $\FPv_1(A\oplus B)=\FPv_1(A)+\FPv_1(B)+1$, and $\FPv_1(A\oplus_2 B)=\FPv_1(A)+\FPv_1(B)$ for $1$- and $2$-sums. We suspect that a similar relation reduces $\FPv_1$ for $3$-sums. Seymour's decomposition theorem \cite{Seymour:DecompositionRegular} combined with $\FPv_1(G)=\cc(G)-1$ and $\FPv_1(R_{10})=-1$ should thus yield a simple characterization of $\FPv_1(M)$ for all regular matroids, in terms of the number $\cc(M)$ of components and the number of $R_{10}$ factors.
\end{remark}

For all biconnected graphs $G$ with at least two edges, we know that $g_M(-1)=-1$, and now also that $g'_M(-1)=1$, are trivial invariants---for $M$ the cycle matroid of $G$. In terms of the coefficients $\FP_i(G)\defas \FP_i(M)$ in the expansion \eqref{eq:gexpand1}, this amounts to
\begin{equation*}
    \FP_0(G)=1 \quad\text{and}\quad \FP_1(G)=0.
\end{equation*}
In the final part of this paper, we thus study the second derivative, $g''_M(-1)=-2\FP_2(G)$, as an interesting invariant of graphs in its own right. From the properties of Speyer's polynomial shown in \cite{Speyer:MatroidKtheory}, and the vanishing $\FP_1(G)=0$ for biconnected $G$, we know that:
\begin{itemize}
    \item $\FP_2(G)=0$ if $G$ has a (self-)loop or a bridge (coloop), or if $G$ is series-parallel,
    \item $\FP_2(G)$ is invariant under series-parallel operations,
    \item $\FP_2(G)=\FP_2(G^\star)$ for planar graphs $G$ with dual $G^\star$,
    \item $\FP_2(G_1\oplus_2 G_2)=\FP_2(G_1)+\FP_2(G_2)$ for the 2-sum of two biconnected graphs.
\end{itemize}
The last property follows from $g_{A\oplus_2 B}(t)=g_A(t)g_B(t)/t$. Also, note $g_{A\oplus B}(t)=g_A(t) g_B(t)$. Since every graph $G$ decomposes, via 1- and 2-sums, into a unique collection of 3-connected components $G_i$ and series-parallel pieces (cycles and bonds) \cite{CunninghamEdmonds:Decomposition}, we find that
\begin{equation}\label{eq:N2-into-3con}
    \FP_2(G)=(-1)^{\cc(G)-1}\left(\binom{\cc(G)-1}{2}+\sum_i \FP_2(G_i)\right).
\end{equation}
Therefore, it suffices to focus on graphs $G$ that are 3-connected, i.e.\ connected graphs with at least four vertices and no 1- or 2-vertex cuts.

In order to compute $g_M(t)$ and $\FP_2(G)$ efficiently, we propose several improvements to Ferroni's algorithm \cite[\S5]{Ferroni:SchubertDelannoySpeyer}. The resulting recursive method is described as \Cref{alg:gRec} in \cref{sec:algrec}, and we provide an open source implementation at
\begin{center}
    \url{\programurl}
\end{center}
This is efficient enough for quite large graphs, e.g.\ the circulant $C^{18}_{1,8}$ (see \cref{fig:xladders}) gives
\begin{align*}
    g_{C^{18}_{1,8}}(t) &=
    135t^{17}+10896t^{16}+228042t^{15}+2253582t^{14}+13169952t^{13}+50942898t^{12} \\ &
    +139071708t^{11}+278452740t^{10}+418235148t^9+476553546t^8+412635306t^7 \\ &
    +269412876t^6+130133556t^5+44892552t^4+10385851t^3+1426796t^2+86189t
\end{align*}
and thus $\FP_2(C_{1,8}^{18})=247$ on a single 1 GHz CPU within one day. This demonstrates a significant improvement, since the original algorithm from \cite{Ferroni:SchubertDelannoySpeyer} involves a sum over all chains of cyclic flats.
The graph $C^{18}_{1,8}$ has \numprint{870137905746} such chains, rendering the original algorithm impractical.
\begin{figure}
    \centering\renewcommand{\arraystretch}{1.3}
    \begin{tabular}{cccccc}
    $\Graph[0.5]{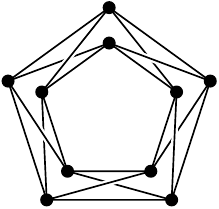}$ & $\Graph[0.5]{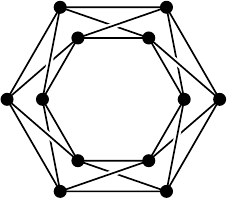}$ & $\Graph[0.5]{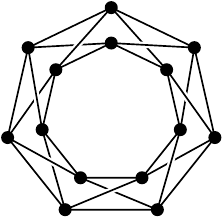}$ & $\Graph[0.5]{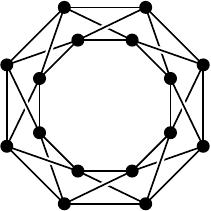}$ & $\Graph[0.5]{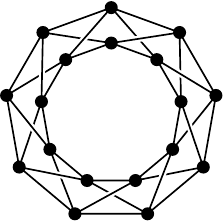}$ & $\Graph[0.5]{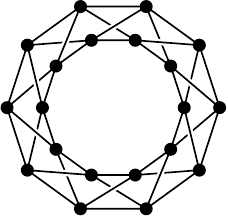}$ \\[2mm]
    $C^{10}_{1,4}$ & $C^{12}_{1,5}$ & $C^{14}_{1,6}$ & $C^{16}_{1,7}$ & $C^{18}_{1,8}$ & $C^{20}_{1,9}$ \\
    \end{tabular}
    \caption{Drawings of the circulant graphs $C^{2n}_{1,n-1}$ on $2n$ vertices.}%
    \label{fig:xladders}%
\end{figure}

\subsection{Experiments: Speyer's polynomial in graph theory}
In \Cref{sec:data} we report on calculations performed with this method. We computed Speyer polynomials for more than 3 million graphs; this open data set \cite{Data:gSpeyer} is available at
\begin{center}
    \url{\dataurl}
\end{center}
Supported by this data,\footnote{When in subsequent examples we state \texttt{graph6} encodings, these refer to the graph label in those files.} we observe several intriguing structural properties of the graph invariant $\FP_2(G)$. In contrast to the identities of $\FP_2(G)$ mentioned earlier, the following new properties are \emph{not} obvious from known properties of Speyer's polynomial:
\begin{conjecture}\label{con:planar=1}
    If a $3$-connected graph $G$ is planar, then $\FP_2(G)=1$.
\end{conjecture}
This condition can detect non-planarity of some graphs, for example one computes that $\FP_2(K_5)=\FP_2(K_{3,3})=0$ for the Kuratowski graphs $K_5$ and $K_{3,3}$.
From \eqref{eq:N2-into-3con} we see that the invariant $\FPv_2(M)\defas(-1)^{\cc(M)-1}\FP_2(M)$ satisfies $\FPv_2(G)\geq 0$ for all planar graphs (not necessarily 3-connected). Similarly to \Cref{rem:FP1-valuative}, we may thus view $\FPv_2(M)\geq 0$ as a valuative constraint that is \emph{necessary} for a matroid $M$ to be (graphic and) planar.
\begin{remark}
    With $\FP_2(G)$ becoming fixed, it might be interesting to study and interpret the next invariant $\FP_3(G)$ especially for planar graphs.
\end{remark}
\begin{conjecture}\label{con:3sum}
    Let $G_1$ and $G_2$ denote two 3-connected graphs, each having a triangle $T_i\subset G_i$, and let $G_1\oplus_3 G_2$ denote a corresponding 3-sum. Let $p_i=\abs{\pi_0(G_i\setminus V(T_i))}$ be the number of connected components left over when the triangle vertices are deleted. Then
    \begin{equation*}
        \FP_2(G_1 \oplus_3 G_2)=\FP_2(G_1)+\FP_2(G_2)-p_1\cdot p_2.
    \end{equation*}
\end{conjecture}
A 3-sum \cite[\S9.3]{Oxley:MatroidTheory} of two graphs $G_i$ along triangles $T_i\subset G_i$ (see \cref{fig:3sum-twist}) is a graph $G$ obtained from gluing $G_1$ to $G_2$ by identifying $T_1$ with $T_2$ and then deleting the triangle edges.\footnote{Since the 3-sum depends on the choice of triangles, and furthermore on the choice of how they are identified, two given graphs may give rise to several, non-isomorphic 3-sums.}
In a 3-sum, the identified vertices $S\subset V(G)$ form a 3-vertex cut, such that the two parts of the edge partition $E(G)=A_1\sqcup A_2$ into $A_i=E(G_i)\setminus E(T_i)$ meet only in $S$. Conversely, given any bipartition $E(G)=A_1\sqcup A_2$ such that both parts meet only in a 3-vertex cut $S$, we can realize $G=G_1\oplus_3 G_2$ as a 3-sum of the graphs $G_i=A_i\sqcup T$ obtained from $A_i$ by adding 3 extra edges $T=\binom{S}{2}$ that form a triangle on $S$.
\begin{figure}
    \centering
    $\Graph[0.5]{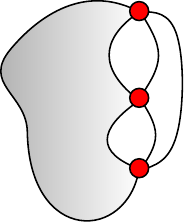} \oplus_3 \Graph[0.5]{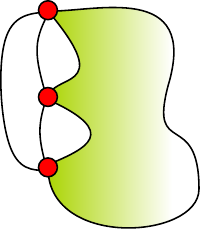} = \Graph[0.5]{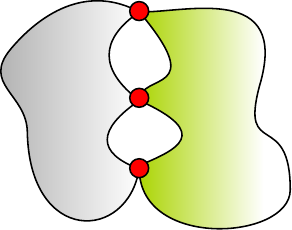}$
    \qquad\qquad
    $\Graph[0.43]{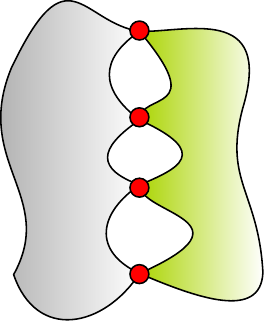} \ \leftrightarrow\  \Graph[0.4]{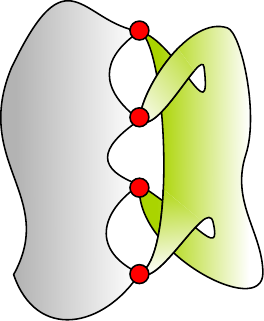}$
    \caption{A 3-sum of two graphs (left) and a twist of a graph along a 4-vertex cut (right).}%
    \label{fig:3sum-twist}%
\end{figure}
\begin{example}
    Every 3-valent vertex $v$ in a graph $G$ gives rise to a 3-sum representation $G=K_4\oplus_3 G_v$, where $G_v=(G\setminus \{v\})\sqcup T$ is the star-triangle (or $Y$--$\Delta$) transform of $G$. Let $S$ denote the 3 neighbours of $v$. Since $\FP_2(K_4)=1$, \Cref{con:3sum} specializes to
\begin{equation}\label{eq:N2star-triangle}
    \FP_2(G)=\FP_2(G_v)+1-\abs{\pi_0(G_v\setminus S)}
    =\FP_2(G_v)-\left(\abs{\pi_0(G\setminus S)}-2\right).
\end{equation}
Hence, a star--triangle reduction typically leaves $\FP_2(G)=\FP_2(G_v)$ invariant. Only when $G\setminus S=(G_v\setminus S)\sqcup \{v\}$ has more than two components, $\FP_2(G)$ decreases.
\end{example}
\begin{remark}\label{rem:K5plusK5}
In contrast to 1- and 2-sums, the polynomial $g_M(t)$ does \emph{not} factorize for 3-sums. For example, apart from the trivial zero $g_M(0)=0$, the polynomial
\begin{equation*}
    g_{K_5\oplus_3 K_5}(t)=14t^6+81t^5+183t^4+202t^3+109t^2+24t
    \quad\text{of}\quad
    \Graph[0.5]{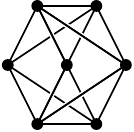}=K_5\oplus_3 K_5
\end{equation*}
is irreducbile over $\Q$ and not simply related to the square of $g_{K_5}(t)=5t^4+15t^3+15t^2+6t$. Note that $24=\beta(K_5\oplus_3 K_5)\neq 36=\beta(K_5)^2$. Nevertheless, we see that $\FP_2(K_5\oplus K_5)=-1$ and $\FP_2(K_5)=0$ are in agreement with \Cref{con:3sum} since $p_1=p_2=1$.
\end{remark}

By \Cref{con:3sum}, the computation of $\FP_2(G)$ can be reduced to 4-connected graphs. Even in this highly connected case, there remain many relations:
\begin{conjecture}\label{con:twist}
    If two 3-connected graphs $G$ and $G'$ are related by a twist along a minimal\footnote{Minimality means that no proper subset of those four vertices already suffices to cut the graph.} 4-vertex cut (see \cref{fig:3sum-twist}), then $\FP_2(G)=\FP_2(G')$.
\end{conjecture}
The twist relation was introduced, in the special case of 4-regular graphs, in \cite{Schnetz:Census}. It is defined as follows: Suppose that $S\subset V(G)$ is a 4-vertex cut and let $E(G)=A\sqcup B$ be a bipartition of the edges such that $A$ and $B$ meet only at the vertices in $S$. Pick a labeling $S=\{v_1,v_2,v_3,v_4\}$ of the vertices in $S$. Then the corresponding twist $G'$ of $G$ has the same vertices as $G$, but as indicated in \cref{fig:3sum-twist}, the edges are given by
\begin{equation*}
    E(G')=A\sqcup \bigg(B\Big|_{\substack{v_1\leftrightarrow v_2\\ v_3\leftrightarrow v_4}}\bigg),
\end{equation*}
that is, by re-attaching the edges in $B$ according to a double transposition of $S$. Note that a single graph may have many non-isomorphic twists (choice of $S$ and labelling).
\begin{example}
    A non-trivial twist $G'$ of the graph $G=K_5\oplus_3 K_5$ is shown in \cref{fig:K5K5-twist}. Since $g_M(t)$ is invariant under parallel operations, we have $g_{G'}(t)=g_{G''}(t)$ for the simplifcation $G''$ of $G'$. The Speyer polynomials of these graphs are
    \begin{align*}
        g_{G}(t)  &= 14t^6+81t^5+183t^4+202t^3+109t^2+24t \quad\text{and}\\
        g_{G'}(t) &= 4t^6+35t^5+99t^4+126t^3+75t^2+18t,
    \end{align*}
    so that indeed $\FP_2(G)=\FP_2(G')=-1$ are equal as predicted by \Cref{con:twist}.
\end{example}
\begin{figure}
    \centering
    $G=\Graph[0.7]{K5K5pretwist} \cong \Graph[0.6]{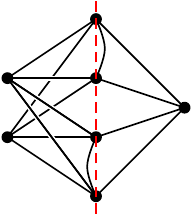} \quad\mapsto\quad G'=\Graph[0.6]{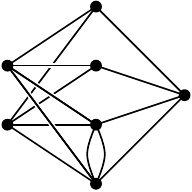} \quad\mapsto\quad G''=\Graph[0.7]{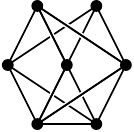}$
    \caption{The graph $G=K_5\oplus_3 K_5$ with an edge bipartition, the corresponding twist $G'$ for the labelling $v_1,v_3,v_2,v_4$ of the cut vertices from top to bottom, and its simplification $G''$ (reducing the parallel pair of edges).}%
    \label{fig:K5K5-twist}%
\end{figure}

Taken together, the series-parallel invariance, the identity $\FP_2(G)=1$ for planar graphs, the 3-sum (e.g.\ star-triangle) reduction, and the twist identity, constitute a stringent set of constraints that determines a huge number of relations between the invariants $\FP_2(G)$ of different graphs. We hope that its intriguing properties will stimulate further research and lead to a better understanding of $\FP_2(G)$ and its role in graph theory.

In our final group of observations, we depart from the focus on $\FP_2(G)$ and the expansion of the polynomial $g_M(t)$ around $t=-1$, and instead report new properties of the first two coefficients $g_G'(0)=\beta(G)$ and $g_G''(0)$ of the expansion around $t=0$. The examples above show that these coefficients do \emph{not} satisfy the general 3-sum and twist identities. However, if we restrict these operations, we find that they do.
\begin{conjecture}\label{con:3edge}
    Let $G$ be a 3-(vertex-)connected graph with a 3-edge cut $C$, that is, a set of 3 edges such that $G\setminus C=S\sqcup T$ has two connected components. Let $A=G/T$ and $B=G/S$ denote the graphs obtained by contracting one side of the cut (replacing the contracted side by a 3-valent vertex). Then
    \begin{equation*}
        g_G(t) \equiv \frac{g_A(t)g_B(t)}{t} \mod t^3\Z[t].
    \end{equation*}
\end{conjecture}
\begin{figure}
    \centering
    $G= \Graph[0.8]{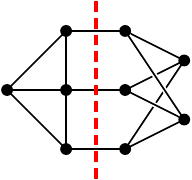} \qquad\mapsto\qquad A=\Graph[0.8]{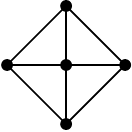}\qquad B=\Graph[0.8]{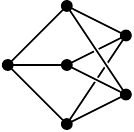}$
    \caption{Splitting of a graph ($G=\texttt{HoCQXZo}$ in \texttt{graph6} format) along a 3-edge cut.}%
    \label{fig:3edge-cut}%
\end{figure}
For the linear coefficent \eqref{eq:gm10}, this amounts to the factorization $\beta(G)=\beta(A)\beta(B)$, which is known \cite{SekineZhang:DecFlow}. The remaining conjecture thus amounts to the identity
\begin{equation*}
    g''_G(0)=\beta(A)g''_B(0)+g''_A(0)\beta(B)
\end{equation*}
for the quadratic coefficient. For example, the graph $G$ in \cref{fig:3edge-cut} splits into $A$ the wheel with four spokes and the complete bipartite graph $B\cong K_{3,3}$. The polynomials are
\begin{align*}
    g_G(t) &= 7t^6+40t^5+92t^4+106t^3+61t^2+15t, \\
    g_A(t) &= t^4+4t^3+5t^2+3t, \\
    g_B(t) &=4t^4+12t^3+12t^2+5t,
\end{align*}
such that indeed, the coefficients of $t$ and $t^2$ satisfy $15=3\cdot 5$ and $61=3\cdot 12+5\cdot 5$.

The graph transformation in the following conjecture is a variant of the twist, but for edges instead of vertices:
Consider a graph $G_1$ with a minimal 4-edge cut $C=\{e_1,\ldots,e_4\}$, so that $G\setminus C=A\sqcup B$ has two connected components and $e_i=\{a_i,b_i\}$ for some vertices $a_i$ in $A$ and $b_i$ in $B$. Then crossing the edges of $C$ according to a double transposition, $e_1$ with $e_2$ and $e_3$ with $e_4$, as illustrated in \cref{fig:edge-twist}, yields a potentially different graph
\begin{equation*}
    G_2=(G\setminus C)\cup\left\{\{a_1,b_2\},\{a_2,b_1\},\{a_3,b_4\},\{a_4,b_3\}\right\}.
\end{equation*}
We call such a graph $G_2$ a \emph{4-edge twist} of $G_1$. Note that a single graph $G_1$ can have many different 4-edge cuts $C$, and at any fixed 4-edge cut, we can produce up to 3 potentially non-isomorphic twists (by permuting the edge labels---there are 3 double transpositions). If all $a_i$ are distinct, then the 4-edge twist is also a (4-vertex-)twist as in \cref{fig:3sum-twist}.
\begin{conjecture}\label{con:edgetwist}
    For every 4-edge twist $G_2$ of a graph $G_1$, $g_{G_1}(t)\equiv g_{G_2}(t) \mod t^3\Z[t]$.
\end{conjecture}
\begin{figure}
    \centering
    $\Graph[0.44]{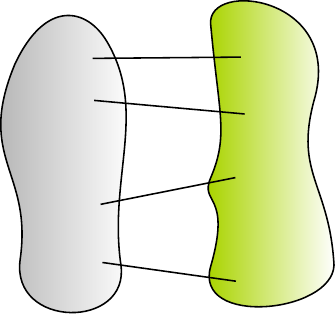} \mapsto \Graph[0.44]{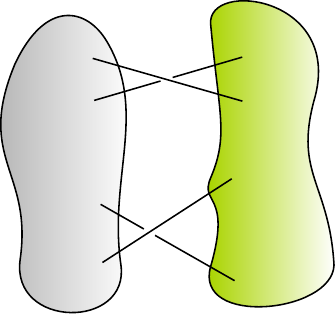}$\qquad
    $G_1=\Graph{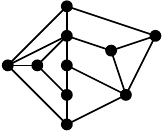} \quad G_2=\Graph{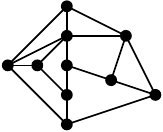}$
    \caption{The general structure of a 4-edge twist (left) and a concrete example (right). In \texttt{graph6} format, $G_1=\texttt{I?@TPrK\{O}$ and $G_2=\texttt{I?ClaZOwW}$. This twist $G_2$ arises from labelling the cut edges as $e_1,e_3,e_4,e_2$ going from top to bottom.}%
    \label{fig:edge-twist}%
\end{figure}
For the linear coefficient, this amounts to the identity $\beta(G_1)=\beta(G_2)$, which in fact follows from \cite{Kochol:DecFlow} (see \Cref{lem:flow-twist}). The remaining conjecture is the identity of the quadratic coefficients, so that $g''_{G_1}(0)=g''_{G_2}(0)$. For example, the graphs in \cref{fig:edge-twist} have
\begin{align*}
    g_{G_1}(t)&=5t^8+51t^7+212t^6+474t^5+621t^4+480t^3+204t^2+38t, \\
    g_{G_2}(t)&=4t^8+43t^7+190t^6+446t^5+604t^4+476t^3+204t^2+38t.
\end{align*}

We close this series of observations with a new relation between $g_M(t)$ and the Tutte polynomial. Recall that for a matroid without loops or coloops, $g'_M(0)=\beta(M)=t_{1,0}(M)=t_{0,1}(M)$ is equal to the linear coefficients of the Tutte polynomial
\begin{equation*}
    \Tutte{M}(x,y)=\sum_{i,j} t_{i,j}(M) x^i y^j=\sum_{A\subseteq M}(x-1)^{\rk(M)-\rk(A)}(y-1)^{\loops(A)}.
\end{equation*}
Our data shows that there is no further linear relation between coefficients of $g_M$ and $\Tutte{M}$ that holds for all graphs (let alone all matroids). But if we specialize further, namely to cubic (i.e.\ 3-regular) graphs, then we do observe a new relation:
\begin{conjecture}\label{con:g-Tutte}
    For a connected cubic graph $G$ with $n$ vertices, we have the identity
    \begin{equation*}
        g_G''(0)= 2n t_{0,1}(G)-4t_{0,2}(G).
    \end{equation*}
\end{conjecture}
Thus, for cubic graphs, the flow polynomial $\Flow{G}(q)=(-1)^{\loops(G)}\Tutte{G}(0,1-q)$ determines the value of the second derivative $g''_G(0)$. To our knowledge, this is the first relation between the Tutte- and $g$-polynomials that goes beyond \eqref{eq:gm10}. This gives a partial answer---albeit only for cubic graphs---to a question from \cite[Remark~3.7]{Ferroni:SchubertDelannoySpeyer}.

Aside from the structural properties above, we discuss in \Cref{sec:data} also closed formulas for the Speyer polynomials (or just the invariants $\FP_2$) of several families of graphs, including the complete graphs (braid matroids). We suggest candidates for graphs that extremize $\FP_2$, and we observe currently unexplained factorizations and zeroes of $g_M(t)$.
In \Cref{sec:Feynman}, we discuss the connection between $\FP_2(G)$ and Feynman integrals in particle physics---which was, in fact, the origin and driver of the discoveries and research presented in this paper. That discussion touches on other invariants of graphs that behave similar to, and perhaps determine, $\FP_2(G)$.


\subsection{Outline}
The three sections of this paper are independent of each other: \Cref{sec:graphs} covers the graph theory (\Cref{thm:betark=bic,thm:betark-k=cc}), \Cref{sec:speyerg} covers the matroid theory (\Cref{prop:N1asbeta} and \Cref{alg:gRec}), and \Cref{sec:data} summarizes observations from the calculated data, which provide the support for the conjectures stated above.

\subsection{Acknowledgments}

I thank Nicholas Proudfoot for discussions about matroid invariants, Luis Ferroni and David Speyer for their beautiful papers and interest in this work, and Francis Brown and Simone Hu for bringing Steinitz's theorem and \cite{BergetFink:ActiPair} to my attention.

Erik Panzer is funded as a Royal Society University Research Fellow through grant {URF{\textbackslash}R1{\textbackslash}201473}.

For the purpose of Open Access, the author has applied a CC BY public copyright licence to any Author Accepted Manuscript (AAM) version arising from this submission.

\section{Graphs, Elser's identity, and Crapo's invariant}\label{sec:graphs}
We assume throughout that graphs $G=(V,E)$ are finite and undirected. We do allow multiedges (several edges with the same endpoints) and self-loops (edges with both ends at the same vertex).
A subset of vertices $S\subseteq V$ is a \emph{vertex cover} if each edge of $G$ has at least one endpoint in $S$. 
The following terminology was introduced in \cite{Elser:ClusterSAW}:
\begin{definition}
    A \emph{nucleus} of $G=(V,E)$ is a connected non-empty subgraph $N=(V',E')$ such that $V'=V(N)$ is a vertex cover of $G$.
    The set of all nuclei of $G$ is denoted $\nuc(G)$.
\end{definition}
We use the following result of Elser \cite[Lemma~1]{Elser:ClusterSAW}. The proof given below follows \cite{Grinberg:Elser}.
\begin{lemma}\label{lem:Elser}
    Let $G$ be a graph with at least one edge, and let $v$ denote any vertex. Then
    \begin{equation*}
        \sum_{\substack{N\in\nuc(G)\\ v \in N}} (-1)^{\abs{E(N)}} = 0.
    \end{equation*}
\end{lemma}
\begin{proof}
    Any subset of edges $A\subseteq E(G)$ defines a graph on the set $V(G)$ of all vertices. In this graph, let $A_v$ denote the connected component which contains $v$. If $A_v$ is a nucleus of $G$, then it contains all of $A$: any edge in $A\setminus A_v$ would have to have an endpoint in $A_v$, thus be connected to $v$ in $A$ and hence belong to $A_v$. The nuclei $N$ containing $v$ thus arise as $A_v$ for exactly one choice of $A$, namely $A=E(N)$. Therefore,
    \begin{equation*}
    0 = \sum_{A\subseteq E(G)} (-1)^{\abs{A}} = \sum_{\substack{N\in\nuc(G)\\ v \in N}} (-1)^{\abs{E(N)}} + \sum_{\substack{A\subseteq E(G) \\ A_v\notin \nuc(G)}} (-1)^{\abs{A}}
    \end{equation*}
    and it suffices to show the vanishing of the rightmost sum over $\mathscr{A}=\{A\colon A_v\notin\nuc(G)\}$. This is easily achieved by a fixed-point free involution on $\mathscr{A}$: Fix any total order of the edges of $G$. Then for any $A\in\mathscr{A}$, let $e$ denote the smallest edge with no endpoint in $A_v$ (there is at least one such edge, since $A_v$ is not a nucleus), and set $A'=A-e$ if $e\in A$ and $A'=A+e$ otherwise. This keeps the component $(A')_v=A_v$ containing $v$ unchanged, and thus $A''=A$ with $\abs{A'}=\abs{A}\pm 1$ such that $(-1)^{\abs{A'}}=-(-1)^{\abs{A}}$.    
\end{proof}
Summing over all vertices $v$, \cref{lem:Elser} gives rise to \cite[Theorem~2]{Elser:ClusterSAW}: In every graph $G$ with at least one edge,
\begin{equation}\label{eq:Elser1}
    \sum_{N\in\nuc(G)} (-1)^{\abs{E(N)}} \abs{V(N)} = 0.
\end{equation}
To apply this identity, we relate alternating sums over all edge subsets to sums over nuclei, as follows. For any graph $G$, let $\pi_0(G)=\{C_1,\ldots,C_k\}$ denote its connected components in the sense of graph theory (not to be confused with the connected components in the sense of matroid theory; the latter partition the edge set of $G$ into 2-connected blocks).
\begin{lemma}\label{lem:additive}
    Let $F$ be any function on graphs (with values in any abelian group) that is additive over connected components, that is $F(G)=\sum_{C\in\pi_0(G)} F(C)$ for every graph $G$. Then, for every graph $G$, we have the identity
    \begin{equation*}
        \sum_{A\subseteq E(G)} (-1)^{\abs{A}} F(A)
        = \sum_{N\in \nuc(G)} (-1)^{\abs{E(N)}} F(N),
    \end{equation*}
    where on the left-hand side the edge set $A$ is viewed in $F(A)$ as a subgraph of $G$ containing all vertices $V(A)=V(G)$.
\end{lemma}
\begin{proof}
    Expanding $F(A)=\sum_{C\in\pi_0(A)}F(C)$ and swapping the sums over $C$ and $A$, setting $B=A\setminus E(C)$ the left-hand side can be written as
    \begin{equation*}
        \sum_{C\subseteq G}(-1)^{\abs{E(C)}} F(C)\sum_{B\subseteq E(G\setminus V(C))} (-1)^{\abs{B}}
    \end{equation*}
    where $C$ ranges over all connected subgraphs of $G$ and $B$ is a subset of edges with no endpoints in $V(C)$. If $C$ is a nucleus, such edges do not exist and the sum over $B$ is equal to $(-1)^{\abs{\emptyset}}=1$. If $C$ is not a nucleus, the sum over $B$ is zero.
\end{proof}
This simple observation gives an independent proof of \eqref{eq:Elser1}: taking $F(G)=\abs{V(G)}$ equates \eqref{eq:Elser1} with $V(A)=V(G)$ times $\sum_A(-1)^{\abs{A}}=0$.
\begin{corollary}\label{lem:beta=els0}
    For every graph with at least one edge, Crapo's $\beta$ invariant is
    \begin{equation*}
        \beta(G)=-(-1)^{\rk(G)} \sum_{N\in\nuc(G)} (-1)^{\abs{E(N)}}.
    \end{equation*}
\end{corollary}
\begin{proof}
    Apply \cref{lem:additive} to the additive function $F(G)=\rk(G)$. Then \eqref{eq:Crapo} becomes
    \begin{equation}\label{eq:beta-nuc-rk}
        \beta(G)=(-1)^{\rk(G)} \sum_{N\in\nuc(G)} (-1)^{\abs{E(N)}} (\abs{V(N)}-1),
    \end{equation}
    since for a connected graph $N$, $\rk(N)=\abs{V(N)}-1$. The claim follows from \eqref{eq:Elser1}.
\end{proof}
In the notation of \cite{DBHLMNVW:Elser}, minus the expression on the right-hand side of \cref{lem:beta=els0} is called the \emph{zeroth Elser number} $\els_0(G)$, and so our result says $\els_0(G)=-\beta(G)$. Since $\beta(G)$ is the coefficient of $x$ in the Tutte polynomial
\begin{equation}\label{eq:Tutte}
    \Tutte{G}(x,y)=\sum_{A\subseteq E(G)} (x-1)^{\rk(G)-\rk(A)} (y-1)^{\loops(A)},
\end{equation}
it inherits the contraction-deletion recursion $\beta(G)=\beta(G/e)+\beta(G\setminus e)$ whenever $e$ is neither a bridge nor a self-loop. This explains the $k=0$ case of \cite[Theorem~8.1]{DBHLMNVW:Elser}, answers the questions about $\els_0(G)$ on \cite[p.~26]{DBHLMNVW:Elser}, and reduces \cite[Theorem~7.5]{DBHLMNVW:Elser} to \cite[Theorem~II]{Crapo:HigherInvariant}.

We now state and prove the main result of this section (\cref{thm:betark=bic} from the introduction). A graph is biconnected if it is connected and every pair of edges is contained in a cycle. A \emph{block} is a maximal biconnected subgraph. We call a block \emph{trivial} if it has no edges (the trivial blocks of a graph are precisely its isolated vertices). We denote $\cc(G)$ the number of non-trivial blocks of $G$, with the convention that each self-loop constitutes a (non-trivial) block in itself. This ensures that $\cc(G)$ is equal to the number $\cc(M)$ of connected components of the cycle matroid $M$ of $G$.
\begin{theorem}\label{thm:betark=cc}
    For every graph $G$ that has at least two edges,
    \begin{equation}\label{eq:betark=cc}
        \beta(G)\rk(G) = (-1)^{\rk(G)}\sum_{A\subseteq E(G)} (-1)^{\abs{A}} \cc(A).
    \end{equation}
\end{theorem}
\begin{proof}
    Firstly, we may assume that $G$ has no self-loops. For if $G$ has a self loop $e$, then $\cc(A)=\cc(A\setminus e)+1$ whenever $e\in A$, wherefore the sum collapses to
    \begin{equation*}
    \sum_{A\subseteq E(G)\setminus\{e\}} (-1)^{\abs{A}} \left(\cc(A)-[\cc(A)+1]\right)
    =-\sum_{A\subseteq E(G)\setminus\{e\}} (-1)^{\abs{A}} = 0
    \end{equation*}
    since $\abs{E(G)\setminus\{e\}}\geq 1$. The presence of a self-loop also forces $\beta(G)=0$ to vanish \cite{Crapo:HigherInvariant}, hence \eqref{eq:betark=cc} holds. Henceforth, we assume that $G$ has no self-loops.

    Since blocks are connected subgraphs, the map $A\mapsto \cc(A)$ is additive over the connected components of the subgraph $A$. Hence we can apply \cref{lem:additive} and restrict the sum in \eqref{eq:betark=cc} to edge sets $A=E(N)$ of nuclei $N\in\nuc(G)$ only. We may furthermore assume that $N$ has at least one edge, since otherwise it is a trivial block with $\cc(N)=0$.

    Thus consider a connected subgraph $N$ of $G$ which has at least one edge. Such $N$ has no isolated vertices, so all of its blocks are non-trivial and contain at least two vertices. The number of blocks that contain a given vertex $v$ is therefore $\abs{\pi_0(N\setminus v)}$, and $v$ is a cut vertex of $N$ precisely when this number exceeds one. The total number of blocks can be counted in terms of theses degrees of the cut vertices in the block-cut tree \cite{Harary:Elementary}:
    \begin{equation*}
        \cc(N)=1+\sum_{v\in V(N)} \left(\abs{\pi_0(N\setminus v)}-1\right) = 1-\abs{V(N)}+\sum_{v\in V(N)} \sum_{C\in \pi_0(N\setminus v)} 1.
    \end{equation*}
    Summing over all nuclei, using \eqref{eq:beta-nuc-rk} we have rewritten the right-hand side of \eqref{eq:betark=cc} as
    \begin{equation*}
        -\beta(G)+(-1)^{\rk(G)}\sum_{N\in\nuc(G)} \sum_{v\in V(N)} \sum_{C\in \pi_0(N\setminus v)} (-1)^{\abs{E(N)}}.
    \end{equation*}
    For any component $C$ of $N\setminus v$, the complement $B\defas N\setminus V(C)$ is, by construction, a connected graph containing $v$. Let $[v,C]\subseteq E(G)$ denote the edges of $G$ with one endpoint at $v$ and the other endpoint in $V(C)$, so that
    \begin{equation*}
        E(N)=E(C)\sqcup E(B)\sqcup ([v,C]\cap E(N)).
    \end{equation*}
    By construction, $N$ must contain at least one edge in $[v,C]$, for otherwise $C$ would not be connected (in $N$) to $v$.
    Also note that $N$ is a nucleus of $G$ if and only if $B$ is a nucleus of $G\setminus V(C)$: there are no constraints from edges with an endpoint in $V(C)$ as those are absent (deleted) in $G\setminus V(C)$ and trivially covered by $V(N)=V(C)\sqcup V(B)$; the edges remaining in $G\setminus V(C)$ are covered by $V(N)$ precisely if they have an endpoint in $V(B)$.

    With these observations, we can further rewrite the right-hand side of \eqref{eq:betark=cc} as
    \begin{equation*}
        -\beta(G)+(-1)^{\rk(G)}\sum_{v\in V(G)} \sum_{\substack{C\subseteq G\setminus v \\ \abs{\pi_0(C)}=1}} (-1)^{\abs{E(C)}} \sum_{\substack{H\subseteq [v,C]\\\abs{H}>0}} (-1)^{\abs{H}} \sum_{\substack{B\in\nuc(G\setminus V(C))\\ v \in B}} (-1)^{\abs{E(B)}}.
    \end{equation*}
    The sum over $H\defas [v,C]\cap E(N)$ is zero if $[v,C]=\emptyset$ and $\sum_{i=1}^k\binom{k}{i}(-1)^i=(1-1)^k-1=-1$ for $k=\abs{[v,C]}>0$. By \cref{lem:Elser}, the sum over $B$ gives zero, unless the graph $G\setminus V(C)$ has no edges. Hence only those terms remain where $C$ is itself a nucleus of $G$:
    \begin{equation*}
        (-1)^{\rk(G)}\sum_{A\subseteq E(G)} (-1)^{\abs{A}} \cc(A)
        =
        -\beta(G)-(-1)^{\rk(G)} \sum_{C\in\nuc(G)} (-1)^{\abs{E(C)}} \sum_{\substack{v\in V(G)\setminus V(C)  \\ [v,C]\neq\emptyset}} 1.
    \end{equation*}
    A vertex $v\notin V(C)$ is either an isolated vertex of $G$ (then $[v,C]=\emptyset$), or it has edges but then those must all be in $[v,C]\neq\emptyset$ since $V(C)$ is a vertex cover. Therefore, the sum over $v$ gives $\abs{V(G)}-\abs{V(C)}-i$ if we let $i$ denote the number of isolated vertices in $G$. By \eqref{eq:Elser1}, summing the term with $\abs{V(C)}$ over all nuclei $C$ gives zero, and for the remaining terms we can apply \cref{lem:beta=els0} to conclude that
    \begin{equation*}
        (-1)^{\rk(G)}\sum_{A\subseteq E(G)} (-1)^{\abs{A}} \cc(A)
        =
        \beta(G)\cdot\left( -1 + \abs{V(G)}-i \right).
    \end{equation*}
    Since $\rk(G)=\abs{V(G)}-\abs{\pi_0(G)}$, this concludes the proof whenever $G$ has exactly one connected component that contains edges, for then $\abs{\pi_0(G)}=1+i$. If $G$ has multiple connected components that contain edges, then the cycle matroid of $G$ is disconnected and $\beta(G)=0$ by \cite{Crapo:HigherInvariant}, so both sides of \eqref{eq:betark=cc} vanish.
\end{proof}
\begin{theorem}\label{thm:cc-AB}
    For any graphic or cographic matroid, the number of connected components is
    \begin{equation}\label{eq:cc-AB}
        \cc(M) = \abs{M}-\sum_{A\subseteq B\subseteq M} (-1)^{\abs{B}-\abs{A}} \loops(A) \rk(B).
    \end{equation}
\end{theorem}
\begin{proof}
    Rewriting the corank $\loops(A)=\abs{A}-\rk(A)$ and recognizing \eqref{eq:Crapo}, the sum over $A$ evaluates to
    \begin{equation*}
        \sum_{A\subseteq B} (-1)^{\abs{A}} \loops(A)
        = \sum_{A\subseteq B} (-1)^{\abs{A}} \abs{A}-(-1)^{\rk(B)}\beta(B)
        = -\delta_{\abs{B},1} - (-1)^{\rk(B)}\beta(B)
    \end{equation*}
    where $\delta_{\abs{B},1}=1$ if $\abs{B}=1$ and $\delta_{\abs{B},1}=0$ otherwise. For a single edge $B=\{e\}$ with $\rk(B)=1=\beta(B)$ (i.e.\ not a self-loop), this combination is zero (since $\loops(A)=0$). For a self-loop, the sum over $A$ gives $-1$; but this does not contribute to \eqref{eq:cc-AB} because $\rk(B)=0$ in this case. Therefore,
    \begin{equation}\label{eq:LARB=betaR}
        -\sum_{A\subseteq B\subseteq M} (-1)^{\abs{B}-\abs{A}} \loops(A) \rk(B)
        =\sum_{\substack{B\subseteq M \\ \abs{B}\neq 1}} (-1)^{\abs{B}} (-1)^{\rk(B)} \beta(B) \rk(B).
    \end{equation}
    Now assume that $M$ is graphic. Then so is every submatroid $B\subseteq M$ and we can invoke \cref{thm:betark=cc} for all $\abs{B}\neq 1$ (the theorem also applies to the trivial case $B=\emptyset$).
    For any single edge $\abs{B}=1$, self-loop or not, we have $\sum_{A\subseteq B} (-1)^{\abs{A}}\cc(A)=-1$. The sum over $\abs{B}\neq 1$ thus combines with the number of edges $\abs{M}$ on the right-hand side of \eqref{eq:cc-AB} to
    \begin{equation*}
        \sum_{B\subseteq M} (-1)^{\abs{B}} \left( \sum_{A\subseteq B} (-1)^{\abs{A}} \cc(A) \right)
        =\sum_{A\subseteq M} \cc(A) \sum_{B\subseteq M, B\supseteq A} (-1)^{\abs{B\setminus A}}
        = \cc(M)
    \end{equation*}
    by inclusion-exclusion. This proves \eqref{eq:cc-AB} for graphic matroids. Now suppose $M$ is cographic. Then the identity holds for its (graphic) dual matroid $M^\star$, which has the same $\cc(M)=\cc(M^\star)$ and $\abs{M}=\abs{M^\star}$. The ranks $r^\star$ and coranks $\loops^\star$ of subsets $A$ of the dual $M^\star$ relate to the coranks and ranks of the complementary sets $A^\sc=M\setminus A$ in $M$ as
    \begin{equation*}
        \loops^\star(A)=\rk(M)-\rk(A^\sc),\quad \rk^\star(B)=\loops(M)-\loops(B^\sc)
        .
    \end{equation*}
    Hence re-indexing the summation variables $(A,B)\mapsto (B^\sc,A^\sc)$, we have the identity
    \begin{align*}
        \abs{M}-\cc(M)
        =\abs{M^\star}-\cc(M^\star)
        &=\sum_{A\subseteq B\subseteq M^\star} (-1)^{\abs{B}-\abs{A}} \loops^\star(A) \rk^\star(B)  \\
        &=\sum_{A\subseteq B\subseteq M} (-1)^{\abs{B}-\abs{A}} (\rk(M)-\rk(B)) (\loops(M)-\loops(A)). 
    \end{align*}
    Expanding the summand, the parts with $\rk(M)$ or $\loops(M)$ sum to zero:
    \begin{equation*}
        \sum_{A\subseteq M} \rk(M)(\loops(M)-\loops(A)) \sum_{B\subseteq M,B\supseteq A}(-1)^{\abs{B\setminus A}}
        =\rk(M)(\loops(M)-\loops(A))\Big|_{A=M} = 0
    \end{equation*}
    by inclusion-exclusion, and similarly the sum over $A$ in $\sum_{A\subseteq B\subseteq M} (-1)^{\abs{B}-\abs{A}} \rk(B)\loops(M)$ localizes to $B=\emptyset$ where $\rk(B)=0$. Therefore, only the sum of $(-1)^{\abs{B}-\abs{A}}\loops(A)\rk(B)$ remains, which concludes the proof that \eqref{eq:cc-AB} also holds for cographic matroids.
\end{proof}
This completes the proof of \cref{thm:betark=bic} as stated in the introduction:
\begin{corollary}
    The identity \eqref{eq:betark=cc} from \cref{thm:betark=cc} also holds for cographic matroids with at least two edges.
\end{corollary}
\begin{proof}
    The first part of the proof of \cref{thm:cc-AB} shows that the identity of \cref{thm:betark=cc} is equivalent, under inclusion-exclusion, to \cref{thm:cc-AB}.
\end{proof}
We now generalize the above relations to higher connectivity (\cref{thm:betark-k=cc} from the introduction). For a graph $G$ without self-loops, all edges have $\beta(e)=1=\rk(e)$. Hence, including the summands with $\abs{B}=1$ in \eqref{eq:LARB=betaR} adds precisely $\abs{E(G)}$, so that \eqref{eq:cc-AB} can be written as
\begin{equation*}
    \cc(G)=\sum_{A\subseteq E(G)} (-1)^{\loops(A)} \beta(A) \rk(A).
\end{equation*}
If $G$ is biconnected, then this sum is equal to $1$; otherwise $\cc_1(G)=1+\sum_v(\abs{\pi_0(G\setminus v)}-1)$ counts the 1-vertex cuts of $G$, weighted by the number of extra components their deletion creates. We generalize $\cc(G)=\cc_1(G)$ as follows:
\begin{definition}\label{def:cci}
    For any graph $G$ and any positive integer $i$, define the sum
    \begin{equation*}\label{eq:cci}
        \cc_i(G)\defas (-1)^{i-1}\sum_{A\subseteq E(G)} (-1)^{\loops(A)} \beta(A) \binom{\rk A}{i}
    \end{equation*}
    where the binomial coefficient is understood to vanish when $i>\rk(A)$.
\end{definition}
The integers $\cc_i(G)$ measure the connectivity of $G$. We call a graph $k$-connected if it has more than $k$ vertices, and deleting any set of less than $k$ vertices leaves a connected graph.
\begin{theorem}\label{thm:cck=S}
    For any $k\geq 1$ and any $k$-connected graph $G$ without self-loops, we have the identities $\cc_1(G)=\ldots=\cc_{k-1}(G)=1$ and
    \begin{equation}\label{eq:cck=S}
    \cc_k(G)=1+ \sum_{\substack{S\subset V(G)\\ \abs{S}=k}} \Big( \pi_0(G\setminus S)-1 \Big).
    \end{equation}
\end{theorem}
In other words, for a $k$-connected graph, the sum $\cc_k(G)$ from \cref{def:cci} counts $k$-vertex cuts, weighted by the number of extra components that their deletion creates.
\begin{proof}
    Let $V(A)\subseteq V(G)$ denote the vertices incident to at least one edge of $A$. Since $\beta(A)=0$ whenever $\cc(A)>1$, the summand in \cref{def:cci} is zero unless $A$ defines a biconnected  subgraph on $V(A)$. In this case, $\rk(A)=V(A)-1$, so that we can interpret $\binom{\rk(A)+1}{i}=\binom{\rk(A)}{i}+\binom{\rk(A)}{i-1}$ as choosing $i$ vertices out of $V(A)$:
    \begin{equation*}
        \cc_i(G)-\cc_{i-1}(G)
        = (-1)^{i-1}\sum_{A\subseteq E(G)} (-1)^{\loops(A)} \beta(A) \binom{\abs{V(A)}}{i}.
    \end{equation*}
    Expanding the binomial into a sum over all subsets $S$ of $i$ vertices, swapping the sums over $S$ and $A$, and expanding $\beta(A)$ according to \eqref{eq:Crapo}, yields
    \begin{equation*}
        \cc_i(G)-\cc_{i-1}(G)
        = (-1)^{i-1}\sum_{\substack{S\subset V(G) \\ \abs{S}=i}} \sum_{B\subseteq E(G)} \rk(B) \sum_{\substack{A\supseteq B\\ V(A)\supseteq S}} (-1)^{\abs{A\setminus B}}.
    \end{equation*}
    For fixed $S$ and $B$, consider the partition $S=X\sqcup Y$ with $X=S\cap V(B)$. Then $A\supseteq B$ reduces the condition $V(A)\supseteq S$ down to $V(A)\supseteq Y$. Any edge $e\in E(G)\setminus B$ with neither endpoint in $Y$ may be added or removed from $A$ without affecting these conditions; however this switch swaps the sign $(-1)^{\abs{B\setminus A}}$ and forces the sum over $A$ to vanish. Therefore, the sum over $B$ receives contributions only from induced subgraphs, where $B=E(G\setminus Y)$ consists of \emph{all} edges with no endpoint in $Y$. Denoting by $E_Y\subseteq E(G)$ those edges with at least one endpoint in $Y$, and relabeling $A\setminus B$ as $A$, we have
    \begin{equation*}
        \cc_i(G)-\cc_{i-1}(G)
        = (-1)^{i-1}\sum_{\substack{S\subset V(G) \\ \abs{S}=i}} \sum_{Y\subseteq S} \rk(G\setminus Y) \sum_{\substack{A\subseteq E_Y\\ V(A)\supseteq Y}} (-1)^{\abs{A}}.
    \end{equation*}
    In fact, in order to have $V(B)=X=S\setminus Y$, the sum should only be over such partitions of $S$ such that each vertex $v\in X$ has at least one neighbour in $V(G)\setminus Y$. However, this is automatically the case for all $i\leq k$, because the existence of a vertex $x\in X\neq\emptyset$ with all neighbours in $Y$ would imply that $Y$ is a vertex cut of size $\abs{Y}=i-\abs{X}<i\leq k$, contradicting the assumption that $G$ is $k$-connected.

    Similarly, if any vertex $v\in Y$ had all its neighbours contained in $Y$, then $Y\setminus\{v\}$ would be a vertex cut of size less than $i\leq k$. Thus, we may assume that each $v\in Y$ has at least one edge to a vertex in $V(G)\setminus Y$. Pick one such edge $e_v$ (by construction, $e\in E_Y$), then the summation domain $\mathcal{A}=\{A\colon A\subseteq E_Y, V(A)\supseteq Y\}=\mathcal{A}_1\sqcup\mathcal{A}_2\sqcup\mathcal{A}_3$ splits into $\mathcal{A}_1=\{e_v\notin A\}$, $\mathcal{A}_2=\{e_v\in A, A\setminus\{e_v\}\in\mathcal{A}\}$, and $\mathcal{A}_3=\{e_v\in A, A\setminus\{e_v\}\notin\mathcal{A}\}$. The bijection $A\mapsto A\sqcup\{e_v\}$ between $\mathcal{A}_1$ and $\mathcal{A}_2$ shows that the corresponding contributions to the sum of $(-1)^{\abs{A}}$ cancel. Since $v$ is the only endpoint of $e_v$ in $Y$, in the remaining sum over $\mathcal{A}_3$ we have $V(A\setminus\{e_v\})=V(A)\setminus \{v\}\supseteq Y\setminus\{v\}$. Thus, by induction,
    \begin{equation*}
        \sum_{\substack{A\subseteq E_Y\\ V(A)\supseteq Y}} (-1)^{\abs{A}}
        =-\sum_{\substack{A\subseteq E_{Y\setminus v}\\ V(A)\supseteq Y\setminus\{v\}}} (-1)^{\abs{A}}
        =\ldots
        =(-1)^{\abs{Y}}.
    \end{equation*}
    Combining this result for the sum over $A$ with the identity $\rk(B)=\rk(G\setminus Y)=\abs{V(G\setminus Y)}-\abs{\pi_0(G\setminus Y)}$, we have shown
    \begin{equation*}
        \cc_i(G)-\cc_{i-1}(G)
        = (-1)^{i-1}\sum_{\substack{S\subset V(G) \\ \abs{S}=i}} \sum_{Y\subseteq S} (\abs{V(G)}-\abs{Y}-\abs{\pi_0(G\setminus Y)}) (-1)^{\abs{Y}}.
    \end{equation*}
    Because $G$ has no vertex cuts of size less than $k\geq i$, we know that $\abs{\pi_0(G\setminus Y)}=1$ unless perhaps for the summand with $Y=S$. We can therefore write the right-hand side as
    \begin{equation*}
        (-1)^{i-1}\sum_{\substack{S\subset V(G) \\ \abs{S}=i}} \sum_{Y\subseteq S} (\abs{V(G)}-\abs{Y}-1) (-1)^{\abs{Y}}
        +\sum_{\substack{S\subset V(G) \\ \abs{S}=i}} (\abs{\pi_0(G\setminus S)}-1).
    \end{equation*}
    The elementary sum over $Y$ depends not on the set $S$, but only its size $i$:
    \begin{equation*}
        \sum_{Y\subseteq S} (\abs{V(G)}-\abs{Y}-1) (-1)^{\abs{Y}}
        =\sum_{r=0}^i \binom{i}{r} (\rk(G)-r) (-1)^r
        =\begin{cases}
            1 & \text{if $i=1$,} \\
            0 & \text{if $i\geq 2$.} \\
        \end{cases}
    \end{equation*}
    In conclusion, for all $1\leq i \leq k$ we have established the identity
    \begin{equation}\label{eq:cci-rec}
        \cc_i(G)
        =\cc_{i-1}(G)
        +\sum_{\substack{S\subset V(G) \\ \abs{S}=i}} (\abs{\pi_0(G\setminus S)}-1)
        +\begin{cases}
            \abs{V(G)} & \text{if $i=1$,} \\
            0 & \text{if $i\geq 2$.} \\
        \end{cases}
    \end{equation}
    It follows immediately that $\cc_{k-1}(G)=\cc_{k-2}(G)=\ldots=\cc_2(G)=\cc_1(G)$, which reduces \cref{thm:cck=S} for the cases $k\geq 2$ to merely the claim that $\cc_1(G)=1$, which in turn follows from \eqref{eq:cck=S} in the case $k=1$. So we are only left having to prove \eqref{eq:cck=S} for $k=1$, which is precisely the relation between the number $\cc_1(G)=\cc(G)$ of blocks and degrees of vertices in the block-cut tree \cite{Harary:Elementary} that we already used to prove \cref{thm:betark=cc}.

    In fact, we can use \eqref{eq:cci-rec} to give an independent proof even of this base case $k=1$, without referring to \cite{Harary:Elementary}: Setting $i=1$ in \eqref{eq:cci-rec} yields
    \begin{equation*}
        \cc_1(G)=
        \cc_0(G)
        +\abs{V(G)}
        +\sum_{v\in V(G)} (\abs{\pi_0(G\setminus \{v\})}-1)
    \end{equation*}
    so that it only remains to note that for a (1-)connected graph,
    \begin{equation*}
        \cc_0(G)=-\sum_{A\subseteq E(G)} (-1)^{\loops(A)}\beta(A)
        =-\sum_{B\subseteq A\subseteq E(G)} \rk(B) (-1)^{\abs{A\setminus B}}
        =-\rk (G)=1-\abs{V(G)}
    \end{equation*}
    where the second sum over $A$ (for all $A\supseteq B$, with fixed $B$) gives zero unless $B=E(G)$.
\end{proof}

\section{Speyer's invariant from Ferroni's algorithm}\label{sec:speyerg}

We use a combinatorial characterization of the polynomial $g_M(t)$, which was proposed in \cite{Speyer:MatroidKtheory} and confirmed in \cite[Theorem~9.15]{FerroniSchroeter:ValInv} and \cite[Definition~1.2]{Ferroni:SchubertDelannoySpeyer}. Namely, Speyer's polynomial is the unique \emph{covaluative} matroid invariant $M\mapsto g_M(t)\in\Z[t]$ such that:
\begin{enumerate}
    \item If $M$ has a loop or coloop, then $g_M(t)=0$.
    \item If $M$ is series-parallel, then $g_M(t)=t$.
    \item If $M=M_1\oplus M_2$ is a direct sum, then $g_M(t)=g_{M_1}(t)\cdot g_{M_2}(t)$.
\end{enumerate}
A \emph{loop} of $M$ is an edge (element of the ground set) of rank zero; in a graphic matroid this is an edge with both endpoints at the same vertex (``self-loop''). A \emph{coloop} is an edge $e$ with $\rk(M\setminus e)=\rk(M)-1$, i.e.\ an edge contained in \emph{every} basis---also called a \emph{bridge} in a graphic matroid. We call a matroid \emph{series-parallel} if it can be obtained from the 2-cycle $M\cong\UM{2}{1}$ by a sequence of series extensions $M\mapsto M\oplus_2 \UM{3}{2}$ (subdivide an edge) and parallel extensions $M\mapsto M\oplus_2 \UM{3}{1}$ (add an edge parallel to an existing edge).

The covaluative property states that $g_M$ can be computed from any decomposition of the matroid polytope of $M$ into other matroid polytopes; see \cite[\S2.3]{Ferroni:SchubertDelannoySpeyer}. The algorithm for computing $g_M$ elaborated in \cite{Ferroni:SchubertDelannoySpeyer} exploits this in two steps:
\begin{enumerate}
    \item Decompose $M$ into \emph{Schubert matroids}, following \cite{Hampe:IntersectionRing} using the lattice of cyclic flats.
    \item Decompose Schubert matroids into direct sums of series-parallel matroids, using a description in terms of lattice paths.
\end{enumerate}
In the sequel, we discuss both of these steps, in order to develop formulas that are explicit and efficient to evaluate. We apply these to prove \Cref{prop:N1asbeta} from the introduction, and in \Cref{sec:algrec} we present our algorithm for efficient computation of $g_M(t)$.

\subsection{Cyclic flats, chains, and the M\"obius function}
A subset $A\subseteq M$ of a matroid is called a \emph{flat} if $\rk(A\cup\{e\})=\rk(A)+1$ for every edge $e\notin A$. Dually, $A$ is called \emph{cyclic} if $\rk(A\setminus\{e\})=\rk(A)$ for every $e\in A$. Any given $A$ contains a unique largest cyclic subset $\cyc(A)$, and any given $A$ is contained in a unique smallest flat $\cl(A)$.
In the case of a graphic matroid, a subset $A\subseteq E(G)$ is cyclic if each $e\in A$ is contained in a cycle that is contained in $A$, and $A$ is a flat if all edges $e\notin A$ have their endpoints in two different connected components of $A$ (that is, each connected component of $A$ is an \emph{induced} subgraph of $G$). The operation $A\mapsto\cyc(A)$ removes all bridges from $A$, and the closure $A\mapsto\cl(A)$ adds all edges $e\in A$ whose endpoints are in the same connected component of $A$.

The lattice $\LCF$ of cyclic flats as studied in \cite{BoninMier:CyclicFlats} consists of all $A\subseteq M$ that are simultaneously cyclic and flat. Cyclic flats form a poset under the inclusion relation $A\leq B\Leftrightarrow A\subseteq B$, and in fact they form a \emph{lattice}: every pair $A,B\in\LCF$ has a \emph{join} (least upper bound) $A\vee B \in \LCF$ and a \emph{meet} (greatest lower bound) $A\wedge B\in \LCF$. Concretely, these are
\begin{equation*}
    A\vee B=\cl(A\cup B)
    \quad\text{and}\quad
    A\wedge B = \cyc(A\cap B)
    .
\end{equation*}
The minimum $0_\LCF=\cl(\emptyset)$ of $\LCF$ is the set of all loops and the maximum $1_\LCF=\cyc(M)$ is the complement of the coloops. Since $g_M(t)$ is only non-zero for matroids without loops or coloops, in fact we will have $0_\LCF=\emptyset$ and $1_\LCF=M$. \Cref{fig:cyclicflats} shows the Hasse diagrams of the lattices of cyclic flats of several graphs.
\begin{figure}\centering
    $\xymatrix@R=4.3mm@C=-1mm{
    & & & \Graph[0.3]{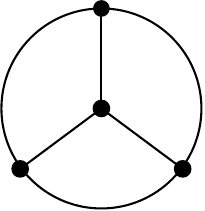} & & & \\
    \Graph[0.3]{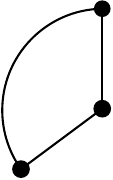} \ar@{-}[urrr] & \quad & \Graph[0.3]{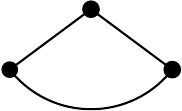} \ar@{-}[ur] & & \Graph[0.3]{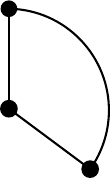} \ar@{-}[ul] & \quad & \Graph[0.3]{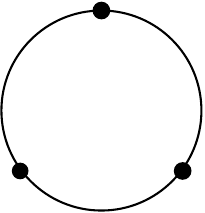} \ar@{-}[ulll] \\
    & & & \emptyset \ar@{-}[ulll] \ar@{-}[ul] \ar@{-}[ur] \ar@{-}[urrr] & & &
    }$
    \qquad
    $\xymatrix@R=3mm@C=4mm{
    & & \Graph[0.3]{ws4A} & & \\
    \Graph[0.3]{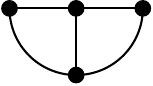} \ar@{-}[urr] & \Graph[0.3]{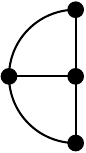} \ar@{-}[ur] & \Graph[0.3]{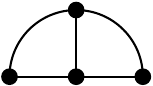} \ar@{-}[u] & \Graph[0.3]{ws4A122} \ar@{-}[ul] & \Graph[0.3]{ws4A4o} \ar@{-}[ull] \\
    \Graph[0.3]{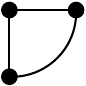} \ar@{-}[u] \ar@{-}[urrr]!/d 4pt/ & \Graph[0.3]{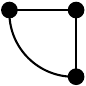} \ar@{-}[ul] \ar@{-}[u] & \Graph[0.3]{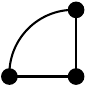} \ar@{-}[ul] \ar@{-}[u] & \Graph[0.3]{ws4A3} \ar@{-}[ul] \ar@{-}[u] &  \\
    & & \emptyset \ar@{-}[ull] \ar@{-}[ul] \ar@{-}[u] \ar@{-}[ur] \ar@{-}@/_1.5pc/[uurr] & &
    }$
    \caption{The Hasse diagrams of the lattices $\LCF$ of cyclic flats for the graphic matroids of the wheel graphs with three and four spokes.}%
    \label{fig:cyclicflats}%
\end{figure}
\begin{definition}
    A chain (of length $r$) in a poset $\mathcal{L}$ is a sequence $C_0<C_1<\ldots<C_r$ (denoted $C_\bullet$) of elements $C_i\in\mathcal{L}$ that are totally ordered by $\mathcal{L}$. A chain $D_\bullet$ is a \emph{refinement} of a chain $C_\bullet$ if each $C_i$ appears somewhere in $D_\bullet$. Given a lattice $\mathcal{L}$, the associated \emph{chain lattice} $\chains(\mathcal{L})$ consists of all chains $C_\bullet$ starting with $C_0=0_{\mathcal{L}}$ and ending in $C_r=1_{\mathcal{L}}$, ordered by refinement, together with an additional element $1_{\chains}$ declared to be the maximum.
\end{definition}
The minimum of the chain lattice is the chain $0_{\chains}=(0_{\mathcal{L}}<1_{\mathcal{L}})$ of length one. So except for the minimum $0_{\chains}$ and the maximum $1_{\chains}$, elements of the chain lattice are in bijection with the \emph{non-empty} chains of $\mathcal{L}^\circ\defas\mathcal{L}\setminus\{0_{\mathcal{L}},1_{\mathcal{L}}\}$, by dropping $C_0=0_{\mathcal{L}}$ and $C_r=1_{\mathcal{L}}$:
\begin{equation*}
    \chains(\mathcal{L})\setminus\{0_{\chains},1_{\chains}\}
    \cong \{ (C_i)_{1\leq i\leq k}\colon 0_{\mathcal{L}}<C_1<\ldots<C_k<1_{\mathcal{L}}\}.
\end{equation*}
This collection of subsets $\{C_1,\ldots,C_k\}$ of $\mathcal{L}^\circ$ forms a simplicial complex, called the \emph{order complex} of $\mathcal{L}$. In the notation of
\cite[Chapter~3]{Stanley:EC1}, 
the order complex $\Delta(\mathcal{L}^\circ)\cong\chains(\mathcal{L})\setminus\{1_{\chains}\}$ also contains the empty simplex of dimension $-1$ corresponding to the chain $0_{\chains}$.
\begin{proposition}[{\cite[Proposition~3.8.6]{Stanley:EC1}}]\label{prop:chi=moebius}
    The M\"{o}bius function of a lattice $\mathcal{L}$ computes the reduced Euler characteristic of the order complex of the poset $\mathcal{L}^\circ=\mathcal{L}\setminus\{0_{\mathcal{L}},1_{\mathcal{L}}\}$:
    \begin{equation*}
        \widetilde{\chi}(\Delta(\mathcal{L}^\circ))=\chi(\Delta(\mathcal{L}^\circ))-1=\mu_{\mathcal{L}}(0_{\mathcal{L}},1_{\mathcal{L}}).
    \end{equation*}
\end{proposition}
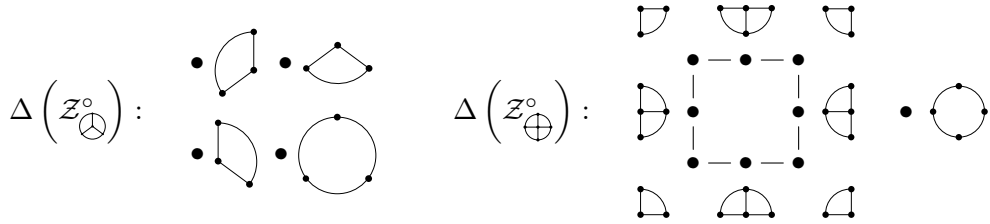
\begin{figure}
    \centering
    $\Delta\left(\LCF_{\Graph[0.1]{w3A}}^\circ\right)\colon\quad
    \vcenter{\xymatrix@=0mm{
    \bullet\ \Graph[0.3]{w3A3}  & \bullet\ \Graph[0.3]{w3A3c} \\
    \bullet\ \Graph[0.3]{w3A3b} & \bullet\ \Graph[0.3]{w3A3d} \\
    }}
    $
    \qquad
    $\Delta\left(\LCF_{\Graph[0.15]{ws4A}}^\circ\right)\colon\quad
    \vcenter{\xymatrix@=0mm{
        \Graph[0.3]{ws4A3b} & & \Graph[0.3]{ws4A122d} & & \Graph[0.3]{ws4A3c} \\
        & \bullet \ar@{-}[r] & \bullet \ar@{-}[r] & \bullet\ar@{-}[d] & \\
        \Graph[0.3]{ws4A122} & \bullet \ar@{-}[u] & & \bullet \ar@{-}[d] & \Graph[0.3]{ws4A122b} \\
        & \bullet \ar@{-}[u] & \bullet \ar@{-}[l] & \bullet \ar@{-}[l] & \\
        \Graph[0.3]{ws4A3}  & & \Graph[0.3]{ws4A122c} & & \Graph[0.3]{ws4A3d} \\
    }}
    \quad
    \bullet\ \Graph[0.3]{ws4A4o}
    $
    \caption{The order complex of the lattice of cyclic flats of the cycle matroid of the wheel graphs with three spokes (left) and four spokes (right).}%
    \label{fig:order-complex}%
\end{figure}
\begin{example}\label{ex:ordercomplex}
    The order complexes of the cyclic flats of the wheel graphs (see \cref{fig:cyclicflats}) with three and four spokes are shown in \cref{fig:order-complex}. For three spokes, the order complex consists of four points, thus
    $\mu_{\LCF}(\emptyset,\Graph[0.14]{w3A})=\widetilde{\chi}(\Delta(\LCF_M^\circ))=4-1=3$.
    With four spokes, the order complex consists of an octagon and one extra point, so
    $\mu_{\LCF}(\emptyset,\Graph[0.2]{ws4A})=\widetilde{\chi}(\Delta(\LCF_M^\circ))=-1+2-1=0$.
\end{example}
More generally, the M\"obius function $(x,y)\mapsto \mu_{\mathcal{L}}(x,y)\in\Z$ of any finite poset $\mathcal{L}$ can be extended to all arguments $x,y\in \mathcal{L}$ such that $x\leq y$. It fulfils $\mu_{\mathcal{L}}(x,x)=1$ for all $x\in\mathcal{L}$ and can be computed as $\mu_{\mathcal{L}}(x,z)=-\sum_{x<y\leq z} \mu_{\mathcal{L}}(y,z)$ for all $x<z$. It follows that
\begin{equation*}
    \sum_{x\leq y\leq z} f(y)\mu_{\mathcal{L}}(y,z)
    =\sum_{x\leq y\leq z} \mu_{\mathcal{L}}(x,y) f(y)
    =\begin{cases}
        f(x) & \text{if $x=z$,} \\
        0    & \text{if $x<z$,} \\
    \end{cases}
\end{equation*}
for all $x\leq z$ and any function $f$ on $\mathcal{L}$ taking values in an abelian group.
\begin{example}
    In the lattice $\LCF$ of cyclic flats of the (cycle matroid of the) wheel graph with three spokes from \cref{fig:cyclicflats}, the pairs $x<z$ and their M\"obius functions are of the form
    \begin{equation*}
        \mu_{\LCF}\left(\Graph[0.3]{w3A3}\,,\,\Graph[0.3]{w3A}\right)
        =\mu_{\LCF}\left(\emptyset,\,\Graph[0.3]{w3A3}\ \right)
        = -1
        \quad\text{and}\quad
        \mu_{\LCF}\left(\emptyset,\,\Graph[0.3]{w3A}\right)
        =3.
    \end{equation*}
\end{example}

The algorithm \cite[\S5]{Ferroni:SchubertDelannoySpeyer} computes the Speyer polynomial $g_M(t)$ of a matroid by decomposing the matroid polytope of $M$ into matroid polytopes of \emph{Schubert matroids}. A Schubert matroid is a matroid whose lattice of cyclic flats is a chain, and given a chain $C_\bullet\in\chains(\LCF_M)$ of cyclic flats of $M$, we denote by $\SM{C_\bullet}$ the Schubert matroid whose lattice of cyclic flats is precisely this chain, so that $\LCF_{\SM{C_\bullet}}=\{C_0,\ldots,C_r\}$.
\begin{theorem}[{\cite[Corollary~5.3]{Ferroni:SchubertDelannoySpeyer}}] \label{thm:Speyer-from-cyclics}
    Given a chain $C_\bullet\in\chains(\LCF_M)$ of cyclic flats, set $\lambda_{C_\bullet}\defas - \mu_{\chains(\LCF_M)}(C_\bullet,1_{\chains})$.
    Then for every matroid $M$ without loops or coloops, 
\begin{equation}\label{eq:Speyer-from-cyclics}
    g_M(t) = (-1)^{\cc(M)-1} \sum_{C_{\bullet}\in \chains(\LCF_M)\setminus\{1_{\chains}\}} \lambda_{C_\bullet} g_{\SM{C_\bullet}}(t).
\end{equation}
\end{theorem}
This can be thought of as a sum over all simplices in the order complex $\Delta(\LCF_M^\circ)\cong \chains(\LCF_M)\setminus\{1_{\chains}\}$ of the poset of the proper, non-empty cyclic flats $\LCF_M^\circ=\LCF_M\setminus\{\emptyset,M\}$.

To evaluate \eqref{eq:Speyer-from-cyclics}, one needs to find the coefficients $\lambda_{C_{\bullet}}\in\Z$. They can be computed as $\lambda_{C_\bullet}=1-\sum_{D_\bullet > C_\bullet} \lambda_{D_\bullet}$, which follows from the defining recursion of the M\"obius function in the lattice of chains $\chains(\LCF_M)$.
Our first contribution is a simple formula for $\lambda_{C_{\bullet}}$ directly in terms of the M\"obius function of the lattice $\LCF_M$. This identity holds in any lattice and is not specific to cyclic flats; a version for the lattice of all flats is implicit in \cite[\S8]{FinkShawSpeyer:Omega}.

Since a lattice is typically much smaller than the lattice of all its chains, this formula provides a much more efficient route to compute $\lambda_{C_\bullet}$ than the defining recursion.
\begin{proposition}\label{lem:MF(chain)}
    For any chain $C_\bullet\in\chains(\mathcal{L})$ of length $r$ in a finite lattice $\mathcal{L}$, we have
    \begin{equation}\label{eq:MF(chain)}
        -\mu_{\chains(\mathcal{L})}(C_\bullet,1_{\chains(\mathcal{L})}) = \prod_{i=1}^{r} \left(-\mu_{\mathcal{L}}(C_{i-1},C_i) \right).
    \end{equation}
\end{proposition}
\begin{proof}
    By the recursion $\mu(x,z)=-\sum_{x\leq y<z} \mu(x,y)$, the left-hand side of \eqref{eq:MF(chain)} is equal to a sum over all chains $D_\bullet\in\chains(\mathcal{L})$, $D_\bullet\neq 1_{\chains}$, which refine $C_\bullet$:
    \begin{equation}\label{eq:MF(C,1)=D}
        -\mu_{\chains(\mathcal{L})}(C_\bullet,1_{\chains(\mathcal{L})})
        =
        \sum_{C_\bullet \leq D_\bullet} \mu_{\chains(\mathcal{L})}(C_\bullet,D_\bullet).
    \end{equation}
    Refinement means $C_i=D_{k_i}$ for some sequence $k_1<\ldots<k_r$. In particular, the subsequences $D^{(i)}_\bullet\defas \{D_{k_i}<D_{k_i+1}<\ldots<D_{k_{i+1}}\}$ are refinements of the length one chains $C^{(i)}_\bullet\defas \{C_i<C_{i+1}\}$. In other words, the interval between $C_\bullet$ and $D_\bullet$ is a product
    \begin{equation*}
        [C_\bullet,D_\bullet]_{\chains(\mathcal{L})} 
        \cong \prod_{i=0}^{r-1} [C^{(i)}_\bullet,D^{(i)}_\bullet]_{\chains(\mathcal{L})},
    \end{equation*}
    where we denote intervals as $[C_\bullet,D_\bullet]_{\chains(\mathcal{L})}=\{X_\bullet\in\chains(\mathcal{L})\colon C_\bullet\leq X_\bullet\leq D_\bullet\}$. By \cite[Proposition~3.8.2]{Stanley:EC1}, the M\"obius function factorizes, and we conclude
    \begin{equation*}
        -\mu_{\chains(\mathcal{L})}(C_\bullet,1_{\chains(\mathcal{L})})
        =\prod_{i=0}^{r-1}
        \sum_{C_\bullet^{(i)}\leq D_\bullet^{(i)}} \mu_{\chains(\mathcal{L})}(C_\bullet^{(i)},D_\bullet^{(i)})
        .
    \end{equation*}
    A chain $D_\bullet^{(i)}$ that refines $C_\bullet^{(i)}$ is nothing but a chain of elements that belong to the subset $J_i\defas[C_i,C_{i+1}]_\mathcal{L}=\{X\in\mathcal{L}\colon C_i\leq X\leq C_{i+1}\}$ of $\mathcal{L}$ formed by the interval between $C_i$ and $C_{i+1}$. Hence we can apply \eqref{eq:MF(C,1)=D} in reverse to each factor, so that
    \begin{equation*}
        -\mu_{\chains(\mathcal{L})}(C_\bullet,1_{\chains(\mathcal{L})})
        =\prod_{i=0}^{r-1}\left(
        -\mu_{\chains(J_i)}(0_{\chains(J_i)},1_{\chains(J_i)})
        \right)
        ,
    \end{equation*}
    where $0_{\chains(J_i)}=C_\bullet^{(i)}$ is the length one chain between the minimal and maximal elements of the lattice $J_i$. We conclude by \cite[Exercise~\S3.141]{Stanley:EC1}, which states that
    \begin{equation*}
        \mu_{\mathcal{L}}(0_{\mathcal{L}},1_{\mathcal{L}})
        =\mu_{\chains(\mathcal{L})}(0_{\chains(\mathcal{L})},1_{\chains(\mathcal{L})}).
    \end{equation*}
    Namely, these two M\"obius functions compute the Euler characteristics of the order complexes of the poset $\mathcal{L}^\circ$ and its chain poset $\chains(\mathcal{L})^\circ$ (\Cref{prop:chi=moebius}). The chains of $\mathcal{L}^\circ$ label the simplices in the barycentric subdivision of $\Delta(\mathcal{L}^\circ)$. Therefore, $\Delta(\mathcal{L}^\circ)$ and $\Delta(\chains(\mathcal{L})^\circ)$ are homotopic and in particular share the same Euler characteristic.
\end{proof}
\begin{example}\label{ex:wheels-schubert}
    For the wheel with three spokes, there are only two types of chains:
    \begin{enumerate}
        \item $C_\bullet=\{\emptyset<\Graph[0.15]{w3A}\}$ with $\lambda_{C_\bullet}=-\mu_{\LCF}(\emptyset,\Graph[0.15]{w3A})=-3$,
        \item $C_\bullet=\{\emptyset<\Graph[0.15]{w3A3} < \Graph[0.15]{w3A}\}$ with $\lambda_{C_\bullet}=(-\mu_{\LCF}(\emptyset,\Graph[0.15]{w3A3}\,))\cdot(-\mu_{\LCF}(\Graph[0.15]{w3A3},\Graph[0.15]{w3A}))=1\cdot 1=1$.
    \end{enumerate}
    All four copies of the second type of chain yield isomorphic Schubert matroids, so \cref{thm:Speyer-from-cyclics} gives
    \begin{equation}\label{eq:Schubert-Wheel3}
        g_{\Graph[0.12]{w3A}}(t)=
        4\times g_{\SM{\emptyset\,<\,\Graph[0.12]{w3A3} \,<\, \Graph[0.12]{w3A}}}(t)
        -3\times g_{\SM{\emptyset\,<\, \Graph[0.12]{w3A}}}(t)
        .
    \end{equation}
    Similarly, listing all chains and computing the M\"obius function of the cyclic flats of the wheel with four spokes (see \cref{fig:cyclicflats}), we obtain\footnote{The minimal chain $\emptyset<\Graph[0.12]{ws4A}$ does not contribute because it has $\lambda_{C_\bullet}=-\mu_{\LCF}(\emptyset,\Graph[0.12]{ws4A})=0$, see \cref{ex:ordercomplex}.}
    \begin{equation*}\begin{aligned}
        g_{\Graph[0.2]{ws4A}}(t)
        =&
        -4\times g_{\SM{\emptyset\,<\,\Graph[0.2]{ws4A122} \,<\, \Graph[0.2]{ws4A}}}(t)
        +1\times g_{\SM{\emptyset\,<\,\Graph[0.2]{ws4A4o} \,<\, \Graph[0.2]{ws4A}}}(t)
        \\ &
        -4\times g_{\SM{\emptyset\,<\,\Graph[0.2]{ws4A3}\,<\, \Graph[0.2]{ws4A}}}(t)
        +8\times g_{\SM{\emptyset\,<\,\Graph[0.2]{ws4A3}\,<\, \Graph[0.2]{ws4A122} \,<\, \Graph[0.2]{ws4A}}}(t)
        .
    \end{aligned}\end{equation*}
\end{example}

\subsection{Lattice paths}\label{sec:lattice-paths}
To compute the Speyer polynomial of a Schubert matroid, \cite{Ferroni:SchubertDelannoySpeyer} exploits a representation of Schubert matroids in terms of lattice paths. We review this description below, and then suggest an efficient method (\Cref{alg:gLPM}) to compute Speyer's polynomial for any Schubert matroid.

A word $P\in\{\PU,\PR\}^{\times}$ in two letters $\PU$ (north) and $\PR$ (east) encodes a lattice path starting at the origin $(0,0)$ with steps $\PU=(0,1)$ and $\PR=(1,0)$. If $P$ has $r$ letters $\PU$ and $n$ letters in total, then this path ends at $(n-r,r)$. All such paths are determined by the positions $I=\{i_1<\ldots<i_r\}$ of the letters $\PU$, such that $P=P(I)$ where
\begin{equation*}
    P(I)\defas\PR^{i_1-1}\PU\PR^{i_2-i_1-1}\PU\ldots\PR^{i_r-i_{r-1}-1}\PU\PR^{n-i_r}.
\end{equation*}
This gives a bijection $I\mapsto P(I)$ from subsets $I$ of $\{1,\ldots,n\}$ of size $r$, to lattice paths of length $n$ with $r$ letters $\PU$.
Given another such subset $J=\{j_1<\ldots<j_r\}$, we write $P(J)\leq P(I)$ if and only if $j_k\geq i_k$ for all $1\leq k\leq r$. Visually, this condition dictates that the path $P(J)$ stays on or below the path $P(I)$, but never protrudes above it; see \cref{fig:paths}.
\begin{figure}
    \centering
    \includegraphics[scale=0.8]{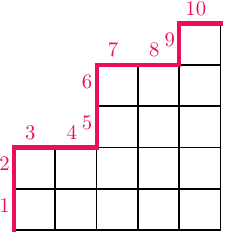} \qquad \includegraphics[scale=0.8]{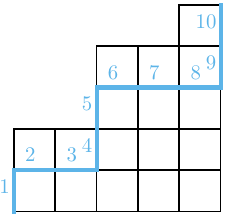}
    \caption{The red lattice path \textcolor{OrangeRed}{$\PU\PU\PR\PR\PU\PU\PR\PR\PU\PR=P(I)$} with $I=\{1,2,5,6,9\}$, passes on or above the blue lattice path \textcolor{CornflowerBlue}{$\PU\PR\PR\PU\PU\PR\PR\PR\PU\PU=P(J)$} with $J=\{1,4,5,9,10\}$. Thus $P(J)\leq P(I)$, so $J$ is one basis (of many) of the matroid $\LPM{P(I)}$.}%
    \label{fig:paths}%
\end{figure}
\begin{definition}
     The \emph{lattice path matroid} $\LPM{P}$ for $P=P(I)\in\{\PU,\PR\}^n$ with $r$ letters $\PU$ is the matroid of rank $r$ on the ground set $\{1,\ldots,n\}$ such that $J=\{j_1,\ldots,j_r\}$ is a basis of $\LPM{P}$ if and only if $P(J)\leq P(I)$.
\end{definition}
\begin{example}\label{eq:uniform-path}
    The path $P=\PU^r\PR^{n-r}=P(I)$ for thus subset $I=\{1,\ldots,r\}$ bounds (from above) the rectangle $[0,n-r]\times[0,r]$. Therefore, $P(J)\leq P(I)$ is true for \emph{every} subset $J$ of $\{1,\ldots,n\}$ with size $r$. Therefore, $\LPM{\PU^r\PR^{n-r}}\cong\UM{n}{r}$ is the uniform matroid.
\end{example}
These lattice path matroids are indeed Schubert matroids: By collecting repeated consecutive letters, any lattice path $P$ can be written uniquely in the form
\begin{equation*}
    P=\PR^{a_0}\PU^{b_1}\PR^{a_1}\PU^{b_2}\PR^{a_2}\ldots \PU^{b_k}\PR^{a_k}\PU^{b_{k+1}}
\end{equation*}
with integers $a_0,b_{k+1}\geq 0$ and $a_1,b_1,\ldots,a_k,b_k\geq 1$. Then by \cite[Lemma~4.3]{BoninMier:LatticePath}, the matroid $\LPM{P}$ has precisely $k+1$ cyclic flats $C_0<\ldots<C_k$, which form a chain. Concretely, $C_i=\{1,\ldots,m_i\}$ corresponds to the initial segment of $P$ of length $m_i$, from the origin to the point where the $i$th sequence $\PU^{b_i}$ of north steps begins; that is $m_i=a_0+\sum_{j\leq i}(a_j+b_j)$. In particular, the loops (coloops) of $\LPM{P}$ are the first $a_0$ (last $b_{k+1}$) edges.
\begin{example}
    For the path $P=\PU\PU\PR\PR\PU\PU\PR\PR\PU\PR=P(\{1,2,5,6,9\})$ from \cref{fig:paths}, we have $k=3$ and $a_0=0,b_1=a_1=b_2=a_2=2,b_3=a_3=1$. Hence the lattice path matroid $\LPM{P}$ has four cyclic flats $\LCF=\{C_0,C_1,C_2,C_3\}$ where $C_0=\emptyset$, $C_1=\{1,2,3,4\}$, $C_2=\{1,\ldots,8\}$, and $C_3=\{1,\ldots,10\}$.
\end{example}
Conversely, consider any Schubert matroid $\SM{C_\bullet}$ with cyclic flats $C_0<\ldots<C_k$ on the ground set $\{1,\ldots,n\}$. For the computation of Speyer's polynomial, we may assume that $\SM{C_\bullet}$ has neither loops nor coloops; that is, $C_0=\emptyset$ and $C_k=\{1,\ldots,n\}$. Such a Schubert matroid is isomorphic to the lattice path matroid $\LPM{P}\cong \SM{C_\bullet}$ with
\begin{equation}\label{eq:path-from-cyc}
    P=\PU^{\rk(C_1)}\PR^{\loops(C_1)}\PU^{\rk(C_2)-\rk(C_1)}\PR^{\loops(C_2)-\loops(C_1)}\ldots\PU^{\rk(C_k)-\rk(C_{k-1})}\PR^{\loops(C_k)-\loops(C_{k-1})}
\end{equation}
where $\rk(C_i)$ and $\loops(C_i)=\abs{C_i}-\rk(C_i)$ denote the rank and corank (nullity), respectively, of a cyclic flat. Combined with \Cref{thm:Speyer-from-cyclics}, we can thus decompose the Speyer polynomial of any matroid (without loops or coloops) into Speyer polynomials of lattice path matroids.
\begin{example}\label{ex:g(wheel)=LPM}
    Depicting a lattice path as in \cref{fig:paths}, the decompositions for the wheels from \cref{ex:wheels-schubert} read
    \begin{align}
        g_{\Graph[0.12]{w3A}}(t) &=
        4\times g_{\LPM{\Graph[0.15]{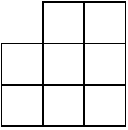}}}(t)
        -3\times g_{\LPM{\Graph[0.15]{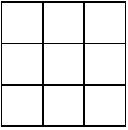}}}(t),
        \label{eq:Path-Wheel3}
        \\
        g_{\Graph[0.2]{ws4A}}(t) &=
        -4\times g_{\LPM{\Graph[0.15]{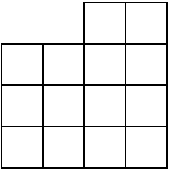}}}(t)
        +1\times g_{\LPM{\Graph[0.15]{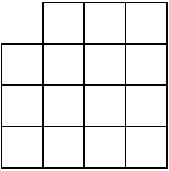}}}(t)
        -4\times g_{\LPM{\Graph[0.15]{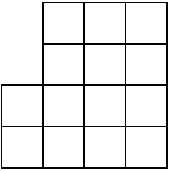}}}(t)
        +8\times g_{\LPM{\Graph[0.15]{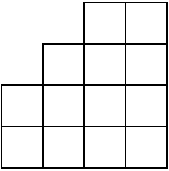}}}(t)
        . \nonumber
    \end{align}
\end{example}
In order to compute the $g$-polynomial of a lattice path matroid, \cite[\S3]{Ferroni:SchubertDelannoySpeyer} exploits the covaluative property once more, to decompose these matroids into sums of series-parallel matroids. The result is a combinatorial description \cite[Theorem~3.4]{Ferroni:SchubertDelannoySpeyer} stating that
\begin{equation}\label{eq:g-from-delannoy}
    g_{\LPM{P}}(t)=\sum_{k=1}^{r} c_k(P) t^k
\end{equation}
where $r$ is the rank of $\LPM{P}$ and $c_k(P)\in\Z_{\geq 0}$ is the number of \emph{admissible Delannoy paths} with $k$ diagonal steps. These paths start at $(1,1)$ and end at $(r,n-r)$ and admit three kinds of steps: $\PU$, $\PR$, plus the diagonal step $(1,1)$. To be \emph{admissible}, a Delannoy path may not protrude above $P$, and in addition the steps $\PU$ and $(1,1)$ are not allowed at any lattice point at which $P$ goes north; see \cite[Figure~3.1]{Ferroni:SchubertDelannoySpeyer}. By considering the three possibilities for the last step of such a Delannoy path, \eqref{eq:g-from-delannoy} yields a recursive description of the $g$-polynomial (\Cref{alg:gLPM}): Let $a\geq0$ denote the number of trailing east steps in $P$, so that $P=Q\PU\PR^a$ for some lattice path $Q$. We may assume that $a\geq1$, for otherwise, $\LPM{P}$ would have a coloop and thus vanishing $g$-polynomial.
If $a=1$, then the last step of the Delannoy path must be north, so
\begin{equation*}
    g_{\LPM{Q\PU\PR}}(t)=g_{\LPM{Q\PR}}(t).
\end{equation*}
If $a\geq2$, then the last step of an admissible Delannoy path can be $\PU$, $\PR$, or $(1,1)$, hence
\begin{equation*}
    g_{\LPM{Q\PU\PR^a}}(t)=g_{\LPM{Q\PR^a}}(t)+g_{\LPM{Q\PU\PR^{a-1}}}(t)+t\cdot g_{\LPM{Q\PR^{a-1}}}(t).
\end{equation*}
Using these recursions, we can reduce to lattice path matroids of rank one. Since we only consider matroids without loops or coloops (otherwise $g_M(t)=0$), this means $P=\PU\PR^{n-1}$ with $n\geq 2$. This is a uniform matroid $\LPM{\PU\PR^{n-1}}\cong\UM{n}{1}$, the cycle matroid of the cycle graph with $n$ edges; hence series-parallel with $g(t)=t$.
\begin{algorithm}
    \caption{gPathMat($P$)}%
    \label{alg:gLPM}%
    \KwIn{path $P=P_1\ldots P_n\in\{\PU,\PR\}^n$ of length $n$}
    \KwOut{Speyer polynomial $g(t)\in\Z[t]$ of the lattice path matroid $\LPM{P}$}
    \If(\tcc*[f]{loops or coloops}){$P_1=\PR$ \KwOr $P_n=\PU$}{
        \Return{$0$}
    }
    $a \gets  \text{number of trailing $\PR$'s in $P$}$\;
    $Q \gets  \text{prefix of $P$ so that $P=Q\PU\PR^a$}$\;
    \If(\tcc*[f]{$\rk(\LPM{P})=1$ (series-parallel)}){$Q=\emptyset$}{
        \Return{$t$}
    }
    \If(\tcc*[f]{last admissible Delannoy step can only be $\PU$}){$a=1$}{
        \Return{\gLPM{$Q\PR$}}
    }
    \Else{
        \Return{$\gLPM{$Q\PR^a$}+\gLPM{$Q\PU\PR^{a-1}$}+t\cdot\gLPM{$Q\PR^{a-1}$}$}
    }
\end{algorithm}

The algorithm proposed in \cite[\S5]{Ferroni:SchubertDelannoySpeyer} evaluates $g_{\LPM{P}}(t)$ via a sum over bases with restricted Tutte-like activities, see \cite[Theorem~4.3]{Ferroni:SchubertDelannoySpeyer}. The downside of this elegant description is that the number of bases of a lattice path matroid can be as big as $\binom{n}{r}$, which is achieved by the uniform matroid $\UM{n}{r}\cong\LPM{\PU^r\PR^{n-r}}$. Explicitly enumerating all admissible Delannoy paths is therefore not efficient. Instead, the recursive \Cref{alg:gLPM}, implemented with dynamic programming, is much more efficient:
\begin{lemma}
    \Cref{alg:gLPM} can be implemented in time and space $\bigO{n^3\log n}$ for paths of length $n$.
\end{lemma}
\begin{proof}
    The shorter paths $Q\PR^a,Q\PU\PR^{a-1},Q\PR^{a-1}$ in the recursive step arise from $P$ my removing the \emph{right-most} occurrence of the letter $\PU$ and/or the right-most occurrence of the letter $\PR$. The entire recursion tree therefore produces at most $\bigO{n^2}$ many subsequences of $P$, since each subsequence is obtained from the initial word $P$ by removing some number $i$ of the right-most occurrences of the letters $\PU$ and some number $j$ of the right-most occurrences of the letters $\PR$. Since $i,j\leq n$, the function \gLPM{$Q$} is called with at most $n^2$ \emph{different} words $Q$ as argument.

    Using dynamic programming, i.e.\ storing the values of \gLPM{$Q$} in a table once they are computed for the first time, and avoiding any re-computation by looking up the result in that table, the function needs to be executed at most $n^2$ times. Since all occurring $g$-polynomials in the recursion have degree at most $\rk(\LPM{P})$, which is at most $n$, storing all these polynomials requires storing at most $n^3$ integer coefficients. 
    
    From the recursion it is clear that the coefficients are bounded by $3^n$, so a single coefficient requires at most $\log_2(3^n)\in\bigO{\log n}$ bits in space. The only arithmetic operations required are $n^2\cdot n$ additions of such integers. Each of these additions can be done in time $\bigO{\log n}$ proportional to the number of bits.
\end{proof}
\begin{example}\label{ex:g-path-rec}
    Using $g_M(t)=t$ for series-parallel lattice path matroids $M=\LPM{\PU\PR^{n-1}}$ and $M=\LPM{\PU^{n-1}\PR}$ as base cases, the recursions readily give
    \begin{align*}
        g_{\LPM{\Graph[0.15]{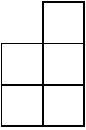}}}(t)
        &=g_{\LPM{\Graph[0.15]{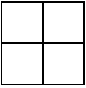}}}(t)
        = \underbrace{g_{\LPM{\Graph[0.15]{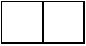}}}(t)}_{t}+\underbrace{g_{\LPM{\Graph[0.15]{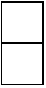}}}(t)}_{t}+t\cdot \underbrace{g_{\LPM{\Graph[0.15]{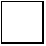}}}(t)}_{t}
        = 2t+t^2,
        \\
        g_{\LPM{\Graph[0.15]{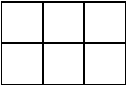}}}(t)
        &= \underbrace{g_{\LPM{\Graph[0.15]{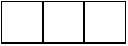}}}(t)}_{t}+\underbrace{g_{\LPM{\Graph[0.15]{UURR}}}(t)}_{2t+t^2}+t\cdot \underbrace{g_{\LPM{\Graph[0.15]{URR}}}(t)}_{t}
        =3t+2t^2,
        \\
        g_{\LPM{\Graph[0.15]{UURURR}}}(t)
        &=\underbrace{g_{\LPM{\Graph[0.15]{UURRR}}}(t)}_{3t+2t^2}+\underbrace{g_{\LPM{\Graph[0.15]{UURUR}}}(t)}_{2t+t^2}+t\cdot \underbrace{g_{\LPM{\Graph[0.15]{UURR}}}(t)}_{2t+t^2}
        =5t+5t^2+t^3,
        \\
        g_{\LPM{\Graph[0.15]{UUURRR}}}(t)
        &=\underbrace{g_{\LPM{\Graph[0.15]{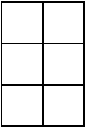}}}(t)}_{3t+2t^2}+\underbrace{g_{\LPM{\Graph[0.15]{UURRR}}}(t)}_{3t+2t^2}+t\cdot \underbrace{g_{\LPM{\Graph[0.15]{UURR}}}(t)}_{2t+t^2}
        =6t+6t^2+t^3.
    \end{align*}
    In the last step we used that $\LPM{\Graph[0.15]{UURRR}}\cong\UM{5}{2}$ is the dual of $\LPM{\Graph[0.15]{UUURR}}\cong\UM{5}{3}$ and thus shares the same $g$-polynomial; instead of this shortcut we could of course also just apply the recursion again. Either way, combining these $g$-polynomials of lattice path matroids with the Schubert decomposition from \eqref{eq:Schubert-Wheel3} in the lattice form \eqref{eq:Path-Wheel3}, we find
    \begin{equation*}
        g_{\Graph[0.12]{w3A}}(t)
        =4\cdot(5t+5t^2+t^3)-3\cdot (6t+6t^2+t^3)
        =2t+2t^2+t^3.
    \end{equation*}
\end{example}

\subsection{First derivative}\label{sec:gdiff}
In this section, we combine the previous observations with relations of various lattices, to prove \Cref{prop:N1asbeta}. This expresses the first derivative of $g_M(t)$ at $t=-1$ in terms of an elementary formula without any reference to chains or cyclic flats:
\begin{equation*}
    g'_M(-1)=(-1)^{\cc(M)-1} \sum_{A\subseteq M} (-1)^{\loops(A)} \beta(A) \rk(A).
\end{equation*}

To prove this identity, we first recall the result of \cite[Theorem~4.3]{Ferroni:SchubertDelannoySpeyer}, which gives a direct combinatorial interpretation of the coefficients $\FP_i(M)$ in the expansion \eqref{eq:gexpand1} of the Speyer polynomial $g_M(t)$ in powers of $(1+t)$, instead of the expansion \eqref{eq:g-from-delannoy} in powers of $t$. Namely, for a lattice path matroid $M=\LPM{P}$ without loops or coloops, we have
\begin{equation}\label{eq:gexpand1main}
    g_{M}(t)=t\sum_{i=0}^{\rk(M)-1} \FP_i(M) \cdot (1+t)^i
\end{equation}
where the integers $\FP_i(M)\in\Z_{\geq 0}$ are non-negative and count the possible configurations of the diagonal steps of admissible Delannoy paths. This identity follows from \eqref{eq:g-from-delannoy} by considering the map which replaces each north-east corner $\ldots\PU\PR\ldots$ in a Delannoy path by a diagonal step. This replacement rule defines a projection from the set of all admissible Delannoy paths, to the subset of all admissible Delannoy paths without north-east corners. Under this projection, the preimage of a (north-east corner free) Delannoy path $Q$ consists of $2^i$ elements, where $i$ is the number of diagonal steps in $Q$; because each diagonal step can either stay or expand into a corner $\PU\PR$. The contribution to \eqref{eq:g-from-delannoy} from the Delannoy paths in the preimage is thus $\sum_{j=0}^i \binom{i}{j} t^j=(1+t)^i$, proving \eqref{eq:gexpand1main}; see also \cite[\S4]{Ferroni:SchubertDelannoySpeyer}. This grouping of Delannoy paths is illustrated in \cref{fig:delannoy}.
\begin{figure}
    \centering
    $\Graph[0.5]{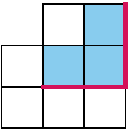}\colon t$ \qquad
    $\Graph[0.5]{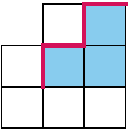},\Graph[0.5]{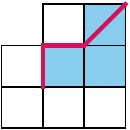},\Graph[0.5]{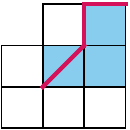} \mapsto \Graph[0.5]{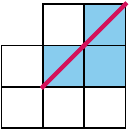}\colon t+t^2+t^2+t^3=t(1+t)^2$ \\
    $\Graph[0.5]{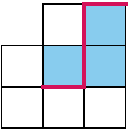} \mapsto \Graph[0.5]{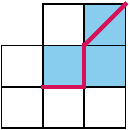}\colon t(1+t)$ \quad
    $\Graph[0.5]{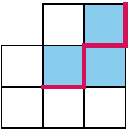} \mapsto \Graph[0.5]{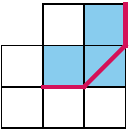}\colon t(1+t)$ \quad
    $\Graph[0.5]{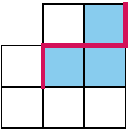} \mapsto \Graph[0.5]{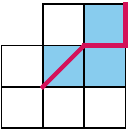}\colon t(1+t)$
    \caption{All 11 admissible Delannoy paths of the matroid $\LPM{\PU\PU\PR\PU\PR\PR}$, grouped by replacing north-east corners with diagonal steps. This collects the $11$ monomials of $g_M(t)=5t+5t^2+t^3$ (\Cref{ex:g-path-rec}) as $g_M(t)=t+3t(1+t)+t(1+t)^2$.}%
    \label{fig:delannoy}%
\end{figure}

We think of diagonal steps $(x,y)\mapsto(x+1,y+1)$ by marking the corresponding lattice square $\square_{x,y}\defas [x,x+1]\times [y,y+1]$. As Delannoy paths start at $(1,1)$, the squares $\square_{x,0}$ on the bottom are excluded. Admissibility furthermore excludes all squares $\square_{x,y}$ that lie immediately to the right of a north step $(x,y)\mapsto(x,y+1)$ of $P$. We denote by $\admsq{P}$ the remaining points $(x,y)$ which correspond to squares $\square_{x,y}$ that \emph{are} admissible. In other words, $\admsq{P}\subset\Z^2$ is the set of \emph{interior} lattice points in the region bounded by $P$, $y\geq 0$, and $x\leq n-r$, where $P$ has $r$ north steps and $n$ steps in total. Then
\begin{equation}\label{eq:FP-squares}
    \FP_i(\LPM{P})=\abs{\left\{(x_1,y_1),\ldots,(x_i,y_i)\in\admsq{P}\colon x_1<\ldots<x_i\ \text{and}\ y_1<\ldots<y_i\right\}},
\end{equation}
because for a prescribed set of diagonal steps, there is exactly one Delannoy path with precisely those diagonal steps and no north-east corners; namely the path
\begin{equation*}
    \PR^{x_1-1}\PU^{y_1-1}(1,1)
    \ldots\PR^{x_i-x_{i-1}-1}\PU^{y_i-y_{i-1}-1}(1,1)\PR^{n-r-x_i-1}\PU^{r-y_i-1}.
\end{equation*}
\begin{corollary}\label{lem:N0(LPM)=1}
    For every Schubert matroid $M$ without loops or coloops, $\FP_0(M)=1$.
\end{corollary}
\begin{proof}
    There is precisely one Delanny path from $(1,1)$ to $(n-r,r)$ that has no north-east corners and zero diagonal steps, namely $\PR^{n-r-1}\PU^{r-1}$.
\end{proof}
\begin{corollary}\label{lem:N1(LPM)=boxes}
    For every Schubert matroid $M$ without loops or coloops, $\FP_1(M)=\abs{\admsq{P}}$ is the number of admissible squares for the lattice path $P$ such that $M\cong\LPM{P}$.
\end{corollary}
\begin{proof}
    For any $(x,y)\in\admsq{P}$, there is exactly one admissible Delannoy path with only this diagonal step and no north-east corners, namely $\PR^{x-1}\PU^{y-1}(1,1)\PR^{n-r-x-1}\PU^{r-y-1}$.
\end{proof}
\begin{example}
    The lattice path $P=\PU\PU\PR\PU\PR\PR$ has three admissible positions for diagonal steps: $\admsq{P}=\{(1,1),(1,2),(2,2)\}$, the highlighted squares in $\Graph[0.25]{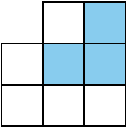}$. We can choose $(1,1)$ and $(2,2)$ together, so $\FP_2(\LPM{P})=1$; and by the corollaries above, $\FP_0(\LPM{P})=1$ and $\FP_1(\LPM{P})=\abs{\admsq{P}}=3$. Hence \eqref{eq:gexpand1main} gives $g_{\LPM{P}}(t)=t+3(1+t)+t(1+t)^2$, which is equal to $5t+5t^2+t^3$ as computed in \Cref{ex:g-path-rec}. See also \cref{fig:delannoy}.
\end{example}
We now extend the definition of $\FP_i(M)$ from Schubert/lattice path matroids to all matroids, such that \eqref{eq:gexpand1main} continues to hold. With the polynomial $g_M$, also these coefficients $\FP_i(M)$ are covaluative matroid invariants. Hence from \Cref{thm:Speyer-from-cyclics}, we have the identity
\begin{equation}\label{eq:FP=chains}
    \FP_i(M)=(-1)^{\cc(M)-1}\sum_{C_\bullet \in\chains(\LCF_M)\setminus\{1_{\chains}\}} \lambda_{C_\bullet} \FP_i(\SM{C_\bullet})
\end{equation}
for any matroid $M$ without loops or coloops. For $i=0$, by \Cref{lem:N0(LPM)=1} the summand simplies to just the M\"obius function $\lambda_{C_\bullet}=-\mu_{\chains}(C_\bullet,1_{\chains})$. By the defining property of the M\"obius function, the sum over $C_\bullet$ thus produces $\mu_{\chains}(1_\chains,1_\chains)=1$ such that
\begin{equation*}
    \FP_0(M)=(-1)^{\cc(M)-1}
\end{equation*}
for \emph{every} matroid without loops or coloops. We have thus (re-)proved the identity $g_M(-1)=-\FP_0(M)=(-1)^{\cc(M)}$ from \eqref{eq:gm10}.
Similarly, we now want to simplify the sum over the chains $C_\bullet$ in \eqref{eq:FP=chains} in the more complicated case $i=1$. To this end, we first express $\FP_1(M)$ of a Schubert matroid in terms of the ranks and sizes of its cyclic flats.
\begin{lemma}\label{lem:N1(SM)}
    For a Schubert matroid on $n$ elements, without loops or coloops, whose cyclic flats are $C_0<\ldots<C_k$, we have the identity
    \begin{equation}\label{eq:N1(SM)}
        \FP_1(\SM{C_\bullet})=1-n+\sum_{i=1}^{k} \rk(C_i)\cdot (\loops(C_i)-\loops(C_{i-1})).
    \end{equation}
\end{lemma}
\begin{proof}
    Let $P$ be the lattice path \eqref{eq:path-from-cyc} such that $M\defas\SM{C_\bullet}\cong\LPM{P}$ and thus $\FP_1(M)=\abs{\admsq{P}}$ by \Cref{lem:N1(LPM)=boxes}. The excluded squares are precisely the $\loops(M)=n-\rk(M)$ squares in the bottom row and the $\rk(M)$ squares to the right of the north steps of $P$. This counts the square $[0,1]\times[0,1]$ twice, so the total number of non-admissible squares is $n-1$. Thus $\FP_1(M)$ equals $1-n$ plus the total number of lattice squares $[x,x+1]\times[y,y+1]$ underneath $P$. These squares are partitioned into rectangles by the intervals $[\loops(C_{i-1}),\loops(C_i)]$ containing $[x,x+1]$, proving the claim (see \cref{fig:N1-counting}).
\end{proof}
\begin{figure}
    \centering
    \includegraphics[scale=1]{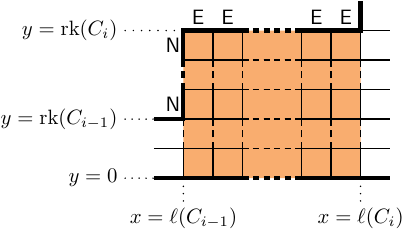}
    \caption{The portion of lattice squares that lie in the rectangle underneath the section $\ldots\PU^{\rk(C_i)-\rk(C_{i-1})}\PR^{\loops(C_i)-\loops(C_{i-1})}\ldots$ of the lattice path from \eqref{eq:path-from-cyc}.}%
    \label{fig:N1-counting}%
\end{figure}
We now combine \Cref{lem:N1(SM)} with \eqref{eq:FP=chains} to express $\FP_1(M)$ of any matroid as a sum over pairs of cyclic flats. For this we use the following well-known result about the lattice of cyclic flats, see \cite[Corollay~1]{FHGHW:CycBin} or \cite[Remark~1]{Eberhardt:TutteCycFlat}:
\begin{lemma}\label{lem:cyc-interval}
    For any two nested cyclic flats $A<B$ of a matroid $M$, the map $C\mapsto C/A$ provides an isomorphism of lattices, from the interval $[A,B]_{\LCF_M}$ of cyclic flats $C$ of $M$ between $A$ and $B$ is isomorphic, to the lattice of cyclic flats of $B/A$.
\end{lemma}
\begin{proposition}\label{lem:N1(M)=cyc}
    For any matroid $M$ on $n$ elements without loops or coloops, we have
    \begin{equation}\label{eq:N1(M)=cyc}
        (-1)^{\cc(M)-1}\FP_1(M)=1-n+\sum_{\substack{A,B\in\LCF_M\\ A\leq B}} \mu_{\LCF_M}(A,B) \loops(A)\rk(B).
    \end{equation}
\end{proposition}
\begin{proof}
    Since $M$ has neither loops nor coloops, the same applies to all cyclic flats of $M$, and thus also all Schubert matroids $\SM{C_\bullet}$ appearing in \eqref{eq:FP=chains} are free of loops and coloops. Hence we may simply substitute \eqref{eq:N1(SM)} into \eqref{eq:FP=chains}.
    
    The contribution $1-n$ in front of the sum in \eqref{eq:N1(SM)} is independent of $C_\bullet$, so its sum is $(1-n)\sum_{C_\bullet\neq 1_\chains} \lambda_{C_\bullet}=(1-n)\cdot 1$ as in the previous calculation of $\FP_0(M)$.

    For the contributions from the sum over $i$ in \eqref{eq:N1(SM)}, we collect $\sum_{C_\bullet\neq 1_\chains}\sum_i$ by the pairs $A<B$ of consecutive cyclic flats $A=C_{i-1}$ and $B=C_{i}$ picked from the chain. Then
    \begin{equation*}
        \sum_{\substack{C_\bullet\in\chains(\LCF_M) \\ C_\bullet\neq 1_\chains}} \sum_{i\geq1} \rk(C_i)\cdot(\loops(C_i)-\loops(C_{i-1}))
        = \sum_{\substack{A,B\in\LCF_M \\ A<B}} \rk(B)(\loops(B)-\loops(A)) \lambda_{A,B}
    \end{equation*}
    where $\lambda_{A,B}=\sum_{C_\bullet}\lambda_{C_\bullet}$ sums over the subset of chains $C_\bullet$ that contain $A$ immediately followed by $B$. Such chains consist of a chain $C_\bullet'=\{C_0<\ldots<A\}$ from $\emptyset=C_0$ to $A$, followed by a chain $C_{\bullet}''=\{B<\ldots<M\}$ from $B$ to $M$. By \Cref{lem:cyc-interval}, the latter chains are in bijection with the chains of cyclic flats of the quotient matroid $M/B$. This lattice isomorphism $[B,M]_{\LCF_M}\cong \LCF_{M/B}$ furthermore implies that
    \begin{equation*}
        \mu_{\LCF_M}(C_i,C_{i+1})=\mu_{\LCF_{M/B}}(C_i/B,C_{i+1}/B)
    \end{equation*}
    for any nested pair $C_i<C_{i+1}$ of cyclic flats of $M$ that contain $B\leq C_i$. Together with \Cref{lem:MF(chain)}, we thus find
    \begin{equation*}
        \lambda_{A,B} = \sum_{\substack{C'_\bullet\in\chains(\LCF_A)\\C'_\bullet\neq 1_\chains}} \left(-\mu_{\chains(\LCF_A)}(C'_\bullet,1_\chains)\right)
        \cdot \left(-\mu_{\LCF_M}(A,B)\right)
        \cdot
        \sum_{\substack{C''_\bullet\in\chains(\LCF_{M/B})\\C''_\bullet\neq 1_\chains}} \left(-\mu_{\chains(\LCF_{M/B})}(C''_\bullet,1_\chains)\right)
    \end{equation*}
    which is equal to $-\mu_{\LCF_M}(A,B)$, because the sum over $C_\bullet'$ gives $\mu_{\chains(\LCF_A)}(1_\chains,1_\chains)=1$ and similarly the sum over $C_\bullet''$. So far, we have shown that
    \begin{equation*}
        (-1)^{\cc(M)-1}\FP_1(M)=1-n+\sum_{\substack{A,B\in\LCF_M\\ A<B}} \mu_{\LCF_M}(A,B)(\loops(A)-\loops(B))\rk(B).
    \end{equation*}
    We can extend the summation range to include equality $A=B$, since then the summand vanishes. The first term of $(\loops(A)-\loops(B))\rk(B)=\loops(A)\rk(B)-\loops(B)\rk(B)$ gives the claimed formula. The other term contributes
    \begin{equation*}
        \sum_{B\in\LCF_M} \loops(B)\rk(B) \sum_{\substack{A\in\LCF_M \\ A\leq B}}\mu_{\LCF_M}(A,B)
        =\sum_{B\in\LCF_M} \loops(B)\rk(B) \begin{cases}
            0, & \text{if $B\neq 0_{\LCF_M}$,} \\
            1, & \text{if $B= 0_{\LCF_M}$,} \\
        \end{cases}
    \end{equation*}
    by the definition of the M\"obius function. The minimal element of the lattice of cyclic flats is $0_{\LCF_M}=\emptyset$, because $M$ has no loops, so this sum reduces to $\loops(\emptyset)\rk(\emptyset)=0$.
\end{proof}

This formula can be rewritten in many ways; in particular, we can write it as a sum over the Boolean lattice $\LP_M=2^M$ consisting of \emph{all} subsets (of the ground set) of $M$. The resulting expression has many more terms, but it foregoes the non-trivial M\"obius function of the lattice of cyclic flats with the simple M\"obius function
\begin{equation*}
    \mu_{\LP_M}(A,B)=(-1)^{\abs{B}-\abs{A}}
\end{equation*}
of the Boolean lattice. To achieve this rewriting, we use an elementary relation between the M\"obius function of a lattice and a sublattice, see \cite[Theorem~1]{Crapo:MoeInvLat} or \cite[\S3.1.4]{KungRotaYan:RotaWay}:
\begin{theorem}\label{thm:lattice-closure}
    Let $\mathcal{L}$ be a finite lattice and $\cl\colon \mathcal{L}\rightarrow\mathcal{L}$ a \emph{closure operator}, i.e.\ a projection $A\mapsto\cl(A)=\cl(\cl(A))$ onto a sublattice $\mathcal{F}=\{A\in\mathcal{L}\colon \cl(A)=A\}$ that is increasing, $A\leq\cl(A)$, and order-preserving, $A\leq B\Rightarrow \cl(A)\leq\cl(B)$.
    
    Then for any elements $A\leq B$ of $\mathcal{L}$ with $B=\cl(B)$ (in other words, $B\in\mathcal{F}$), we have
    \begin{equation}\label{eq:lattice-closure}
        \sum_{\substack{X\in\mathcal{L} \\ A\leq X, \cl(X)=B}} \mu_{\mathcal{L}}(A,X)
        = \begin{cases}
            0, &\text{if $A\notin \mathcal{F}$,} \\
            \mu_{\mathcal{F}}(A,B), &\text{if $A\in\mathcal{F}$.} \\
        \end{cases}
    \end{equation}
\end{theorem}
\begin{proposition}\label{lem:N1(M)=bool}
    For any matroid $M$ on $n$ elements, without any loops or coloops,
    \begin{equation}\label{eq:N1(M)=bool}
        (-1)^{\cc(M)-1}\FP_1(M)=1-n+\sum_{A\subseteq B\subseteq M} (-1)^{\abs{B}-\abs{A}} \loops(A)\rk(B).
    \end{equation}
\end{proposition}
\begin{proof}
    More generally, we show that the expression
    \begin{equation*}
        \Phi(\mathcal{L})\defas \sum_{\substack{A,B\in \mathcal{L}\\A\leq B}} \mu_{\mathcal{L}}(A,B) \loops(A)\rk(B)
    \end{equation*}
    takes the same value for $\mathcal{L}$ any of the following lattices: cyclic flats $\LCF_M$, all flats $\LF_M$, and all subsets $\LP_M$. These lattices are related by closure operators. The matroid closure
    \begin{equation*}
        \cl\colon\LP_M\longrightarrow \LF_M,\quad
        A\mapsto \cl(A)=\{e\in M\colon \rk(A\cup\{e\})=\rk(A)\}
    \end{equation*}
    projects the Boolean lattice $\LP_M$ onto the lattice $\LF_M$ of flats. Since $\rk(B)=\rk(\cl(B))$, we can collect the double lattice sum over $A\leq B$ by $B'=\cl(B)$ to find
    \begin{equation*}
        \Phi(\LP_M)=\sum_{\substack{A\in\LP_M,B'\in\LF_M\\ A\leq B'}} \loops(A)\rk(B')
        \sum_{\substack{B\in\LP_M\\ A\leq B, \cl(B)=B'}} \mu_{\LP_M}(A,B)
        =\Phi(\LF_M),
    \end{equation*}
    since by \Cref{thm:lattice-closure} the sum over $B$ produces $\mu_{\LF_M}(A,B')$ and restricts the sum over $A$ to flats. Similarly, we have a projection from flats to cyclic flats, given by the operator
    \begin{equation*}
        \cyc\colon \LF_M\longrightarrow \LCF_M,\quad
        A\mapsto\cyc(A)=\{e\in A\colon \rk(A\setminus\{e\})=\rk(A) \}.
    \end{equation*}
    This is a \emph{downward} closure operator, fulfilling all the axioms of a closure operator but with $\cyc(A)\leq A$ instead of $A\leq\cl(A)$. The corresponding version of \eqref{eq:lattice-closure} is that for all $A\in\mathcal{F}\defas\{A\in\mathcal{L}\colon \cyc(A)=A\}$ and $B\in\mathcal{L}$ with $A\leq B$,
    \begin{equation*}
        \sum_{\substack{X\in\mathcal{L} \\ \cyc(X)=A, X\leq B}} \mu_{\mathcal{L}}(X,B)
        = \begin{cases}
            0, &\text{if $B\notin\mathcal{F}$,} \\
            \mu_{\mathcal{F}}(A,B), &\text{if $B\in\mathcal{F}$.} \\
        \end{cases}
    \end{equation*}
    Since $\loops(A)=\loops(A')$ for $A'=\cyc(A)$, we can collect the sum over $A$ in $\Phi$ by $A'$ to conclude
    \begin{equation*}
        \Phi(\LF_M)
        = \sum_{\substack{A'\in\LCF_M,B\in\LF_M \\ A'\leq B}} \loops(A')\rk(B) \sum_{\substack{A\in\LF_M\\ \cyc(A)=A',A\leq B}} \mu_{\LF_M}(A,B)
        = \Phi(\LCF_M),
    \end{equation*}
    since the sum over $A$ produces $\mu_{\LCF_M}(A',B)$ and restricts the sum over $B$ to cyclic flats $B=\cyc(B)$. This finishes the proof that $\Phi(\LP_M)=\Phi(\LCF_M)$, so \eqref{eq:N1(M)=cyc} and \eqref{eq:N1(M)=bool} agree.
\end{proof}
The sum of $(-1)^{\abs{A}}\loops(A)$ can be expressed in terms of Crapo's beta invariant, as in the proof of \Cref{thm:cc-AB}. The result of \Cref{lem:N1(M)=bool} can therefore also be stated as
\begin{equation*}
    (-1)^{\cc(M)-1}\FP_1(M)=1-\sum_{B\subseteq M} (-1)^{\loops(B)}\beta(B)\rk(B).
\end{equation*}
This proves \Cref{prop:N1asbeta}, and combined with \Cref{thm:cc-AB}, we obtain \Cref{thm:gprime=cc}.

\subsection{Recursive computation}
\label{sec:algrec}

We exploit the recursive structure of \Cref{thm:Speyer-from-cyclics} that arises from grouping the chains $C_\bullet=\cdots<C_{k-1}<C_k=M$ by to their penultimate element. The chains with $C_{k-1}=A$ fixed are in bijection with chains of the lattice of cyclic flats of $A$. We can therefore write
\begin{equation*}
    g_M(t)=(-1)^{\cc(M)-1} \sum_{\substack{A\in\LCF_M\\ A<M}} (-\mu_{\LCF_M}(A,M)) \sum_{\substack{C_\bullet\in\chains(\LCF_A)\\ C_\bullet< 1_{\chains}}} \lambda_{C_\bullet} g_{\SM{C_\bullet<M}}(t)
\end{equation*}
where we used \Cref{lem:MF(chain)} and we denote by $C_\bullet<M$ the extension of the chain $C_\bullet$ (a chain of cyclic flats ending in $A$) by the element $M$. In order to exploit this structure, we need a recursion for the Speyer polynomial of the Schubert matroid $\SM{C_\bullet<M}$ of this extended chain in terms of data associated with the shorter chain $C_\bullet$.
\begin{definition}
    For a lattice path $P$ from $(0,0)$ to $(\loops,r)$, let $\admsq{P}\subset\Z^2$ denote the set of pairs $(x,y)$ with $1\leq x<\loops$ and $1\leq y<r$, that lie strictly below the path $P$. Define
    \begin{equation*}
        \FP_i^{<k}(\LPM{P})
        \defas \abs{\left\{(x_1,y_1),\ldots,(x_i,y_i)\in\admsq{P}\colon x_1<\ldots<x_i\ \text{and}\ y_1<\ldots<y_i<k\right\}}
    \end{equation*}
    for all integers $i,k\geq 0$. For a Schubert matroid, define $\FP_i^{<k}(\SM{C_\bullet})$ via its lattice path representation \eqref{eq:path-from-cyc}. For an arbitrary matroid without loops or coloops, set
    \begin{equation*}
        \bar{g}_{M}^{<k}(t) \defas
        \sum_{\substack{C_\bullet\in\chains(\LCF_M) \\ C_\bullet\neq 1_\chains}}
        \Big(-\mu_{\chains(\LCF_M)}(C_\bullet,1_\chains)\Big)
        \sum_{i=0}^{\rk(M)-1} t^i
        \FP_i^{<k}(\SM{C_\bullet}).
    \end{equation*}
\end{definition}
\begin{remark}
    For a lattice path matroid $M=\LPM{P}$ on $\{1,\ldots,n\}$ with rank $r$, i.e.\ a path $P$ with $r$ north steps and $n$ steps in total, we can interpret $\FP_i^{<k}(M)$ for $k<r$ in terms of the $(r-k)$-fold \emph{truncation} matroid $\trM^{r-k}(M)$. This is the matroid whose independent sets are those independent sets of $M$ that have size at most $k$. Deleting the last $r-k$ elements of the ground set of the truncation matroid amounts to chopping off all squares above $y=k$ in the lattice path picture. Therefore,
    \begin{equation*}
        \FP_i^{<k}(M)=\FP_i\big((\trM^{r-k}(M))|_{\{1,\ldots,n-r+k\}}\big).
    \end{equation*}
\end{remark}
The numbers $\FP_i^{<k}(M)$ count the Delannoy paths with $i$ diagonal steps under the additional constraint that only steps with $y_i$ below $k$ are allowed. This extra condition becomes vacuous for $k\geq \rk(M)$, hence by \eqref{eq:FP-squares} we recover the coefficients
\begin{equation*}
    \FP_i^{<\rk(M)}(M)=\FP_i(M)
\end{equation*}
of the Speyer polynomial of a Schubert matroid in the expansion \eqref{eq:gexpand1}. Therefore, \Cref{thm:Speyer-from-cyclics} says that for any matroid $M$ without loops or coloops,
\begin{equation*}
    g_M(t)=(-1)^{\cc(M)-1} \cdot t \cdot \bar{g}_M^{<\rk(M)}(1+t).
\end{equation*}
The point of introducing the additional parameter $k$ into $\bar{g}_M^{<k}(t)$ is that it tracks enough information to allow for a recursive computation, see \Cref{alg:gRec}. For this recursion, it is convenient to abbreviate the following polynomials:
\begin{equation*}
    Q^{(1)}_{r,l}(t) = \sum_{i=0}^{\min\{r,l\}} \binom{r}{i}\binom{l}{i} t^i \quad\text{and}\quad
    Q^{(2)}_{r,l}(t) = \sum_{i=1}^{\min\{r+1,l\}} \binom{r}{i-1}\binom{l}{i} t^i.
\end{equation*}
\begin{algorithm}
    \caption{gRecursive($M$, $k$)}%
    \label{alg:gRec}%
    \KwIn{non-empty matroid $M$ without loops or coloops, and integer $k$}
    \KwOut{polynomial $g_M^{<k}(t)\in\Z[t]$}
    $acc \gets  -\mu_{\LCF_M}(\emptyset,M) \cdot Q^{(1)}_{k-1,\loops(M)-1}$\;
    \ForEach{$A\in\LCF_M\setminus\{\emptyset,M\}$}{
        \eIf{$k\leq \rk(A)$}{
            $acc \gets acc -\mu_{\LCF_M}(A,M)\cdot \gRec{A,k}$\;
        }
        {
            $acc \gets acc -\mu_{\LCF_M}(A,M)\cdot \gRec{A,$\rk(A)$}\cdot Q^{(1)}_{k-\rk(A),\loops(M)-\loops(A)-1}$\;
        }
        \For{$k'=1$ \KwTo $\min\{k-1,\rk(A)-1\}$}{
            $acc \gets acc -\mu_{\LCF_M}(A,M)\cdot \gRec{A,k'}\cdot Q^{(2)}_{k-1-k',\loops(M)-\loops(A)}$\;
        }
    }
    \Return{$acc$}
\end{algorithm}
\begin{proposition}
    \Cref{alg:gRec} computes the function $\bar{g}_M^{<k}(t)$.
\end{proposition}
\begin{figure}
    \centering
    \includegraphics[scale=1]{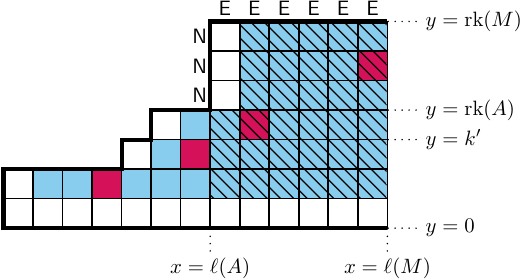}
    \caption{The lattice path $P=\PU^2\PR^4\PU\PR\PU\PR^2\PU^3\PR^6$ with its admissible squares $\admsq{P}$ highlighted in blue. The shaded region indicates $\admsq{P}\setminus\admsq{P'}$. The red squares illustrate a configuration $S$ that contributes to $\FP_4(\LPM{P})$.}%
    \label{fig:latticepath-truncation}%
\end{figure}
\begin{proof}
Let $P$ denote the lattice path corresponding to a cyclic flat $C_\bullet$ ending in $M$ and with penultimate element $A$. Then by \eqref{eq:path-from-cyc}, $P=P'\PU^{\rk(M)-\rk(A)}\PR^{\loops(M)-\loops(A)}$ where $P'$ denotes the lattice path corresponding to the truncated chain stopping already at $A$. The admissible squares (internal lattice points) of $P$ consist of those of $P'$ and an additional rectangle---except for the lattice points on the north steps---see \cref{fig:latticepath-truncation}:
\begin{align*}
    \admsq{P}
    =\admsq{P'} &\sqcup \{\loops(A),\ldots,\loops(M)-1\}\times\{1,\ldots,\rk(A)-1\} \\
    &\sqcup \{\loops(A)+1,\ldots,\loops(M)-1\}\times\{\rk(A),\ldots,\rk(M)-1\}.
\end{align*}
Therefore, we can partition every collection $S=\{(x_1,y_1),\ldots,(x_i,y_i)\}\subseteq \admsq{P}$ of lattice points contributing to $\FP_i^{<k}(\LPM{P})$ into two subsets $S'=S\cap\admsq{P'}$ and $S''=S\setminus\admsq{P'}$. The cases with $S''=\emptyset$ are counted by $\FP_i^{<k}(\LPM{P'})$. In the other cases, let $k'$ denote the smallest $y$-coordinate of the points in $S''$. All $(x_j,y_j)\in S'$ are to the left of $S''$, and thus, by definition of $\FP_i^{<k}(\LPM{P})$, also below $S''$, i.e.\ $y_j<k'$. For $k\leq\rk(A)$, we thus have
\begin{multline*}
    \FP^{<k}_i(\LPM{P})
    =\FP^{<k}_i(\LPM{P'})+ \sum_{\substack{i=i'+i''\\ i''>0}} \FP^{<k'}_{i'}(\LPM{P'}) \\[-5mm]
    \times \sum_{k'=1}^k \binom{k-1-k'}{i''-1}\binom{\loops(M)-\loops(A)}{i''}
\end{multline*}
where the binomials count the possibilities for choosing the $i''-1$ remaining $y$-coordinates (apart from $k'$) and the $i''$ many $x$-coordinates of the points in $S''$. When $k>\rk(A)$, almost the same relation holds; the only difference is that for summands with $k'>\rk(A)$, the $x$-coordinate of the point with $y_j=k'$ in $S''$ must be $x_j\geq \loops(A)+1$, because the point $(\loops(A),k')$ lies on the path $P$ hence does not belong to $\admsq{P}$. Therefore, when $k'>\rk(A)$, the second binomial factor is to be replaced by $\binom{\loops(M)-\loops(A)-1}{i''}$, and also we then have
\begin{equation*}
    \FP_{i'}^{<k'}(\LPM{P'})=\FP_{i'}^{<\rk(A)}(\LPM{P'}).
\end{equation*}
The \Cref{alg:gRec} is nothing but the implementation of this recursion, by grouping the chains of cyclic flats of $M$ by their penultimate element $A$, expressed in terms of the polynomials $\bar{g}_M^{<k}(t)$.

The first line of the algorithm accounts for the contribution where $A=\emptyset$, that is, the trivial chain $C_\bullet=\emptyset<M$. In this case, the Schubert matroid $\SM{C_\bullet}\cong\UM{n}{r}$ is uniform with rank $r=\rk(M)$ on $n=\abs{M}$ elements. By \Cref{eq:uniform-path}, the corresponding lattice path bounds a rectangle so that $\admsq{P}=\{1,\ldots,n-r-1\}\times\{1,\ldots,r-1\}$. Picking $i$ points $(x_j,y_j)$ that contribute to $\FP_i$ is thus equivalent to choosing their $x$- and $y$-coordinates independently. Therefore,
\begin{equation*}
    \FP_i^{<k}(\SM{\emptyset<M})=\FP_i^{<k}(\UM{n}{r})=\binom{n-r-1}{i}\binom{k-1}{i}
\end{equation*}
which gives the contribution $-\mu(\emptyset,M)\cdot Q^{(1)}_{k-1,\loops(M)-1}(t)$ to $\bar{g}_M^{<k}(t)$.
\end{proof}
\begin{example}
    For a uniform matroid $M=\UM{n}{r}$ with $2\leq r\leq n-2$, the algorithm terminates immediately with just the base case. The resulting expression is different from the formula given in \cite[Proposition~10.1]{Speyer:MatroidKtheory}, but easily confirmed to be equivalent:
    \begin{equation*}
        g_{\UM{n}{r}}(t)
        =t Q^{(1)}_{r-1,n-r-1}(1+t)
        =t\sum_{i=0}^{\min\{r-1,n-r-1\}} \binom{r-1}{i}\binom{n-r-1}{i}(1+t)^i.
    \end{equation*}
\end{example}
We implemented \Cref{alg:gRec} in the {\MapleTM} computer algebra system.{\MapleNote}
This implementation is available freely as open source, see \Cref{sec:implementation} for details.
In practice, the runtime of our implementation is dominated by the calculation of all values $\mu(A,B)$ of the M\"obius function of the lattice of cyclic flats, which amounts to ennumerating all ordered triples $A<C<B$, that is, all 2-dimensional faces of the order complex of $\LCF$.

To ensure the correctness of our implementation, we computed $g_M(t)$ for 81 graphs with $\leq 13$ edges using the {\Sage} program provided with \cite{Ferroni:SchubertDelannoySpeyer}, and found agreement with the results from our program. Furthermore, for the entire data set discussed in \cref{sec:data} (consisting of more than 3 million graphs) we checked $g'_M(0)=\beta(M)$ against the computation of $\beta(M)=t_{1,0}$ from the Tutte polynomial by {\Maple}'s \texttt{GraphTheory} package, and we also confirmed that $g''_M(-1)=-1$ in all these cases. We also confirmed up to $r\leq18$ that our code correctly reproduces \cite[Proposition~10.2]{Speyer:MatroidKtheory}, namely that
\begin{equation}\label{eq:g(wheel)}
    g_{W_r}(t)=(1+t)^r-1-t-t^2
\end{equation}
for the wheel graphs $W_r$ with $2r$ edges and $r+1$ vertices, e.g.\ $W_3\cong K_4$. The runtime for $r=18$ was less than two hours on a single cpu with 1 GHz frequency. This wheel graph has $\abs{\LCF_{W_{18}}}=24915$ cyclic flats, the Hasse diagram of this lattice has \numprint{158762} edges, and there are \numprint{7070763} comparable pairs $A<B$ of cyclic flats.

In comparison, the number of chains of cyclic flats is $\abs{\chains(\LCF_{W_{18}})}=\numprint{17696253846612}$, so that it is impossible in practice to apply the algorithm from \cite[\S5]{Ferroni:SchubertDelannoySpeyer} verbatim to compute the Schubert decomposition underlying \Cref{thm:Speyer-from-cyclics} and obtain $g_{W_{18}}(t)$ that way.

\section{Calculations and conjectures}
\label{sec:data}

Since $g_M(t)$ is invariant under series-parallel operations, and multiplicative over direct sums, we only consider graphs with minimum degree 3 that are simple (no self-loops or multiedges) and biconnected.
We used the program \texttt{geng} from {\nauty} \cite{McKayPiperno:II} to enumerate isomorphism classes of such graphs. Our main data set consists of \numprint{3293662} (non-isomorphic) graphs in total, namely all those graphs that have less than 10 vertices or less than 22 edges (see \cref{tab:dataset} for details).

With our implementation of \Cref{alg:gRec}, we then computed $g_G(t)$ for all these graphs. The results are publicly available at
\begin{center}
    \url{\dataurl}
\end{center}
In addition to this complete data set of small graphs, we also computed $g_G(t)$ for some larger graphs in special families (complete, complete bipartite, regular, and circulant graphs), see \cref{sec:data}.

In the following, we discuss a number of observations from these calculations. We state these mostly as conjectures. Our inquiry focused mainly on properties of the coefficient $\FP_2(G)$, as motivated in the introduction, and we discuss the relevance of this invariant to Feynman integrals in \cref{sec:Feynman}. However, we did notice properties of the polynomial $g_G(t)$ that do not refer only to $\FP_2(G)$, and we will mention those as well.

\subsection{\texorpdfstring{Range of $\FP_2$}{Range of N\_2}}
Among the \numprint{3293662} in the data set \cref{tab:dataset}, the invariant $\FP_2(G)$ takes only 18 different values, lying between $-14$ and $+11$. The vast majority of these graphs ($94\%$) has $\FP_2(G)\in\{-1,0,1\}$, see \cref{tab:N2stats} for the precise distribution.
\begin{table}
    \centering
    \begin{tabular}{rr@{\hspace{12mm}}rr@{\hspace{12mm}}rr}
\toprule
$\FP_2(G)$ & number of $G$ & $\FP_2(G)$ & number of $G$ & $\FP_2(G)$ & number of $G$ \\
\midrule
 $-14$ & 1 & $-2$ & \numprint{18796} & 4 & \numprint{2535} \\
 $-9$ & 9 & $-1$ & \numprint{339537} & 5 & 114 \\
 $-6$ & 42 & 0 & \numprint{2008419} & 6 & 12 \\
 $-5$ & 147 & 1 & \numprint{740184} & 7 & 3 \\
 $-4$ & 155 & 2 & \numprint{147533} & 8 & 2 \\
 $-3$ & \numprint{3344} & 3 & \numprint{32826} & 11 & 3 \\
\bottomrule
    \end{tabular}
    \caption{The number of times each value of $\FP_2(G)$ appears among the graphs in \cref{tab:dataset}.}%
    \label{tab:N2stats}%
\end{table}
\begin{remark}
This behaviour is in stark contrast to Crapo's $\beta$ invariant, which takes $855$ different values for this set of graphs, ranging between $\beta(K_4)=2$ and $\beta(K_7)=5040$.
\end{remark}
Despite this sparsity of values for $\FP_2(G)$ for small graphs, this invariant can in fact take arbitrarily large values. For example, we computed $g_G(t)$ for all complete bipartite graphs with $\leq 14$ vertices, which range between $\FP_2(K_{3,11})=-44$ and $\FP_2(K_{4,10})=39$; see \Cref{tab:bipartite}. We have strong evidence (\Cref{con:g(K3n)}) that for all $n\geq 3$,
\begin{equation}\label{eq:N2(K3n)}
    \FP_2(K_{3,n}) = -\frac{n(n-3)}{2}.
\end{equation}
At each $3\leq n\leq 7$, this graph is in fact the \emph{unique} graph that minimizes the value of $\FP_2(G)$ among all graphs with $3n$ edges that are biconnected and have minimum degree 3 or higher (this is confirmed by our calculations, since our data set includes all such graphs with $21$ edges or less). The graphs $K_{3,n}$ (and their series-parallel equivalents) are thus candidates for the minimizers of $\FP_2(G)$.
\begin{remark}
    The identity \eqref{eq:N2(K3n)} would follow from the star-triangle relation \eqref{eq:N2star-triangle}: Applied to a 3-valent vertex, the latter predicts $\FP_2(K_{3,n})=\FP_2(K_{1,1,1,n-1})-(n-2)$ where $K_{1,1,1,n-1}=K_3\vee \overline{K}_{n-1}$ is the join of a triangle with $n-1$ isolated vertices. Applying \eqref{eq:N2star-triangle} again to one of these 3-valent vertices gives $\FP_2(K_{1,1,1,n-1})=\FP_2(K_{1,1,1,n-2})-(n-3)$, because the resulting doubled triangle may be replaced by a simple triangle ($\FP_2$ is invariant under parallel-operations). This induction stops at $K_{1,1,1,1}\cong K_4$ which has $\FP_2(K_4)=1$ and thus yields $\FP_2(K_{3,n})=\FP_2(K_{1,1,1,n})=-n(n-3)/2$ for all $n\geq3$.
\end{remark}
\begin{table}
    \centering
    \begin{tabular}{r@{\hspace{9mm}}lr@{\hspace{9mm}}lr@{\hspace{9mm}}lr@{\hspace{9mm}}lr@{\hspace{9mm}}lr}
    \toprule
    $V$ & $G$ & $\FP_2$ & $G$ & $\FP_2$ & $G$ & $\FP_2$ & $G$ & $\FP_2$ & $G$ & $\FP_2$ \\
    \midrule
    $6$ &  $K_{3,3}$ &   $0$ & & & & & & & & \\
    $7$ &  $K_{3,4}$ &  $-2$ & & & & & & & & \\
    $8$ &  $K_{3,5}$ &  $-5$ &  $K_{4,4}$ &  $6$ & & & & & & \\
    $9$ &  $K_{3,6}$ &  $-9$ &  $K_{4,5}$ &  $3$ & & & & & & \\
   $10$ &  $K_{3,7}$ & $-14$ &  $K_{4,6}$ & $13$ & $K_{5,5}$ & $-10$ & & & & \\
   $11$ &  $K_{3,8}$ & $-20$ &  $K_{4,7}$ & $12$ & $K_{5,6}$ &  $-4$ & & & & \\
   $12$ &  $K_{3,9}$ & $-27$ &  $K_{4,8}$ & $24$ & $K_{5,7}$ & $-19$ & $K_{6,6}$ & $20$ & & \\
   $13$ & $K_{3,10}$ & $-35$ &  $K_{4,9}$ & $25$ & $K_{5,8}$ & $-15$ & $K_{6,7}$ &  $5$ & & \\
   $14$ & $K_{3,11}$ & $-44$ & $K_{4,10}$ & $39$ & $K_{5,9}$ & $-32$ & $K_{6,8}$ & $31$ & $K_{7,7}$ & $-28$ \\
    \bottomrule
    \end{tabular}
    \caption{The invariant $\FP_2(G)$ for complete bipartite graphs with $V(G)\leq14$ vertices. We omit $K_{2,n}$, since these are series-parallel and therefore $\FP_2(K_{2,n})=0$.}%
    \label{tab:bipartite}%
\end{table}

In the other direction, we found that the circulant graphs from \cref{fig:xladders} have exponentially growing positive values of $\FP_2(G)$.
The circulant graphs $C^n_{1,k}$ are obtained from the cycle graph on vertices $\{1,\ldots,n\}$ by adding $n$ edges $\{i,i+k\}$, where $1\leq i\leq n$ and vertex labels are identified modulo $n$. The cases $1<k<n/2$ give rise to families of 4-regular graphs, which we computed up to $n\leq 20$.
We calculated $g_G(t)$ for all these graphs with up to 20 vertices, confirming the following conjecture for $n\leq 10$:
\begin{conjecture}\label{con:N2(xladder)}
    For every integer $n\geq 5$, we have $\FP_2\big(C^{2n}_{1,n-1}\big)=2^{n-1}-n$.
\end{conjecture}
The first instance, $\FP_2(C^{10}_{1,4})=11$, is in fact the \emph{unique} graph that maximizes $\FP_2(G)$ among all graphs with 20 edges (or less); so this family of circulants (and its series-parallel descendants), with its exponentially growing value of $\FP_2$, is a candidate for the maximizers more generally. Hyperlinks into the data base \cite{HoG} for the first few of these graphs are $C^{10}_{1,4}=\HoG{45705}$, $C^{12}_{1,5}=\HoG{32806}$, $C^{14}_{1,6}=\HoG{51911}$, $C^{16}_{1,7}=\HoG{51913}$, $C^{18}_{1,8}=\HoG{51914}$, and $C^{20}_{1,9}=\HoG{51916}$.

\subsection{Closed forms}
The only formula in the literature for the $g$-polynomial of an infinite family of (non-series-parallel) graphs is \eqref{eq:g(wheel)} for the wheel graphs, from \cite[Proposition~10.2]{Speyer:MatroidKtheory}. In this section, we propose several more closed formulas, supported by our calculations, for families of graphs of particular interest.

\begin{figure}
    \centering
    $
    \Graph[0.8]{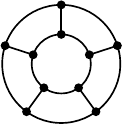} \quad \Graph[0.8]{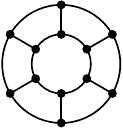} \quad \Graph[0.8]{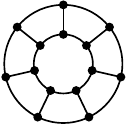}$
    \qquad\qquad
    $
    \Graph[0.8]{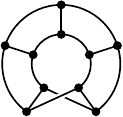} \quad \Graph[0.8]{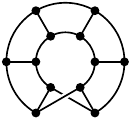} \quad \Graph[0.8]{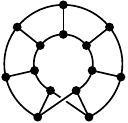}$
    \caption{The prism graphs $K_2\times C_n$ (left) and M\"obius ladders $C^{2n}_{1,n}$ (right) for $n=5,6,7$.}%
    \label{fig:moebius-ladders}%
\end{figure}
The prisms $K_2\times C_n$ and the M\"obius ladders $C^{2n}_{1,n}$ on $2n$ vertices are two well-known families of 3-regular graphs, illsutrated in \cref{fig:moebius-ladders}. Their Speyer polynomials are particularly simple: almost all coefficients $\FP_i$ vanish. Our calculations confirm the following for $n\leq 10$:
\begin{conjecture}\label{con:prism-moebius}
    For every integer $n\geq 2$, the prisms and M\"obius ladders have
    \begin{align*}
        g_{K_2\times C_n}(t) &= t\Big(1+(1+t)^2+(2^n-n-3)(1+t)^n\Big)\quad\text{and} \\
        g_{C^{2n}_{1,n}}(t) &= t\Big(1+(2^n-n-1)(1+t)^n\Big).
    \end{align*}
\end{conjecture}
In particular, we have $\FP_2(K_2\times C_n)=1$ (as required by \Cref{con:planar=1} since these graphs are planar) and $\FP_2(C^{2n}_{1,n})=0$ for $n\geq 3$. The prisms play an extremal role among 3-regular graphs (see \cref{sec:Feynman}), but not for the invariant $\FP_2$. Taking derivatives at $t=0$, our conjecture reproduces correctly $\beta(K_2\times C_n)=2^n-n-1$ and $\beta(C^{2n}_{1,n})=2^n-n$ which can be read off from the chromatic polynomials in \cite[\S4]{BiggsDamerellSands:RecFam}.
\begin{figure}
    \centering
    $\Graph[0.7]{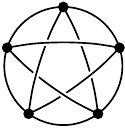}$ \quad $\Graph[0.7]{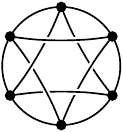}$ \quad $\Graph[0.7]{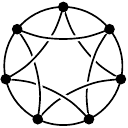}$ \quad $\Graph[0.7]{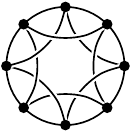}$ \quad $\Graph[0.7]{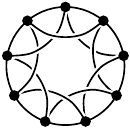}$
    \caption{The circulant graphs $C^n_{1,2}$ on $n=5,6,7,8,9$ vertices.}%
    \label{fig:zigzags}%
\end{figure}

Among 4-regular graphs, the circulants $C^n_{1,2}$ (see \cref{fig:zigzags}) play a special role (see \cref{sec:Feynman}). Our calculations confirm the following closed form for all $n\leq 20$:
\begin{conjecture}\label{con:zigzag}
    For all integers $n\geq 5$, the circulant graph $C^n_{1,2}$ has
\begin{equation*}
    g_{C^{n}_{1,2}}(t)=t+\sum_{i=1+\lfloor n/2\rfloor}^{n-2} \frac{n}{i}\binom{i}{n-i} t(1+t)^i
    +\begin{cases}
        t(1+t)^2 & \text{if $n$ is even,}\\
        0        & \text{if $n$ is odd.} \\
    \end{cases}
\end{equation*}
\end{conjecture}
In particular, these circulants have $\FP_2(C^n_{1,2})=1$ for even $n$ (these cases are planar, so this is required by \Cref{con:planar=1}) and $\FP_2(C^n_{1,2})=0$ for odd $n$. We confirmed that the derivative at $t=0$ correctly reproduces $\beta(C^n_{1,2})=\varphi^n+(1-\varphi)^n-n-(1+(-1)^n)/2$ with $\varphi=(1+\sqrt{5})/2$. To obtain this closed form, we specialized the recurrences for the chromatic polynomial in \cite{Whitehead:ChordedCycles} to $\beta$, found $-1,1,1,\varphi,1-\varphi$ for the spectrum of the corresponding transfer matrix, and used the first five values to fix initial conditions.
\begin{remark}
    In the examples of \Cref{con:prism-moebius} and \Cref{con:zigzag}, all coefficients $\FP_i(G)$ are non-negative and in fact many of them vanish. We currently lack graph-theoretic criteria to explain why this happens. If the spanning tree polytopes of these families of graphs decompose into direct sums of series-parallel matroid polytopes, \cite[Conjecture~4.6]{Ferroni:SchubertDelannoySpeyer} would predict non-negativity. We have not investigated if this is the case.
\end{remark}
The join $G\vee H$ of two graphs is obtained from the union $G\sqcup H$ by adding all edges connecting vertices from $G$ with vertices from $H$. For example, the complete bipartite graphs are joins $K_{n,m}=\overline{K}_n\vee\overline{K}_m$ of the edgeless graphs with $n$ and $m$ vertices, and the complete graph $K_n\cong K_1\vee\ldots\vee K_1$ is an $n$-fold join of the 1-vertex graph.
\begin{figure}
    \centering
    $\Graph[0.5]{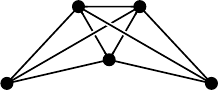}$ \quad $\Graph[0.5]{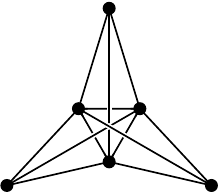}$ \quad $\Graph[0.5]{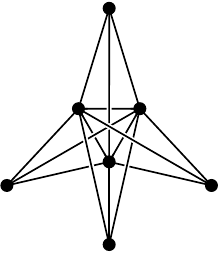}$ \quad $\Graph[0.5]{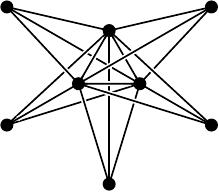}$ \quad $\Graph[0.5]{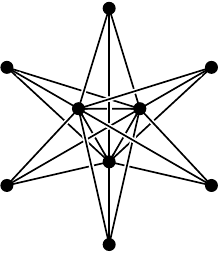}$
    \caption{The 4-partite graphs $K_{1,1,1,n}\cong K_3\vee\overline{K}_n$ from $n=2$ (left) to $n=6$ (right).}%
    \label{fig:K111n}%
\end{figure}
\begin{conjecture}\label{con:g(K3n)}
    For all integers $n\geq 3$, the complete bipartite graph $K_{3,n}=\overline{K}_3\vee\overline{K}_n$ and the complete 4-partite graph $K_{1,1,1,n}\cong K_3\vee \overline{K}_n$ (\cref{fig:K111n}) have the Speyer polynomials
\begin{equation*}\begin{aligned}
    g_{K_{3,n}}(t) &= t(nt^2+nt+t+2)(2+t)^{n-1} -3t(1+t)^{n+1} \quad\text{and}\\
    g_{K_{1,1,1,n}}(t) &= t(nt^2+nt+t+2)(2+t)^{n-1}. \\
\end{aligned}\end{equation*}
\end{conjecture}
We confirmed this conjecture for all $n\leq 60$, using a program that implements \Cref{alg:gRec} specifically for these graph families and that exploits concrete descriptions of the cyclic flats in terms of integer partitions and closed formulas for the M\"obius function.\footnote{Apart from $\emptyset$, the lattice $\LCF(K_{3,n})$ consists of one copy of $K_{3,i}$ and three copies of $K_{2,i}$ for each subset of $2\leq i\leq n$ vertices from the $K_n$ side. Boolean sublattices show $\mu(K_{2,i},K_{2,n})=\mu(K_{3,i},K_{3,n})=(-1)^{n-i}$ and product lattices yield $\mu(K_{2,i},K_{3,n})=-(-1)^{n-i}$. From the defining recursion one finds thus $\mu(\emptyset,K_{2,n})=(1-n)\cdot (-1)^n$ and $\mu(\emptyset,K_{3,n})=2\cdot(n-1)\cdot(-1)^n$.}
The derivative at $t=0$ correctly reproduces $\beta(K_{3,n})=2^n-3$, which follows from \cite{Swenson:ChromKnm}.

\Cref{con:g(K3n)} implies the quadratic growth of $\FP_2(K_{3,n})=\FP_2(K_{1,1,1,n})=-n(n-3)/2$ mentioned previously in \eqref{eq:N2(K3n)}.

All graph families considered so far are sparse. In the dense regime, the complete graphs are of particular interest. Using a dedicated implementation of \Cref{alg:gRec} that takes advantage of the description of the cyclic flats of $K_n$ in terms of partitions of the set $\{1,\ldots,n\}$ of vertices (without parts of size two), we confirmed for all $n\leq 40$:
\begin{conjecture}\label{con:g(Kn)}
    For every integer $n\geq 5$, Speyer polynomials of complete graphs satisfy
    \begin{equation*}
    g_{K_n}(t)=(n-2+t(n-3))g_{K_{n-1}}(t)+t(1+t)g'_{K_{n-1}}(t) + (1+t)(t-n+3)g_{K_{n-2}}(t)
    .
\end{equation*}
\end{conjecture}
With this recursion and the initial values $g_{K_3}(t)=t$ and $g_{K_4}(t)=2t+2t^2+t^3$, any $g_{K_n}(t)$ can be computed rapidly. Taking the derivative at $t=0$, the recursion becomes $\beta(K_n)=(n-1)\beta(K_{n-1})-(n-3)\beta(K_{n-2})$ which is compatible with $\beta(K_n)=(n-2)!$ from \cite[Proposition~9]{Crapo:HigherInvariant}. For the coefficients $\FP_i(K_n)$ in \eqref{eq:gexpand1}, the recursion means
\begin{equation}\label{eq:N(Kn)-rec}\begin{aligned}
    \FP_i(K_n)&=-(i-1)\FP_i(K_{n-1})+(n+i-3)\FP_{i-1}(K_{n-1})
    \\ & \quad
    -(n-2)\FP_{i-1}(K_{n-2})+\FP_{i-2}(K_{n-2}).
\end{aligned}\end{equation}
Setting $i=2$ it follows that $\FP_2(K_n)=1$ for $n$ even and $\FP_2(K_n)=0$ for $n$ odd. The higher coefficients grow rapidly, see \cref{tab:gKn}; for example $\FP_3(K_n)=(2^n-(-1)^n)/3-n-1$.
\begin{table}
    \centering
    \begin{tabular}{rrrrrrrrrr}
\toprule
$M$ & $\FP_0$ & $\FP_1$ & $\FP_2$ & $\FP_3$ & $\FP_4$ & $\FP_5$ & $\FP_6$ & $\FP_7$ & $\FP_8$ \\
\midrule
$K_{3}$ & $1$ & $0$ & & & & & & &\\
$K_{4}$ & $1$ & $0$ & $1$ & & & & & &\\
$K_{5}$ & $1$ & $0$ & $0$ & $5$ & & & & &\\
$K_{6}$ & $1$ & $0$ & $1$ & $-14$ & $36$ & & & &\\
$K_{7}$ & $1$ & $0$ & $0$ & $35$ & $-245$ & $329$ & & &\\
$K_{8}$ & $1$ & $0$ & $1$ & $-76$ & $1135$ & $-3996$ & $3655$ & &\\
$K_{9}$ & $1$ & $0$ & $0$ & $161$ & $-4410$ & $30219$ & $-68775$ & $47844$ &\\
$K_{10}$ & $1$ & $0$ & $1$ & $-330$ & $15610$ & $-182952$ & $769825$ & $-1283150$ & $721315$\\
\bottomrule
    \end{tabular}
    \caption{Coefficients of the Speyer polynomials $g_M(t)=t\sum_i \FP_i(M) (1+t)^i$ for the cycle matroids of complete graphs $K_n$ on $n$ vertices (braid matroids of rank $n-1$).}%
    \label{tab:gKn}%
\end{table}
The coefficent of $t^{\rk(M)}$ in $g_M(t)$ was coined ``$\omega$-invariant'' in \cite{FinkShawSpeyer:Omega}. For the complete graphs, \eqref{eq:N(Kn)-rec} gives the following recursion for $\omega(K_n)=\FP_{n-2}(K_n)$:
\begin{equation*}
    \omega(K_n)=(2n-5)\omega(K_{n-1})+\omega(K_{n-2}).
\end{equation*}
More information on this sequence can be found in \cite[\oeis{A278990}]{OEIS}.

\begin{remark}
    Our formulas for Speyer's polynomial of $K_n$ might be useful to address \cite[Conjecture~7.2]{FerroniFink:PolMat} via the invariants $\FP_i$. That conjecture claims that braid matroids extremize \emph{some} (co-)valuative invariant.
\end{remark}

\subsection{Planarity}
Our data set of small graphs (\cref{tab:dataset}) contains \numprint{208151} graphs that are planar and 3-connected. They all have $\FP_2(G)=1$, supporting \Cref{con:planar=1}. Further evidence comes from \Cref{con:3sum} (about the 3-sum), because already its special case \eqref{eq:N2star-triangle} (the star-triangle identity) would imply \Cref{con:planar=1}. This can be seen as follows:

The proof of Steinitz's theorem given in \cite[\S13.1]{Gruenbaum:ConvexPolytopes} shows that every 3-connected planar simple graph $G_0$ can be reduced to $G_n=K_4$ by a suitable sequence of moves $G_0\mapsto G_1\mapsto G_2\mapsto \ldots$ where each move consists of either a star-triangle operation followed by simplification of any parallel edges thus created (if any), or of a triangle-star operation followed by series reduction of any 2-valent vertices thus created (if any). The construction ensures that each $G_i$ along the sequence is itself a 3-connected, planar, simple graph. Since $\FP_2$ is invariant under the series and parallel simplifications, $\FP_2(G_0)=1$ thus follows inductively from $\FP_2(K_4)=1$ and the star-triangle identities $\FP_2(G_i)=\FP_2(G_{i+1})$. The latter are implied by \eqref{eq:N2star-triangle} because for any 3-vertex cut $S$ of $G_i$ we must have $\abs{\pi_0(G_i\setminus S)}=2$. For if $G_i\setminus S=C_1\sqcup\ldots\sqcup C_r$ would have $r\geq 3$ connected components, then contracting each of these components to a single vertex produces a complete bipartite graph $K_{3,r}$ is a minor of $G_i$---which would contradict planarity.\footnote{Each $C_i$ must have edges to each of the 3 vertices $S$, for otherwise, $G_i$ would not be 3-connected.}

The next invariant $\FP_3(G)$ is not trivial for 3-connected planar graphs. For the \numprint{208151} such graphs in our data set, it takes only few values: $\FP_3(G)\in\{-4,-3,-2,-1,0,1,2\}$.
\begin{conjecture}\label{con:N3=0}
    If $G$ is planar and 4-connected, then $\FP_3(G)=0$.
\end{conjecture}
The support for this conjecture from our data set of small graphs (\Cref{tab:dataset}) alone is weak, because this set contains merely 26 planar graphs that are 4-connected. To seriously probe \Cref{con:N3=0}, we therefore used the program \texttt{plantri} \cite{BrinkmannMcKay:plantri} to generate all 4-connected planar graphs with 13 vertices or less. There are \numprint{21078} such graphs, and for all of them, we found $\FP_3(G)=0$.

\subsection{Vertex cuts}\label{sec:vertex-identities}
For each of the small graphs (\cref{tab:dataset}), we computed all 3-sum decompositions and all twists. In each case, \Cref{con:3sum} and \Cref{con:twist} are satisfied. Let us note two consequences of these conjectures:
\begin{enumerate}
    \item[(1)] Let $G$ be a 3-connected graph and $S$ a 3-vertex cut of $G$. Then $\FP_2(G)=\FP_2(G\setminus H)$ where $H\subset E(G)$ denotes all edges of $G$ with both endpoints in $S$.
    \item[(2)] Let $G$ be a 3-connected graph and $S$ a minimal 4-vertex cut of $G$. Suppose that $G$ has an edge $e=\{v_1,v_3\}$ with both endpoints in $S=\{v_1,v_2,v_3,v_4\}$, and set $e'=\{v_2,v_4\}$. Then $\FP_2(G)=\FP_2( (G\setminus \{e\})\sqcup\{e'\})$. In particular, if $e'$ was already present in $G$, then $\FP_2(G)=\FP_2(G\setminus\{e\})$.
\end{enumerate}
Property (1) follows from \Cref{con:3sum}: Any decomposition of $G\setminus H=G_1\oplus_3 G_2$ induced by $S$ also gives a decomposition $G=(G_1\sqcup H)\oplus_3 G_2$, where $\FP_2(G_1)=\FP_2(G_1\sqcup H)$ because adding $H$ amounts to adding edges in parallel to the edges of the gluing triangle (recall that $\FP_2$ is invariant under parallel operations). For example, this explains why
\begin{equation*}
    \FP_2(K_{1,1,1,n})=\FP_2(K_{3,n})
\end{equation*}
are equal, as observed in \Cref{con:g(K3n)}.

Property (2) follows from \Cref{con:twist}, because we can assign an edge with both ends in $S$ (like $e$) to either side of a vertex bipartition $E(G)=A\sqcup B$. Twisting first with $e\in B$ (which in particular turns $e$ into $e'$), we can then apply the same twist again, only with $e'$ now considered part of $A$. After these two twists, the graph is back almost in its original configuration---only the edge $e$ got replaced by $e'$. Equivalently, we could view this operation as a single twist for the edge bipartition $A=E(G)\setminus\{e\}$ and $B=\{e\}$. If the edge $e'$ was present in $G$ to begin with, then this creates a double edge (like in \cref{fig:K5K5-twist}) and thus $\FP_2(G)=\FP_2(G\setminus \{e\})$.

This observation explains several simplifications. For example, a subgraph consisting of two adjacent squares, connected to the rest of the graph at four vertices $v_i$ with $v_1$ and $v_3$ being 3-valent, can be replaced by a single square according to
\begin{equation*}
    \Graph{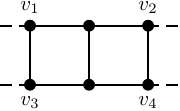} \quad\mapsto\quad
    \Graph{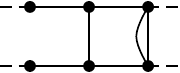} \quad\mapsto\quad
    \Graph{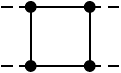}
\end{equation*}
without changing $\FP_2$. Note that $v_1$ and $v_3$ become 2-valent after twisting $e=\{v_1,v_3\}$ to $e'=\{v_2,v_4\}$, hence they can be eliminated by series reductions. Applied to the prism and M\"obius ladder graphs from \cref{fig:moebius-ladders}, Property (2) thus explains why for all $n\geq 4$,
\begin{equation*}
    \FP_2(C^{2n}_{1,n})=\FP_2(C^{2n-2}_{1,n-1})
    \qquad\text{and}\qquad
    \FP_2(K_2\times C_n)=\FP_2(K_2\times C_{n-1})
\end{equation*}
are independent of $n$---in agreement with \Cref{con:prism-moebius}. Similarly, we can replace a strip of four triangles by only two consecutive triangles, which explains the identities $\FP_2(C^n_{1,2})=\FP_2(C^{n-2}_{1,2})$ for the circulants from \Cref{con:zigzag}:
\begin{align*}
    &\Graph{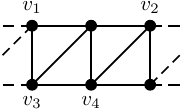} \quad\mapsto\quad
    \Graph{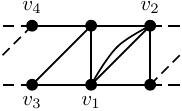} \quad\mapsto\quad
    \Graph{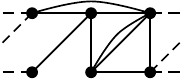} \\
    \mapsto\quad&
    \Graph{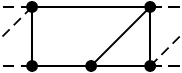} \quad\mapsto\quad
    \Graph{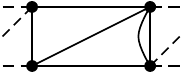} \quad\mapsto\quad
    \Graph{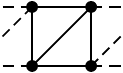}
\end{align*}
Here we first applied Property (2) twice to move edges, then applied a parallel reduction, then a star-triangle transform, and finally another parallel reduction.

\subsection{Edge cuts}\label{sec:edge-identities}
\begin{figure}
    \centering
    $G_1\colon\ \verb|H?]RCNo| = \Graph[0.8]{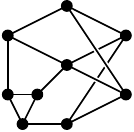}$ \quad $G_2\colon\  \verb|H_l@Gno| = \Graph[0.8]{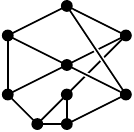}$ \quad $A\colon\  \verb|F_lv_| = \Graph[0.8]{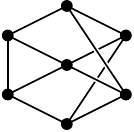}$
    \caption{The non-isomorphic graphs ($G_1$ and $G_2$) with a non-trivial 3-edge cut, leading (in both cases) to the same two graphs: $K_4$ (from the triangle side) and $A$.}%
    \label{fig:3edge-g}%
\end{figure}
We confirmed that \Cref{con:3edge} holds for all 3-edge cuts of all graphs in our data set. Furthermore, we can see that this conjecture for the linear and quadratic coefficients of $g(t)$ does not extend to further coefficients. For example,
\begin{align*}
    g_{G_1}(t) &= 8t^6+45t^5+102t^4+116t^3+66t^2+16t \\
    g_{G_2}(t) &= 7t^6+42t^5+99t^4+115t^3+66t^2+16t
\end{align*}
differ in all further coefficients, where the graphs are from \cref{fig:3edge-g}. Cutting out the triangle in $G_1$ and $G_2$ leads, in both cases, to the same pair of factor graphs: $A$ and $B=K_4$. So although this example is compatible with the conjecture as stated, since
\begin{align*}
    \frac{g_{A}(t)g_{K_4}(t)}{t} 
    &=\frac{(4t^5+19t^4+33t^3+25t^2+8t)(t^3+2t^2+2t)}{t} \\
    &= 4t^7+27t^6+79t^5+129t^4+124t^3+66t^2+16t,
\end{align*}
it also demonstrates that knowing the pair of factor graphs alone is not sufficient to determine $g(t)$ to higher orders than $t^2$.

The linear term of \Cref{con:3edge} amounts to the identity $\beta(G)=\beta(A)\beta(B)$ for Crapo's $\beta$ invariant of a graph with a 3-edge cut. This identity follows from \cite[Theorem~1]{SekineZhang:DecFlow}, which establishes the factorization
\begin{equation}\label{eq:flow-3cut}
    \Flow{G}(q)=\frac{\Flow{A}(q)\Flow{B}(q)}{(q-1)(q-2)}
\end{equation}
of the flow polynomials. The derivative of this identity at $q=1$ implies the factorization of $\beta$, since $\Flow{G}(q)=(-1)^{\loops(G)}\Tutte{G}(0,1-q)=(-1)^{\loops(G)}(1-q)\beta(G)+\asyO{(1-q)^2}$.

Turning to the twists of 4-edge cuts, again our full data set confirms \Cref{con:edgetwist}, and examples of such twist pairs in the data show that the identity does not extend to higher coefficients. For the linear coefficient, we can prove the conjecture, namely that Crapo's invariant $\beta(G_1)=\beta(G_2)$ agrees for two graphs that differ by a 4-edge twist. As before, this follows because the entire flow polynomial is preserved:
\begin{lemma}\label{lem:flow-twist}
    If two graphs $G_1$ and $G_2$ are related by a 4-edge twist (as in \cref{fig:edge-twist}), then their flow polynomials are equal: $\Flow{G_1}(q)=\Flow{G_2}(q)$.
\end{lemma}
\begin{proof}
    Let $A\sqcup B=G_1\setminus C$ denote the two sides of the 4-edge cut $C=\{e_1,\ldots,e_4\}$. Let $a_i$ and $b_i$ denote the end points of these edges $e_i=\{a_i,b_i\}$, with $a_i$ in $A$ and $b_i$ in $B$. Construct graphs $A_i$ from $A$ (and analogously $B_i$ from $B$) for $i=1,\ldots,4$ as follows:
    \begin{align*}
        A_1&=A\cup\{\{a_1,v\},\{a_2,v\},\{a_3,v\},\{a_4,v\}\}, \\
        A_2&=A\cup\{\{a_1,a_2\},\{a_3,a_4\}\},\\
        A_3&=A\cup\{\{a_1,a_3\},\{a_2,a_4\}\},\\
        A_4&=A\cup\{\{a_1,a_4\},\{a_2,a_3\}\}.
    \end{align*}
    Note that $A_1\cong G_1/B$ is obtained by connecting a new vertex $v$ to all $a_i$, whereas $A_2,A_3,A_4$ arise from $A$ after attaching two edges. Then the flow polynomials of the pieces $A_i$ and $B_j$ already determine the flow polynomial of $G_1$, namely
    \begin{equation}\label{eq:flow-4cut}\tag{$\ast$}
        \Flow{G_1}(q) = \sum_{i,j=1}^4 \left(M_4(q)^{-1}\right)_{ij} \Flow{A_i}(q) \Flow{B_j}(q)
    \end{equation}
    for some explicit $4\times4$ matrix $M_4(q)$; see \cite[Theorem~1]{Kochol:DecFlow}. The twisted graph $G_2$ can be represented with the exact same $C$ and $B$, by merely trading $A$ for $\sigma(A)$, where $\sigma$ only swaps the labels of the $a_i$ vertices, according to $a_1\leftrightarrow a_2$ and $a_3\leftrightarrow a_4$. Since all four parts $A_i=\sigma(A_i)=\sigma(A)_i$ are invariant under this relabelling, we conclude that the right-hand sides of the decompositions \eqref{eq:flow-4cut} for $\Flow{G_1}(q)$ and $\Flow{G_2}(q)$ are equal.
\end{proof}
The above identity gives a construction of pairs of nonisomorphic, 4-connected graphs with the same flow polynomial. This gives a partial answer to \cite[Problem~2]{SekineZhang:DecFlow}. However, not every coincidence of flow polynomials between 4-connected, non-isomorphic graphs in our data set can be explained by (sequences of) 4-edge twists.

\subsection{Relation to the Tutte polynomial}
\begin{figure}
    \centering
    $A\colon \quad \verb.H?^Bfm~.=\Graph{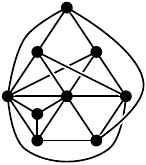}$ \qquad\qquad $B\colon\quad\verb.H?dfb~|.=\Graph{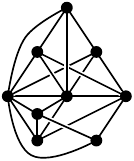}$
    \caption{Two graphs with the same Tutte polynomial but different Speyer polynomials; as drawings and as strings in \nauty's \texttt{graph6} format \cite{McKayPiperno:II}.}%
    \label{fig:Tutte-g}%
\end{figure}
Our data set contains several pairs of graphs with the same Tutte-, but different $g$-polynomials. For example, the graphs in \cref{fig:Tutte-g} have equal Tutte polynomial, but
\begin{align*}
    g_A(t) &=78t^8+623t^7+2096t^6+3854t^5+4176t^4+2658t^3+916t^2+132t \quad\text{and}\\
    g_B(t) &=79t^8+637t^7+2143t^6+3919t^5+4216t^4+2667t^3+916t^2+132t
\end{align*}
differ. Therefore, the Tutte polynomial does not determine $g$---confirming a conjecture by Speyer \cite[p.~31]{Speyer:MatroidKtheory}.
Note that not even the leading coefficients $\omega(A)=78$ and $\omega(B)=79$ agree; indeed the only coefficient of $g$ that is determined by the Tutte polynomial is the linear coefficient $g'_M(0)=\beta(M)$ from \eqref{eq:gm10}.\footnote{The equality of the quadratic coefficient $916$ in our example is a coincidence---there are other pairs of graphs with equal Tutte polynomial for which also the quadratic coefficients of $g(t)$ differ.}
Our invariants $\FP_2(A)=1$ and $\FP_2(B)=-1$ also disagree in this example, so $\FP_2$ is not determined by the Tutte polynomial either.

Furthermore, a rank calculation with our data shows that the vector space of linear relations between coefficients of $g$ and coefficients of the Tutte polynomial has dimension three. Therefore, for biconnected graphs, there are no more linear constraints than those we know from \eqref{eq:gm10} and \Cref{thm:gprime=cc}, namely:
\begin{equation*}
    \FP_0(G)=1,\qquad
    \FP_1(G)=0,\qquad \text{and}\qquad
    \sum_i \FP_i(G)=\beta(G).
\end{equation*}
The Tutte- and $g$-invariants of biconnected graphs (let alone arbitrary connected matroids) thus appear to be essentially unrelated---apart from the fact that both track Crapo's $\beta$.

Surprisingly, we nevertheless found a new relation---by restricting our attention to a subclass of graphs. Namely, we considered cubic (3-regular) graphs. Minimal counterexamples to problems like Tutte's 5-flow conjecture or the Berge-Fulkerson conjecture are known to be cubic (if existent), so it could be useful to identify relations of graph invariants even if they only apply to cubic graphs.

Our data set includes all 587 biconnected cubic graphs with 14 vertices (21 edges) or less and all \numprint{85642} cyclically 4-connected cubic graphs with 20 vertices (30 edges) or less. This covers \numprint{86118} graphs in total.\footnote{Both sets overlap in the 111 cyclically 4-connected cubic graphs with 14 vertices or less.} We used {\Maple}'s \texttt{GraphTheory} package to compute their flow polynomials and thereby the Tutte coefficients $t_{01}$ and $t_{02}$. In each case, we confirmed \Cref{con:g-Tutte}.

Since a cubic graph has $\abs{V(G)}=2(\loops(G)-1)$ edges, and the flow polynomial is $\Flow{G}(1-y)=(-1)^{\loops(G)} (y t_{01}(G)+y^2 t_{02}(G))+\asyO{y^3}$, we can state \Cref{con:g-Tutte} also as
\begin{equation}\label{eq:g-flow}
    g_G(t)=R_G(t)+\asyO{t^3}
    \qquad\text{if we set}\qquad
    R_G(t) = \frac{\Flow{G}(1+2t)}{2\cdot (2t-1)^{\loops(G)-1}}.
\end{equation}
For the right-hand side, \Cref{lem:flow-twist} shows that $R_{G_1}(t)=R_{G_2}(t)$ for 4-edge twists, whereas \eqref{eq:flow-3cut} shows that $R_G(t)=R_A(t)R_B(t)/t$ for a 3-edge cut. \Cref{con:g-Tutte} is therefore compatible with \Cref{con:3edge} and \Cref{con:edgetwist}.

As a final piece of positive evidence, we use the flow polynomial
\begin{equation*}
    \Flow{K_2\times C_n}(q)=(q-2)^n+(q-1)(q-3)^n+(q^2-3q+1)(-1)^n
\end{equation*}
of the prism graphs, computed in \cite[\S4]{BiggsDamerellSands:RecFam}, to determine that the right-hand side is
\begin{equation*}
    R_{K_2\times C_n}(t) = t\Big(2^n-n-1\Big)+t^2\Big(n 2^n-n^2-3n+2\Big)+\asyO{t^3}.
\end{equation*}
Substituting \Cref{con:prism-moebius} into the left-hand side of \eqref{eq:g-flow}, we find that both sides agree up to $\asyO{t^3}$. In other words, \Cref{con:prism-moebius} and \Cref{con:g-Tutte} are compatible.

\subsection{Factorizations}
It was shown in \cite{Speyer:MatroidKtheory} that any decomposition of a matroid into a direct sum or 2-sum forces a factorization of the polynomial $g_M(t)$:
\begin{equation*}
    g_{M_1\oplus M_2}(t)=g_{M_1}(t)\cdot g_{M_2}(t), \qquad
    g_{M_1\oplus_2 M_2}(t)=\frac{g_{M_1}(t)\cdot g_{M_2}(t)}{t}.
\end{equation*}
These factorizations occur for disconnected graphs or for graphs that have a 1- or 2-vertex cut. For a 3-connected graph, neither decomposition applies, and in fact the presence of a 3-vertex cut does not necessarily imply any factorization (see \Cref{rem:K5plusK5}).

However, in contrast to the behaviour of the Tutte polynomial \cite{MerinoMierNoy:IrredTutte}, there are examples of 3-connected graphs for which $g_G(t)$ \emph{does} factorize non-trivially. For example, the 4-partite graph $K_{1,1,1,6}$ from \cref{fig:K111n} has
\begin{equation*}
    g_{K_{1,1,1,6}}(t)=t(3t+2)(2t+1)(t+2)^5
\end{equation*}
and more generally, $g_{K_{1,1,1,n}}(t)$ is divisible by $(t+2)^{n-1}$ according to \Cref{con:g(K3n)}.
\begin{conjecture}
    For all integers $r,s\geq 1$ the polynomial $g_{K_{2+r}\vee\overline{K}_{s+r}}(t)$ has a zero of order $s$ at $t=-1-1/r$. 
\end{conjecture}
This pattern emerged after computing many examples; e.g.\ 
\begin{align*}
g_{K_4\vee\overline{K}_7}(t) &= t(2t+3)^7(238t^4+605t^3+527t^2+177t+18), \\
g_{K_5\vee\overline{K}_6}(t) &= t(3t+4)^3(4874t^6+21834t^5+39456t^4+36451t^3+17892t^2+4320t+384).
\end{align*}
The set of small graphs (\cref{tab:dataset}) contains \numprint{2858560} graphs that are 3-connected. In most cases, $g_G(t)/t$ is irreducible; factorizations occur for only \numprint{240591} of the 3-connected graphs. The vast majority (\numprint{232906}) of factorizations is due to a zero at $t=-2$. Further rational zeros occur at $t\in\{-3/2,-4/3,-2/3,-1/2\}$. There remain \numprint{6238} graphs without rational zeroes but factorizations into higher degree polynomials, e.g.\ for joins of the path graph $P_k$ on $k$ vertices with $K_2=P_2$ we find\footnote{In the \texttt{graph6} format of the data set files, these graphs are $P_5\vee K_2=$\texttt{FDZ\~{}w} and $P_7\vee K_2=\texttt{H@IQV\~{}\~{}}$.}
\begin{align*}
    g_{P_5\vee K_2}(t) &= t(2t^2+5t+4)(4t^3+12t^2+11t+4), \\
    g_{P_7\vee K_2}(t) &= t(4t^3+14t^2+17t+8)(8t^4+32t^3+48t^2+31t+8).
\end{align*}

It would be interesting to find any graph-theoretic interpretations/explanations for any of the rational zeroes, or for higher degree factorizations. There do not seem to be any known such criteria in the literature.

\subsection{Connection to Feynman integrals}
\label{sec:Feynman}

In perturbative quantum field theory, scattering amplitudes of particles are computed as a sum over graphs (Feynman diagrams). The contribution of each graph is determined by a (Feynman) integral. In the simplest cases, these integrals evaluate to real numbers\footnote{In general, the integrals are functions that depend on masses and momenta of the particles.}, called \emph{Feynman periods} \cite{BlochEsnaultKreimer:MotivesGraphPolynomials,Schnetz:Census}.

These Feynman periods are hard to compute and they are only defined for sufficiently connected graphs (otherwise the integral diverges). For example, let $G$ be a cyclically 6-connected 4-regular graph, that is, a 4-regular graph which has no 4-edge cuts other than those that isolate one of the vertices. Then for any choice of vertex $v$, the Feynman period of the graph $G\setminus\{v\}$ is well-defined. Various formulas for these Feynman periods, which we denote $\Period(G\setminus\{v\})\in\R$, can be found in \cite{Schnetz:Census}. For example,
\begin{equation*}
    \Period(K_5\setminus\{v\})=\Period\left(\Graph[0.25]{w3A}\right) = 6\zeta(3),\quad
    \Period(C^6_{1,2}\setminus\{v\})=\Period\left(\Graph[0.3]{ws4A}\right) = 20\zeta(5).
\end{equation*}
where $\zeta(s)=\sum_{n=1}^{\infty} 1/n^s$ is the Riemann zeta function.

Crucially for the link to $g_G(t)$, the Feynman period integrals are independent of the choice of the vertex $v$. Therefore, we may view these integrals as defining a function
\begin{equation*}
    G\mapsto\PeriodC(G)\defas\Period(G\setminus\{v\})
\end{equation*}
on the set of cyclically 6-connected 4-regular graphs. This construction generalizes to cyclically 4-connected 3-regular graphs.\footnote{Feynman periods are defined for more general graphs, but we limit our attention to the cubic and quartic cases, which are the most interesting from the point of view of physics.}

The Feynman period satisfies identities for twists and $3$-sums:
\begin{equation*}
    \PeriodC(G)=\PeriodC(G')\qquad\text{and}\qquad
    \PeriodC(G_1 \oplus_3 G_2) = \PeriodC(G_1) \cdot \PeriodC(G_2)
\end{equation*}
in the notation of \Cref{con:3sum} and \Cref{con:twist}; indeed the twist was discovered in this context \cite{Schnetz:Census}.
The invariant $\FP_2(G)=-g_G''(-1)/2$ obtained from Speyer's polynomial $g_G(t)$ thus behaves similar to the Feynman period. In fact, referring to the tables of periods for 4-regular graphs \cite{PanzerSchnetz:Phi4Coaction} and 3-regular graphs \cite{BorinskySchnetz:RecursivePhi3}, we confirmed the following conjecture for all graphs whose periods are known.
\begin{conjecture}\label{con:Period-N2}
    If two cyclically 6-connected 4-regular graphs (or cyclically 4-connected 3-regular graphs) $G_1,G_2$ have equal period $\PeriodC(G_1)=\PeriodC(G_2)$, then also $\FP_2(G_1)=\FP_2(G_2)$.
\end{conjecture}
This claim goes beyond \Cref{con:3sum} and \Cref{con:twist}, because there exist pairs of graphs that are not related by twists or products but which nevertheless share the same period. For example, it is known that if $G\setminus \{v\}$ is planar, then its Feynman period $\Period(G\setminus \{v\})=\Period( (G\setminus \{v\})^\star)$ equals that of the planar dual graph $(G\setminus \{v\})^\star$. Hence \Cref{con:Period-N2} implies in particular:
\begin{conjecture}
    If two cyclically 6-connected 4-regular graphs $G_1,G_2$ have vertices $v_i\in V(G_i)$ so that $G_i\setminus \{v_i\}$ are planar and dual to each other, then $\FP_2(G_1)=\FP_2(G_2)$.
\end{conjecture}
We stress that this identity is unrelated to \Cref{con:planar=1}, because $G_i$ need not be planar. An example of such a duality ``after deleting a vertex'' is shown in \cite[Figure~5.4]{PanzerYeats:c2Martin}.

Only a few other invariants of graphs are currently known that satisfy the identities of Feynman periods like $\FP_2$ in \Cref{con:Period-N2}: the $c_2$-invariant \cite{Schnetz:Fq}, the graph permanent \cite{CrumpDeVosYeats:Permanent,Crump:ExtendedPermanent}, the Hepp bound \cite{Panzer:HeppBound}, and the Martin invariant \cite{PanzerYeats:c2Martin}. It would be very interesting to understand how these invariants are related to $\FP_2(G)$. In particular, we expect:
\begin{conjecture}\label{con:N2-from-Martin}
    For a cyclically 6-connected 4-regular (or cyclically 4-connected 3-regular) graph $G$, the Martin sequence $\Martin(G^\bullet)$ determines the value of $\FP_2(G)$.
\end{conjecture}
The Martin sequence is an integer sequence whose first entry, called \emph{Martin invariant}, counts partitions of $G\setminus v$ into spanning trees \cite{PanzerYeats:c2Martin}. This number alone does \emph{not} determine $\FP_2(G)$, i.e.\ we found pairs of graphs with different $\FP_2(G)$ but equal Martin invariants.

The distribution of the values of $\FP_2(G)$ for these 3- and 4-regular graphs is given in \Cref{tab:N2-phi3} and \Cref{tab:N2-phi4}.
At each even vertex number $n\leq 12$, $\FP_2(G)$ is maximized among the cyclically 6-connected 4-regular graphs by a unique graph: the circulant from \cref{fig:xladders}, whose $\FP_2$ grows exponentially (\Cref{con:N2(xladder)}).
Among the cyclically 4-connected 3-regular graphs, $\FP_2(G)$ has a unique maximizer at $n=16$ and $n=20$, given by the crossed prism graphs $W_{4,2}$ and $W_{5,2}$ illustrated in \cref{fig:cubic-maximizer}. The family of crossed prism graphs was noted in \cite{GuoMohar:LargeBip1} for not having any eigenvalues (of the adjacency matrix) in the interval $(-1,1)$.
We expect that $\FP_2(W_{n,2})=2^{n-1}-n$ for all $n\geq 3$ (confirmed by calculation for $n\leq 7$), the same value as for the graph $C^{2n}_{1,n-1}$ when $n\geq 5$.
\begin{figure}
    \centering
    $\Graph[0.6]{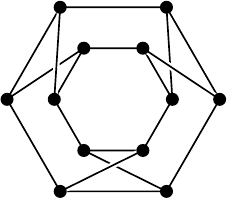} \qquad \Graph[0.6]{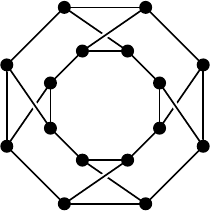} \qquad \Graph[0.6]{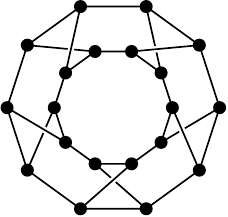} \qquad \Graph[0.6]{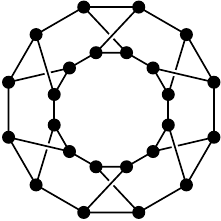}$
    \caption{The crossed prism graphs $W_{n,2}$ from \cite{GuoMohar:LargeBip1} for $n=3,4,5,6$. In the database \cite{HoG}, these are indexed $W_{3,2}=\HoG{1086}$, $W_{4,2}=\HoG{27419}$, $W_{5,2}=\HoG{36306}$, $W_{6,2}=\HoG{36323}$.}%
    \label{fig:cubic-maximizer}%
\end{figure}

\appendix

\section{Data set}
All graphs are encoded as strings in \nauty's compact \texttt{graph6} format \cite{McKayPiperno:II}, which is widely supported by mathematical software. The graphs in our data set are simple, biconnected, and have minimum degree three or higher.
The data set is openly available at
\begin{center}
    \url{\dataurl}
\end{center}

For small numbers of vertices ($v\leq 9$) or edges ($e\leq 21$), we computed complete lists of such graphs with {\nauty}'s program \texttt{geng}. The numbers of these graphs are summarized in \cref{tab:dataset}. 
The corresponding results are provided in the file \texttt{gsmall.csv}, one line per graph, as a comma separated list. The format is
\begin{equation*}
    G,v,e,\FP_2(G),g_G(t).
\end{equation*}
For example, the first line of the file,
\begin{verbatim}
    C~,4,6,1,t^3+2*t^2+2*t
\end{verbatim}
encodes a graph with 4 vertices, 6 edges, $\FP_2(G)=1$, and $g_G(t)=t^3+2t^2+2t$. There is only one biconnected graph with minimum degree three, $v=4$ and $e=6$: the complete graph $K_4=\Graph[0.12]{w3A}$, and indeed, \verb.C~. is the \texttt{graph6} encoding of $K_4$. The Speyer polynomial of this graph was computed in \cref{ex:g-path-rec}.
\begin{remark}
    All graphs were generated by \texttt{geng} in their \emph{canonical} labelling (option \verb.-l.). Therefore, any given graph can be looked up in the file by first computing its canonical label with {\nauty}, and then searching our data file for this string.
\end{remark}
\begin{table}
    \centering
    \begin{tabular}{rrrrrrrrrrrr}
    \toprule
          & $v=4$ & 5 & 6 & 7 & 8 & 9 & 10 & 11 & 12 & 13 & 14 \\
    \midrule
  $e=6$ & 1 & & & & & & & & & & \\
  7 & & & & & & & & & & & \\
  8 & & 1 & & & & & & & & & \\
  9 & & 1 & 2 & & & & & & & & \\
  10 & & 1 & 4 & & & & & & & & \\
  11 & & & 5 & 4 & & & & & & & \\
  12 & & & 4 & 17 & 5 & & & & & & \\
  13 & & & 2 & 30 & 33 & & & & & & \\
  14 & & & 1 & 34 & 133 & 25 & & & & & \\
  15 & & & 1 & 29 & 307 & 277 & 18 & & & & \\
  16 & & & & 17 & 464 & 1352 & 340 & & & & \\
  17 & & & & 9 & 505 & 3953 & 3387 & 195 & & & \\
  18 & & & & 5 & 438 & 7939 & 18439 & 4788 & 81 & & \\
  19 & & & & 2 & 310 & 11897 & 64715 & 49738 & 4111 & & \\
  20 & & & & 1 & 188 & 14131 & 163538 & 300550 & 83588 & 1853 & \\
  21 & & & & 1 & 103 & 13827 & 318940 & 1235667 & 868664 & 87666 & 480 \\
  22 & & & & & 52 & 11465 & & & & & \\
  23 & & & & & 23 & 8235 & & & & & \\
  24 & & & & & 11 & 5226 & & & & & \\
  25 & & & & & 5 & 2966 & & & & & \\
  26 & & & & & 2 & 1537 & & & & & \\
  27 & & & & & 1 & 737 & & & & & \\
  28 & & & & & 1 & 333 & & & & & \\
  29 & & & & & & 144 & & & & & \\
  30 & & & & & & 62 & & & & & \\
  31 & & & & & & 25 & & & & & \\
  32 & & & & & & 11 & & & & & \\
  33 & & & & & & 5 & & & & & \\
  34 & & & & & & 2 & & & & & \\
  35 & & & & & & 1 & & & & & \\
  36 & & & & & & 1 & & & & & \\
  \bottomrule
    \end{tabular}
    \caption{The numbers of small graphs with $v$ vertices and $e$ edges in our data set \cite{Data:gSpeyer}.}%
    \label{tab:dataset}%
\end{table}
\begin{table}
    \centering
    \begin{tabular}{rrrrrrrrrr}
    \toprule
    $\FP_2$ 
    & $n=4$ & 6 & 8 & 10 & 12 & 14 & 16 & 18 & 20 \\
    \midrule
$-1$ & & & & & & & & & 15 \\
$0$ & & 1 & 1 & 3 & 13 & 56 & 336 & 2478 & 22410 \\
$1$ & 1 & & 1 & 2 & 5 & 28 & 258 & 3445 & 52320 \\
$2$ & & & & & & & 8 & 146 & 3580 \\
$3$ & & & & & & & 4 & 28 & 450 \\
$4$ & & & & & & & 1 & 3 & 37 \\
$5$ & & & & & & & & & 4 \\
$6$ & & & & & & & & & 5 \\
$8$ & & & & & & & & & 2 \\
$11$ & & & & & & & & & 1 \\
\midrule
 total & 1 & 1 & 2 & 5 & 18 & 84 & 607 & 6100 & 78824 \\
\bottomrule
    \end{tabular}
    \caption{The number of cyclically 4-connected 3-regular graphs with $n$ vertices that take a fixed value of $\FP_2(G)$.}
    \label{tab:N2-phi3}%
\end{table}
\begin{table}
    \centering
    \begin{tabular}{rrrrrrrrrr}
    \toprule
    $\FP_2$ 
    & $n=5$ & 6 & 7 & 8 & 9 & 10 & 11 & 12 & 13 \\
    \midrule
$-4$ & & & & & & & & & 2 \\
$-3$ & & & & & & & & & 1 \\
$-2$ & & & & & & & & & 11 \\
$-1$ & & & & & & & & & 11 \\
$0$ & 1 & & 1 & 1 & 9 & 20 & 99 & 350 & 2308 \\
$1$ & & 1 & & 2 & 2 & 17 & 81 & 673 & 5615 \\
$2$ & & & & & & & 5 & 77 & 517 \\
$3$ & & & & & & 1 & 2 & 56 & 141 \\
$4$ & & & & & & 2 & 3 & 13 & 56 \\
$5$ & & & & & & & & 3 & 15 \\
$6$ & & & & 1 & & & & 1 & \\
$7$ & & & & & & & & 1 & \\
$8$ & & & & & & & & 4 & 7 \\
$11$ & & & & & & 1 & & 2 & 3 \\
$17$ & & & & & & & & 1 & \\
$26$ & & & & & & & & 1 & \\
\midrule
 total & 1  & 1 & 1 & 4 & 11 & 41 & 190 & 1182 & 8687  \\
\bottomrule
    \end{tabular}
    \caption{The number of 4-connected cyclically 6-connected 4-regular graphs with $n$ vertices for each value of $\FP_2(G)$.}
    \label{tab:N2-phi4}%
\end{table}

In addition to the complete lists of small graphs as per \cref{tab:dataset}, we also computed results for larger graphs in special families: complete graphs $K_n$, complete bipartite graphs $K_{n,m}$, and 4-regular circulant graphs $C^n_{a,b}$. These results are in the file \texttt{gspecial.csv}. In addition to \texttt{graph6} strings, these graphs are also encoded by their actual names for convenience. Note that there are many isomorphisms between circulant graphs, in particular $C^n_{a,b}\cong C^n_{ua,ub}$ for any $u$ coprime to $n$ (e.g.\ $C^8_{1,2}\cong C^8_{2,3}$ via $u=3$).

Finally, in our data set we also include larger 3- and 4-regular graphs (files \Filename{g3.csv} and \Filename{g4.csv}) that play a particular role as Feynman graphs in particle physics, see \cref{sec:Feynman}. Concretely, these are the cyclically 4-connected 3-regular graphs (i.e.\ the only 3-edge cuts are those that isolate a single vertex) and the cyclically 6-connected 4-regular graphs (the only 4-edge cuts are those isolating vertex). In the 4-regular case, we furthermore restrict to the 4-connected graphs, i.e.\ we exclude graphs with 3-vertex cuts from the data set (graphs with 3-vertex cuts are called ``product graphs'' in \cite{Schnetz:Census}).
To ease comparison with other papers on Feynman period invariants, for each graph we also give the label $P_{L,i}$ under which it appears in the tables
\begin{itemize}
    \item of 3-regular graphs in the paper \cite{BorinskySchnetz:RecursivePhi3}, supplied there as a file \Filename{PeriodsPhi3.txt},
    \item of 4-regular graphs in the paper \cite{Schnetz:Census}, available from the file \Filename{Periods} of \cite{PanzerSchnetz:Phi4Coaction}.
\end{itemize}
The first index $L$ is called \emph{loop number} in physics. It refers not to the corank of the regular graph itself, but rather to the corank of any of the graphs obtained after deleting a single vertex. This amounts to the relations $(v,e)=(2L+2,3L+3)$ for 3-regular graphs and $(v,e)=(L+2,2L+4)$ for 4-regular graphs. For example, the first line
\begin{verbatim}
    D~{,5,10,"P[3,1]",0,5*t^4+15*t^3+15*t^2+6*t
\end{verbatim}
of \Filename{g4.txt} encodes the complete graph $K_5$, which is called $P_{3,1}$ in \cite{Schnetz:Census}.
The distribution of the values of $\FP_2(G)$ for these 3- and 4-regular graphs is given in \Cref{tab:N2-phi3} and \Cref{tab:N2-phi4}.

\section{Software and implementation details}\label{sec:implementation}

Our software is available from the repository
\begin{center}
    \url{\programurl}
\end{center}
The key routines provided by this package are:
\begin{itemize}
    \item \texttt{gSpeyer($E$)} computes $g_M(t)$ for the cycle matroid $M$ of a graph given as an edge list $E=[e_1,\ldots,e_n]$ with each $e_i=\{u_i,v_i\}$ given as a pair of vertex labels. This implements \Cref{alg:gRec}.
    \item \texttt{SchubertDecomposition($E$)} computes the (unique) formal linear combination $\sum_P \lambda_P \LPM{P}$ of lattice path matroids that represents the matroid polytope of the cycle matroid $M$ of the graph $E$ in the valuative group; see \cite[\S3.1]{FerroniFink:PolMat} where the vector $\lambda_\bullet$ is denoted $p_{[M]}$. This implements the decomposition underlying \Cref{thm:Speyer-from-cyclics} from \cite[\S5]{Ferroni:SchubertDelannoySpeyer}, using \Cref{lem:MF(chain)} to compute $\lambda_{C_\bullet}$. Then $\lambda_P$ is computed as the sum of all $\lambda_{C_\bullet}$ with $\SM{C_\bullet}\cong \LPM{P}$.
    \item \texttt{gSchubert($n,I$)} computes the $g$-polynomial of the lattice path matroid $\LPM{P(I)}$ on the ground set $\{1,\ldots,n\}$ where the path $P(S)\in\{\PU,\PR\}^n$ is encoded by the set of positions $S=\{i_1<\ldots<i_r\}\subset\{1,\ldots,n\}$ of the north steps (see \Cref{sec:lattice-paths}). This implements \Cref{alg:gLPM}.
    \item \texttt{CyclicFlats($E$)} computes the cyclic flats of the graph with edge list $E$. We implement a bottom-up traversal of the lattice of flats, with early pruning to filter out the cyclic flats. For efficient computation of the Hasse diagram (covering relations), we employ a linear time algorithm \cite{GGIK:2edge} for enumeration of 2-edge cuts.
\end{itemize}

\subsection{Examples}
To compute the $g$-polynomial of the complete graph on 8 vertices using \Cref{alg:gRec}:
\begin{MapleInput}
K8:=[seq(seq({i,j},j=i+1..8),i=1..7)]:
gSpeyer(K8);
\end{MapleInput}
\begin{MapleMath}
    3655t^7+17934t^6+35980t^5+37604t^4+21448t^3+6264t^2+720t
\end{MapleMath}
which reproduces the result from \Cref{tab:gKn}.

The decompositions into Schubert matroids (represented as lattice path matroids) from \Cref{ex:g(wheel)=LPM} of the 3- and 4-spoke wheel graphs can be reproduced as
\begin{MapleInput}
W3:=[{1,2},{2,3},{1,3},{1,4},{2,4},{3,4}]:
SchubertDecomposition(W3);
\end{MapleInput}
\begin{MapleMath}
    -3 \Schubert(6, \{1, 2, 3\}) + 4 \Schubert(6, \{1, 2, 4\})
\end{MapleMath}
\begin{MapleInput}
W4:=[{1,2},{2,3},{3,4},{1,4},{1,5},{2,5},{3,5},{4,5}]:
SchubertDecomposition(W4);
\end{MapleInput}
\begin{MapleMath}
    \Schubert(8, \{1, 2, 3, 5\}) - 4 \Schubert(8, \{1, 2, 3, 6\}) - 4 \Schubert(8, \{1, 2, 4, 5\}) \\
    + 8 \Schubert(8, \{1, 2, 4, 6\})
\end{MapleMath}
Schubert matroids are encoded by the set of north steps of corresponding lattice paths, e.g.\ 
\begin{align*}
    \Schubert(6, \{1, 2, 4\})&=\LPM{\Graph[0.3]{UURURR}}
    \quad\text{and}\\
    \Schubert(8, \{1, 2, 3, 6\})&=\LPM{\Graph[0.3]{UUURRURR}}.
\end{align*}

The Speyer $g$-polynomials of a lattice path matroid can be computed with
\begin{MapleInput}
gSchubert(6,{1,2,4});
\end{MapleInput}
\begin{MapleMath}
t^3+5t^2+5t
\end{MapleMath}
which reproduces a calculation from \Cref{ex:g-path-rec}.

Combining the routines \verb|SchubertDecomposition| and \verb|gSchubert| gives another way to compute the $g$-polynomial of any graph with our code. In fact, this amounts to the original algorithm from \cite{Ferroni:SchubertDelannoySpeyer}. For example,
\begin{MapleInput}
W3:=[{1,2},{2,3},{1,3},{1,4},{2,4},{3,4}]:
eval(SchubertDecomposition(W3),Schubert=gSchubert);
\end{MapleInput}
\begin{MapleMath}
    t^3+2t^2+2t
\end{MapleMath}
reproduces the result for the complete graph $K_4$ (wheel with 3 spokes). However, especially for large graphs, it is much more efficient to compute the polynomial $g$ directly using \verb|gSpeyer| (i.e.\ \Cref{alg:gRec}), which avoids the intermediate step of explicitly computating the Schubert decomposition.

\subsection{Implementation details}
Our implementation is tailored towards graphs, but the code can be extended to other matroids by replacing/extending the function \texttt{CyclicFlats}. Repeated calculations of the same quantity are avoided by storing intermediate results in tables for subsequent lookup from memory. Thus storing the $g_A^{<k}(t)$ in memory for all cyclic flats $A\in\LCF_M$, the primary driver of the runtime of \Cref{alg:gRec} are:
\begin{enumerate}
    \item[(I)]\label{i:Z} computing the lattice $\LCF$, i.e.\ its elements and covering relations,
    \item[(II)]\label{i:mu} computing the M\"obius function $\mu_{\LCF}(A,B)$ for all pairs $A<B$,
    \item[(III)]\label{i:sum} computing for each $B$ and $k$ the sum over $A<B$ and $k'$.
\end{enumerate}
A direct implementation of \Cref{alg:gRec} would use $\asyO{\abs{\{(A,B)\in\LCF_M^2\colon A\leq B\}}\cdot \rk(M)^2}$ polynomial additions and multiplications for (III), but the sum over $k'$ does not depend on the rank of $A$. Hence we first sum $\mu_\LCF(A,B)\cdot\text{\gRec{$A,k'$}}$ over all $A$, grouped by the pairs $(k',\loops(A))$. Then for each $B$, there are at most $\rk(B)$ values of $k'$ and $\loops(B)$ values of $\loops(A)$ to consider. Due to this implementation, we reduced the complexity for (III) to $\asyO{\abs{\{(A,B)\in\LCF_M^2\colon A\leq B\}}\cdot \rk(M)+\abs{\LCF_M}\cdot\rk(M)\cdot\loops(M)}$.

Asymptotically, the most time consuming parts of our implementation are therefore steps (I) and (II). We implemented (II) by computing once, for each pair $A<B$, the defininig recursion $\mu(A,B)=-\sum_{A\leq C<B} \mu(A,C)$. Hence we can estimate the time for step (II) by the number
\begin{equation*}
    \abs{\{(A,C,B)\in\LCF_M^3\colon A<C<B\}}
\end{equation*}
of 2-simplices in the order complex of $\LCF$. Asymptotic improvements would thus require a faster method for M\"obius inversion.
However, the integer additions for $\mu$ are much quicker than the calculations with polynomials \gRec{$A$,$k'$} in step (III), such that in practice we often find that the time spent on steps (II) and (III) is comparable.

For very large graphs, also step (I) takes significant time, because our implementation effectively enumerates the lattice of all flats (and then filters out the cyclic ones), which can be exponentially larger than the lattice of only cyclic flats. In comparison, however, for the largest graphs we calculated, the runtime is dominated by step (II).

\bibliography{refs}

\end{document}